\documentclass[10pt]{amsart}
\usepackage{enumitem}   
\usepackage{amsmath}
\usepackage{amssymb}
\usepackage{amsthm}
\usepackage{comment}
\usepackage{hyperref}
\usepackage{mathrsfs} 
\usepackage{breqn}
\usepackage{xcolor}
\usepackage{graphicx}
\usepackage{geometry}
\let\vec\mathbf

\newtheorem{thm}{Theorem}[section]

\newtheorem{prop}[thm]{Proposition}
\newtheorem{lem}[thm]{Lemma}

\newtheorem{dfn}[thm]{Definition}
\newtheorem{rmk}[thm]{Remark}

\newcommand{\UU}{\mathcal{U}}
\newcommand{\CC}{\mathcal{C}}
\newcommand{\BB}{\mathcal{B}}
\newcommand{\FF}{\mathscr{F}}

\newcommand{\LL}{\mathcal{L}}
\newcommand{\PP}{\mathcal{P}}
\newcommand{\Ss}{\mathcal{S}}

\newcommand{\rr}{\vec{r}_t}

\newcommand{\T}{\mathbb{T}^3}
\newcommand{\R}{\mathbb{R}}

\newcommand{\E}{\mathbb{E}}
\newcommand{\p}{\mathbb{P}}

\newcommand{\N}{\mathbb{N}}
\newcommand{\B}{\mathscr{B}}
\usepackage{authblk}

\numberwithin{equation}{section}

\newcommand{\dd}{\mathrm{d}}

\newcommand{\dt}{\,\mathrm{d}t}

\newcommand{\bfr}{\mathbf r}

\newcommand{\CV}{\mathcal{V}_{t,x}}
\usepackage[foot]{amsaddr}

\title{Dissipative Solutions and Markov selection to the complete stochastic Euler system}
\author{Thamsanqa Castern Moyo}
\address{Maxwell Institute for Mathematical Sciences and Department of Mathematics, Heriot-Watt University, Edinburgh, United Kingdom.}
\email{tcm4@hw.ac.uk}

\begin{document}
\maketitle
\begin{abstract}
    We introduce the concept of \textit{stochastic measure-valued solutions} to the {complete} Euler system describing the motion of a compressible inviscid fluid subject to stochastic forcing, where the nonlinear terms are described by defect measures. These solutions are weak in the probabilistic sense (probability space is not a given `priori', but part of the solution) and analytical sense (derivatives only exists in the sense distributions). In particular, we show that: existence, weak-strong principle;
    a weak measure-valued solution coincides with a strong solution provided the later exists, all hold true provided they satisfy some form of energy balance. Finally, we show the existence of Markov selection to the associated martingale problem.\newline\\
    
    
\end{abstract}
\section{Introduction}
In 1985, Diperna\cite{DiP} proposed the concept of measure-valued solutions to nonlinear systems of partial differential equations (PDEs) admitting the uncontrollable oscillations (in the context of conservation laws). In the framework of fluid dynamics, Diperna and  Majda \cite{DiMa}, introduced the concept of measure-valued solutions to the incompressible Euler system. Later on, the concept of measure-valued solutions was extended to compressible fluid dynamics by Neustupa \cite{Neu}, Kr\"oner and Zajackowski \cite{KrZa}, and revisited recently by Breit et al \cite{HoFeBr}, Feireisl et. al. in \cite{FEJB,FeGwS} and references therein, where they developed the concept of \textit{dissipative measure-valued solutions}. In this paper, we consider the \textit {Complete Euler System} describing the motion of a temperature dependent compressible inviscid fluid flow driven by stochastic forcing. The fluid model is described by means of three basic state variables: the mass density $\varrho=\varrho(t,x)$, the velocity field $\vec u = \vec (t,x)$, and the (absolute) temperature $\vartheta=\vartheta(t,x)$, where $t$ is the time, $x$ is the space variable (Eulerian coordinate system). The time evolution of the fluid flow is governed by a system of partial differential equations (mathematical formulations of the physical principles) given by 
\begin{eqnarray}\label{Euler}
\dd \varrho + \mathrm{div}_{x}(\varrho \vec u)\, \dd t &=&0 \quad \text{in}\, Q,\nonumber\\
\dd(\varrho \vec u)+ \mathrm{div}_x(\varrho \vec u \otimes \vec u)\, \dd t + \nabla_x p(\varrho,\vartheta)\,\dd t&=&\varrho \phi \,\dd W\quad \text{in}\, Q,\\
\dd \left(\frac{1}{2}\varrho|\vec u|^2+ \varrho e(\varrho,\vartheta)\right) + \mathrm{div}_x\left[\left(\frac{1}{2}\varrho|\vec u|^2+ \varrho e(\varrho,\vartheta)+p(\varrho,\vartheta)\right)\vec u\right]\dd t &=&\frac{1}{2}\|\sqrt{\varrho}\phi\|_{L_2}^2 \,\dd t+ \varrho\phi\cdot \vec u\dd W,\nonumber
\end{eqnarray}
describing: the balance of mass, momentum, total energy, respectively. Here, $p(\varrho,\vartheta)$ denotes pressure, the driving force is represented by a cylindrical Wiener process W, and $\phi$ is a Hilbert-Schmidt operator, see Section \ref{s:stoc} for details. For completeness, the system (\ref{Euler}) is supplemented by a set of constitutive relations characterising the physical principles of a compressible inviscid fluid. In particular, we assume that the pressure $p(\varrho,\vartheta)$ and the internal energy $e =e(\varrho,\vartheta)$ satisfy the caloric equation of state
\begin{equation}\label{caloric}
    p=(\gamma-1)\varrho e, 
\end{equation}
where $\gamma>1$ is the adiabatic constant. In addition, we suppose that the absolute temperature $\vartheta$ satisfies the Boyle-Mariotte thermal equation of state:
\begin{equation}\label{boyle}
    p =\varrho\vartheta \quad \mathrm{yielding}\quad e= c_v\vartheta, c_v =\frac{1}{\gamma-1}.
\end{equation}
Finally, we suppose that the pressure $p=p(\varrho, \vartheta)$, the specific internal energy $e =e(\varrho,\vartheta)$, and the specific entropy  $ s =  s(\varrho,\vartheta)$ are interrelated through Gibbs' relation

\begin{equation}\label{gibbs}
   \vartheta D s (\varrho,\vartheta) = D e(\varrho,\vartheta)+ p(\varrho,\vartheta)D\left(\frac{1}{\varrho}\right).
\end{equation}
If $p,e,s$ satisfy (\ref{gibbs}), in context of any \textit{smooth} solutions to (\ref{Euler}), the Second law of thermodynamics is enforced through the entropy balance equation

\begin{equation}\label{entrB}
  \dd (\varrho  s (\varrho,\vartheta))+\mathrm{div}_x(\varrho  s(\varrho,\vartheta)\vec u)\, \dd t =0, 
\end{equation}
where $s(\varrho,\vartheta)$ denotes the (specific) entropy and is of the form
\begin{equation}\label{entropy}
    s(\varrho,\vartheta)=\log(\vartheta^{c_v})-\log(\varrho).
\end{equation}
For weak solutions, the equality in (\ref{entrB}) no longer holds, the entropy balance is given as an inequality, for more details see \cite{SBGD}. To circumvent problems from physical boundaries, we impose periodic boundary conditions, the physical domain $\T$ can be identified with a flat torus
\[
\T =([0,1]|_{0,1})^3.
\]
Finally, the initial state of fluid emanates from random initial data
\begin{equation}\label{data}
    \varrho(0,\cdot)=\varrho_0, \quad \vartheta(0,\cdot) =\vartheta_0, \quad \vec {u}(0,\cdot)=\vec u_0.
\end{equation}
For physical relevant solutions, the problem is augmented by the total energy balance
\begin{equation}\label{eq:tenergy}
    \dd \int_{\T}\left [\frac{1}{2}\varrho|\vec u|^2 +\varrho e\right]\, \dd x= \int_{Q}\varrho\phi \cdot\vec u\,\dd W+ \int_{0}^{T}\frac{1}{2}\|\sqrt{\varrho}\phi\|_{L_2}^{2}\,\dd t.
\end{equation}
{The strong solutions of the system (\ref{Euler}) satisfy (\ref{eq:tenergy}), but in weak solutions it has to be added in the definition.}\newline

The deterministic counterpart of the Cauchy problem (\ref{Euler}) has been extensively studied, and it is well-known that its classical solutions exist only for a finite time after which singularities  may develop no matter how smooth or small the initial data are. Consequently, the concept
of weak (distributional) solutions is sought to study global-in-time
behavior of the system (\ref{Euler}).  Furthermore, the weak solutions may not be uniquely determined by their initial data. Hence, an admissibility criteria condition must be imposed to select physically relevant solutions. In addition, more recently, the results of DeLellis, Sz\'ekelyhidi and their collaborators \cite{ChDeO,chi,DeSz} show the existence and non-uniqueness of weak solutions to the isentropic Euler system via the method of convex integration. In particular, non-uniqueness was established for weak solutions satisfying the standard entropy admissibilty criteria, and to be precise, the deterministic compressible Euler system is ill-posed,
see Buckmaster et al \cite{Buc} and references therein  for more details. The existence of weak solutions to Euler system were further extended in \cite{ChFrOk} to incorporate a compressible heat conducting gas. Finally, results established in \cite{HoFeBr} show that it is possible to select a system of solutions satisfying the classical semiflow property for the complete Euler system.\newline 

The study of stochastically perturbed equations of motion is motivated in two folds: 
(i) modelling perturbations (numerical, empirical, and physical uncertainties) and thermodynamics fluctuations present in fluid flows; in particular, turbulence, (ii) to circumvent the issue of  deterministically ill-posed problems, researchers adopted the use of stochastic perturbation with hope it will provide a regularising effect  to the underlying systems. And indeed, recently, the results by Flandoli and Luo \cite{FlLu} showed that a noise of transport type improves the vorticity blow-up control in the Navier-Stokes. In the stochastic context, existence of global-in-time weak solutions for (\ref{Euler}) were shown in \cite{ChFeFl} using convex integration, and they identified  a large class of initial data for which the \textit{complete} Euler system is ill-posed; that is, there exist infinitely many global in time  solutions.\newline

Our goal in the present paper is to show the existence of \textit{dissipative} measure-valued solutions to a stochastically driven Euler system (\ref{Euler}), properties of solutions, that is, weak-strong principle and Markov selection. In particular, measure-valued solutions satisfy the admissibility criterion; they conserve the total energy and satisfy an appropriate form version of the entropy (entropy admissibility criterion), and they exist global-in-time for any finite initial data. The study of measure-valued solutions is motivated by the following arguments: Inspired by results in Brenier et al \cite{Bren} in the incompressible Euler system, we see measure-valued solutions as  possibly the largest class in which the family of smooth (classical) solutions is stable. Specifically, the weak (measure-valued)-strong uniqueness principle holds, and solutions emanating from numerical schemes can be shown to converge to a measure-valued solutions while the convergence to weak solutions is either not known or computational expensive, see Fjordholm et al \cite{Fjh} and references therein for more details. The concept of measure-valued solutions to fluid model systems driven by stochastic forcing is fairly a new subject area of research. To the best of our knowledge, the study of  stochastic measure-valued solutions to the complete Euler system governing the motion of an inviscid, temperature dependent, and compressible fluid subject to stochastic forcing is still an open question. Hence, this is a first attempt to characterise the concept of measure-valued solutions to the full stochastic Euler system. However, it is worth mentioning that, the existence of measure-valued solutions for stochastic incompressible Euler has been established in \cite{BrMo,MZz,kim}.\newline

In this present work, we prove the existence of measure-valued martingale solutions to the \textit{complete} Euler system (\ref{Euler}) following the strategy in \cite{ BrMo,MaUjUt,kim}. These solutions are weak, in the
analytical sense (derivatives only exists in the sense of distributions) and in the probabilistic sense
(the probability space is not a given priori, but an integral part of the solution). We start with an approximate system with high order diffusion, see Section \ref{Bsec}, and show existence of solutions to the original problem in the limit via stochastic compactness arguments based on Jakubowski's variant of the Skorokhod representation theorem \cite{Jaku}. The latter is needed due to the complicated \textit{path space} which arises because of the presence of measures describing the oscillations and concentrations in the nonlinearities of the Euler system. Moreover, we deduce the  relative entropy inequality (see Section \ref{WK}); a tool that allows us to establish the weak(measure-valued)-strong uniqueness principle. Although we do not expect solutions to be unique, there is some hope to select solutions which are in a sense continuous with respect to the initial data. This is the Markov property; the memoryless of the stochastic process, the probability law of the future only depends on the current state of the process, but it is independent from the past, see the monograph by Stroock and Varadhan \cite{STVAR} for a thorough exposition. Our work shows the stochastic analog results of \cite{HoFeBr}, that is, the existence of Markov selection to the associated martingale problem following the presentations in \cite{DbEfMh,FlROm,GoRo}. At first sight, the overall proof outline is rather similar, however, we encountered several challenges in this paper. A major challenge originates in the use of defect measures. The defect measures are an equivalence class in time and not stochastic processes in the classical sense. Therefore, it is not clear as to how one applies the Markov selection. To circumvent this problem, we introduce auxiliary continuous stochastic variables $[\Ss,\vec R]$ such that 
\begin{equation*}
   \Ss = \int_{0}^{\cdot}S\,\dd s, \quad \vec R = \int_{0}^{\cdot}(\mathcal{R}_{\text{conv}},\mathcal{R}_{\text{press}},\mathcal{V}_{t,x})\,\dd s,
   \vspace{-0.265in}
\end{equation*}
 and this allows us to show Markov selection for $[\varrho, \vec m, \Ss, \vec R]$, see Section \ref{markov}. Here, $\Ss$ and $\vec R$ are the defect measures related to entropy balance and momentum equation, respectively. Note, such an approach is reminiscent of \cite{DbEfMh}, where a similar idea was used for the velocity field. It is important to note that, different to \cite{DbEfMh,FlROm,GoRo}, we obtain a strong Markov selection. This is due to the energy equality which is a feature of the system (\ref{Euler}) and is not known to hold for the problems studied in \cite{DbEfMh,FlROm,GoRo}. However, a first result on strong Markov selection has been obtained recently by Hofmanov\'a-Zhu-Zhu \cite{MZz}. In \cite{MZz} they study the incompressible stochastic Euler equations and obtain the energy equality by introducing a defect measure for the energy which is included in the Markov selection.\newline
 
 The rest of the paper is organised as follows: we start off with mathematical framework and main results in Section \ref{Asec}. In section \ref{Bapp}, we show existence of martingale solutions for the approximate system. In Section \ref{mvsproof}, we prove the existence of measure-valued martingale solutions using stochastic compactness arguments. Finally, Section \ref{WK} is dedicated to showing the weak (measure-valued)-strong uniqueness principle, while Section \ref{markov} is dedicated to the Markov selection.

\section{Mathematical framework and main results}\label{Asec}
In this section we present the the probability framework for Markov selection and stochastic framework. In particular we consider tools required for analysis in subsequent sections and state the main results of the paper. Throughout the paper, we assume that $C$ denotes various generic constant. 

\subsection{Probability framework}
Let $X$ be a topological space. We denote by $\B(X)$ the $\sigma$-algebra of Borel subsets of $X$. Let $\PP$ be a Borel measure on $X$, the symbol $\overline{\B(X)}$ denotes the $\sigma$-algebra of all Borel subsets of $X$ augmented by all zero measure sets. Let Prob$[X]$ denote the set of all Borel probability measures on a topological space $X$. Furthermore, let $([0,1],\overline{\B[0,1]}, \LL)$ denote the standard probability space, where $\LL$ is the Lebesgue measure.
\subsubsection{Trajectory/Path spaces}\label{path}
In regards to the notion of solutions in this paper, let $(X,d_{X})$ be a Polish space. For $t>0$ we introduce the path spaces
\[
\Omega_{X}^{[0,t]} =C([0,t];X),\qquad\Omega_{X}^{[t,\infty)} =C([t,\infty);X),\qquad\Omega_{X}^{[0,\infty)} =C([0,\infty);X),
\]
where the path spaces are Polish as long as $X$ is Polish, and we denote by $\B_t=\B(\Omega_{X}^{[0,t]})$ the Borel $\sigma$-algebra. Then, for $\omega \in \Omega_{X}^{[0,t]}$ we define a time shift operator 
\[
\Phi_{\tau}:\Omega_{X}^{[t,\infty)}\to \Omega_{X}^{[t+\tau,\infty)},\qquad \Phi_{\tau}[\omega]_s=\omega_{t-\tau}, s\geq t+\tau,
\]
where $\Phi_{\tau}$ is an isometry mapping from $\Phi_{\tau}:\Omega_{X}^{[t,\infty)}$ to $ \Omega_{X}^{[t+\tau,\infty)}$. For a Borel measure $\nu$ on $\Omega_{X}^{[t,\infty)}$, the time shift $\Phi_{-\tau}[\nu]$ is a Borel measure on the space $\Omega_{X}^{[t-\nu,\infty)}$ given by

\[
\Phi_{-\tau}[\nu](B)=\nu(\Phi_{\tau}(B)), \qquad B\in \B (\Omega_{X}^{[t-\tau,\infty)}).
\]
In the following parts, we recall important results of Stroock and Varadhan
\cite{STVAR}. Firstly, from Theorem 1.1.6 in \cite{STVAR} we obtain disintegration results, that is, existence of regular conditional probability law. This implies that a regular
\begin{thm}[Disintegration]\label{Dis}
Let $X$ be a polish space. Given $\PP\in$ Prob$[\Omega_{X}^{[0,\infty)}]$, let $T>0$ be a finite $\B_t$-stopping time. Then there exists a unique family of probability measures
\[
\PP|_{\B_{T}}^{\tilde{\omega}}\in \, \text{Prob}[\Omega_{X}^{[T,\infty)}]\,\, \text{for}\,\, \PP\text{-a.a.}\tilde{\omega}
\]
 such that the mapping 
 \[
 \Omega_{X}^{[0,\infty)} \ni \tilde{\omega}\mapsto \PP|_{\B_{T}}^{\tilde{\omega}}\in \, \text{Prob}[\Omega_{X}^{[T,\infty)}]
 \]
 is $\B_T$-measurable and the following properties hold
 \begin{itemize}
     \item [(a)]  For $\omega\in \Omega_{X}^{[T,\infty)}$ we have $\PP|_{\B_T}^{\tilde{\omega}}$-a.s.
     \[
     \omega(T)=\tilde{\omega}(T);
     \]
     \item[(b)] For any Borel set $A\subset\Omega_{X}^{[0,T]}$ and any Borel set $B\subset\Omega_{X}^{[T,\infty)}$,
     \[
     \PP(\omega|_{[0,T]}\in A,\omega|_{[T,\infty)}\in B)= \int_{\tilde{\omega}|_{[0,T]}\in A}\PP|_{\B_T}^{\tilde{\omega}}(B)\,\dd \PP(\tilde{\omega}).
     \]
 \end{itemize}
\end{thm}
Note, a conditional probability corresponds to disintegration of probability measure with respect to a $\sigma$-field. Accordingly, \textit{reconstruction} can be understood as the opposite process of disintegration, that is, some sort of ``gluing together" procedure. Based on \cite{STVAR} Lemma 6.1.1 and Theorem 6.1.2 we have the following results on reconstruction.

\begin{thm}[Reconstruction]\label{Rec}
Let $X$ be a Polish space. Let $\PP\in$ Prob$[\Omega_{X}^{[0,\infty)}]$ and $T$ be a finite $\B_t$-stopping time. Suppose that $Q_{\omega}$ is a family of probability measures, such that
\[
\Omega_{X}^{[0,\infty)} \ni \omega \mapsto Q_{\omega} \in\text{ Prob$[\Omega_{X}^{[T,\infty)}]$},
\]
is $\B_T$-measurable. Then there exists a unique probability measure $\PP\otimes_{T}Q$ such that :
\begin{itemize}
    \item [(a)]For any Borel set $A \in \Omega_{X}^{[0,T]} $ we have 
    \[
    (\PP\otimes_{T}Q)(A)=\PP(A);
    \]
    \item[(b)] For $\tilde{\omega}\in \Omega$ we have $\PP$-a.s.
    \[
    (\PP\otimes_{T}Q)|_{\B_T}^{\tilde\omega}=Q_{\tilde \omega}
    \]
\end{itemize}
\end{thm}
\subsubsection{Markov processes}
Following the abstract framework on Markov processes along lines of \cite{DbEfMh} and references therein we have: Assuming $(X, d_{X})$ and $(F,d_F)$ are two Polish spaces, let the embedding $F\hookrightarrow X$ be continuous and dense. Moreover, let $Y$ be a Borel subset of $F$. Since $(Y,d_F)$ is not necessarily a complete \textit{space}, we assume that the embedding $Y\hookrightarrow X$ is not dense.
A family of probability measures $\{\PP_y\}_{y\in Y}$ on $\Omega_{X}^{[0,\infty)}$ is called Markovian if we have for $y\in Y$ that 
\[
\PP_{\omega(\tau)}=\Phi_{-\tau}\PP_{y}|_{\B_T}^{\omega}-\text{a.a.}\,\omega \in \Omega_{X}^{[0,\infty)}\,\text{and all}\, \tau \geq 0.
\]

Next we define probability measures with support only on certain subset of a Polish space.
\begin{dfn}
Let $Y$ be a Borel subset of $F$ and let $U\in$ Prob$[\Omega_{X}^{[0,\infty)}]$. A family of probability measures $U$ is concentrated  on the paths with values in $Y$ if there is some $A \in \B(\Omega_{X}^{[0,\infty)})$ such that $\PP(A)=1$ and $A\subset\{\omega\in \Omega_{X}^{[0,\infty)}:\omega(\tau)\in Y\forall \tau \geq 0\}$. We write $U\in$ $\mathrm{Prob}_{Y}[\Omega_{X}^{[0,\infty)}]$.
\end{dfn}
We generalise the classical Markov property (\textit{to a situation where it only holds for a.e time-point}) as follows:
\begin{dfn}[Almost Sure Markov property]
Let $y\mapsto \PP_y$ be a measurable map defined on a measurable subset $Y\subset F$ with values in $\mathrm{Prob}_{Y}[\Omega_{X}^{[0,\infty)}]$. The family $\{\PP_y\}_{y\in Y}$ has the almost sure Markov property if for each $y\in Y$ there is a set $\mathfrak{Z}\subset(0,\infty)$ with zero Lebesgue measure such that
\[
\PP_{\omega(\tau)}=\Phi_{-\tau}\PP_{y}|_{\B_T}^{\omega}\quad\text{for}\,\,\PP_{y}-\text{a.a.} \omega \in \Omega_{X}^{[0,\infty)}
\]
and all $\tau \not\in \mathfrak{Z} $
\end{dfn}
Finally, based on the link between disintegration and reconstruction as observed by \cite{Krl}, we have the following definition.
\begin{dfn}[Almost sure pre-Markov family]\label{almostMark}
Let $Y$ be a Borel subset of $F$. Let $\mathcal{C}:Y\to \mathrm{Comp}(\mathrm{Prob}[\Omega_{X}^{[0,\infty)}]\cap\mathrm{Prob}_{Y}[\Omega_{X}^{[0,\infty)}])$ be a measureable map, where Comp($\cdot$) denotes the family of all compact subsets. The family $\{\mathcal{C}(y)\}_{y\in Y}$ is almost surely pre-Markov if for each $y\in Y$ and $U\in \mathcal{C}(y)$ there is a set $\mathfrak{Z}\subset(0,\infty)$ with zero Lebesgue measure such that the following holds for all $\tau \in \mathfrak{Z}$:
\begin{itemize}
    \item [(a)] The disintegration property holds, that is, we have
    \[
    \Phi_{-\tau}\PP|_{\B_T}^{\omega} \in \mathcal{C}(\omega(\tau))\quad\text{for}\,\,\PP\text{-a.a.}\,\, \omega \in \Omega_{X}^{[0,\infty)};
    \]
    \item[(b)] The reconstruction property holds, that is, for each $\B_T$-measurable map $\omega\mapsto Q_\omega:\Omega_{X}^{[0,\infty)}\to \mathrm{Prob}(\Omega_{X}^{[\tau,\infty)})$ with 
    \[
    \Phi_{-\tau}Q_{\omega} \in \mathcal{C}(\omega(\tau)) \quad\text{for}\,\,\PP\text{-a.a.}\,\, \omega \in \Omega_{X}^{[0,\infty)}
    \]
    we have $\PP\otimes_{\tau}Q\in \mathcal{C}(y)$.
\end{itemize}
\end{dfn}
Note Definition \ref{almostMark} is motivated by results in \cite{FlROm,GoRo}. We conclude our probability framework by stating the following results.
\begin{thm}[Markov Selection ?]\label{Mthm} Let $Y$ be a Borel subset of $F$. Let $\{\mathcal{C}(y)\}_{y\in Y}$ be an almost sure pre-Markov family (as defined in )with nonempty convex values. Then there is a measurable map $y\mapsto U_y$ defined on $Y$ with values in $\mathrm{Prob}_{Y}[\Omega_{X}^{[0,\infty)}]$ such that $\PP_{y}\in \mathcal{C}(y)$ for all $y\in Y$ and $\{\PP_{y}\}_{y\in Y}$ has the almost sure Markov property (as defined in )
\end{thm}

The following proposition is proved in \cite{FlROm}.
\begin{prop}[\cite{FlROm}, Proposition B.1]\label{quad} Let $\alpha$ and $\beta$ be two real-valued continuous and $(\B_t)$-adapted stochastic processes on $\Omega$ such that and let $t_0\geq 0$. For $U\in \mathrm{Prob}[\Omega]$ the following conditions are equivalent:
\begin{itemize}
    \item [(a)] $(\alpha_t)_{t\geq 0}$ is a $((\B_t)_{t\geq 0},U)$-square integrable martingale with quadratic variation $(\beta_t)_{t\geq 0}$
    \item[(b)] For $U$-a.a. $\omega \in \Omega$ the stochastic process $(\alpha_t)_{t\geq t_0}$ is a $((\B_t)_{t\geq t_0},U_{\B_{t_0}}^{\omega})$-square integrable martingale with quadratic variation $(\beta_t)_{t\geq t_0}$ and we have $\E^{U}[\E^{U|_{\B_{t_0}}^{\cdot}}[\beta_t]]<\infty$ for all $t\geq t_0$.
\end{itemize}

\end{prop}


\subsection{Stochastic analysis}\label{s:stoc}
Let $(\Omega, \FF, (\FF_t)_{t\geq 0})$ be a complete stochastic basis with a probability measure $\p$ on $(\Omega,\FF)$ and right-continuous filtration $(\FF_t)_{t\geq 0}$. Let $\UU$ be a separable Hilbert space with orthonormal basis $(e_k)_{k\in \N}$. We denote by $ L_2(\UU,L^2(\T))$ the set of Hilbert-Schmidt operators from $\UU$ to $L^2(\T)$. The stochastic process $W$ is a cylindrical Wiener process $W =(W_t)_{t \geq 0}$ in $\UU$, and is of the form
\begin{equation}
W(s) = \sum_{k\in \N}e_k \beta_k(s),
\label{Dwiener}
\end{equation}

where $(\beta_k)_{k \in \N}$ is a sequence of independent real-valued Wiener process relative to $(\FF_t)_{t\geq0}$. To identify the precise definition of the diffusion coefficient, set $\UU =\ell^2$ and consider $\varrho \in L^{\gamma}(\T), \varrho > 0$, then the mapping $\phi \in L_2(\UU,L^2(\T))$, that is, $\phi:\UU \to L^2(\T)$ is defined as follows

\[
\phi(e_k) =\phi_k.
\]

We suppose that $\phi$ is a Hilbert-Schmidt operator such that

\begin{equation}\label{eq:HS}
    \sum_{k\geq 1}\|\phi(e_k)\|_{L^{\infty}(\T)}^2<\infty.
\end{equation}
Consequently, since $\phi_k$ is bounded we deduce 
\begin{equation}\label{fyB}
\|\sqrt{\varrho}\phi_k\|_{L_2(\UU, L^2(\T))}^2\lesssim c(\phi)( \|\varrho\|_{L^{1}(\T)}).
\end{equation}
The stochastic integral 
\[
\int_{0}^{\tau} \varrho \phi \,\dd W = \sum_{k\geq1}\int_{0}^{\tau}\varrho\phi(e_k)\, \dd \beta_k,
\]
takes values in the Banach space  $C([0,T];W^{-k,2}(\T))$ in the sense that
\begin{equation}
    \int_{\T}\left(\int_{0}^{\tau}\varrho\phi \,\dd W\cdot \varphi \right)\,\dd x =\sum_{k\geq1}\int_{0}^{\tau}\left(\int_{\T}\varrho\phi(e_k)\cdot \varphi \,\dd x\right)\,\dd \beta_k, \quad \varphi \in W^{k,2}(\T), k>\frac{3}{2}.
\end{equation}
Finally, we define the auxiliary space $\UU_0$ with  $\UU \subset \UU_0$ as 
\begin{eqnarray}
\UU_0 :&=&\Bigg\{ e = \sum_{k}\alpha_k e_k:\sum_{k}\frac{\alpha_{k}^{2}}{k^2} < \infty \Bigg\}, \nonumber\\
\| e\|_{\UU_0}^{2}:& =& \sum_{k}^{\infty}\frac{\alpha_{k}^{2}}{k^2}, \ e=\sum_{k}\alpha_ke_k,
\label{auxiliary}
\end{eqnarray}
so that the embedding $\UU \hookrightarrow \UU_0$ is Hilbert-Schmidt, and the trajectories of $W$ belong $\p$-a.s. to the class $C([0,T];\UU_0)$ (see \cite{Prato}).

\subsection{Strong solutions}
The concept of weak(measure-valued)-strong uniqueness principle for dissipative measure-valued solutions requires existence of strong solutions. These solutions are strong in the probabilistic and PDE sense, at least locally in time. In particular, the Euler system (\ref{Euler}) will
be satisfied pointwise (not only in the sense of distributions) on the given stochastic basis associated
to the cylindrical Wiener process W.

\begin{dfn}[Strong Solution]
Let $(\Omega, \FF, (\FF_t)_{t\geq 0}, \p)$ be a complete stochastic basis with right continuous filtration, let $W$ be an $(\FF_t)_{t\geq 0}$-cylindrical Wiener process. The triplet $[r, \Theta, \vec U]$ and a stopping time $\mathfrak{t}$ is called a (local) strong solution to the system (\ref{Euler}) provided:

\begin{itemize}
    \item the density $r> 0$ $\p$-a.s., $t \mapsto r (t,\cdot)\in W^{3,2}(\T)$ is $(\FF_t)_{t\geq 0}$-adapted,
    \[
    \E\left[\sup_{t\in [0,T]}\|r (t,\cdot)\|_{W^{3,2}}^p\right]<\infty \quad\text{for all $1\leq p <\infty$};
    \]
    \item the temperature $\Theta > 0$ $\p$-a.s., $t \mapsto \Theta (t,\cdot)\in W^{3,2}(\T)$ is $(\FF_t)_{t\geq 0}$-adapted,
    \[
    \E\left[\sup_{t\in [0,T]}\|\Theta (t,\cdot)\|_{W^{3,2}}^p\right]<\infty \quad\text{for all $1\leq p <\infty$};
    \]
    \item the velocity $t\mapsto \vec U(t,\cdot) \in W^{3,2}(\T)$is $(\FF_t)_{t\geq 0}$-adapted,
    \[
    \E\left[\sup_{t\in [0,T]}\|\vec U (t,\cdot)\|_{W^{3,2}}^p\right]<\infty \quad\text{for all $1\leq p <\infty$};
    \]
    \item  for all $t\in [0,T]$ there holds $\p$-a.s.
    \[
    r(t\wedge\mathfrak{t}) = \varrho(0)- \int_{0}^{t\wedge \mathfrak{t}}\mathrm{div}_x(r\vec U)\,\dd t,
    \]
    \[
    (r\vec U)(t\wedge \mathfrak{t})=(r\vec U)(0)-\int_{0}^{t\wedge}\mathrm{div}(r\vec U\otimes \vec U)\, \dd t-\int_{0}^{t\wedge\mathfrak{t}}\nabla_x p(r,\Theta)\,\dd t +\int_{0}^{t\wedge \mathfrak{t}}r\phi\, \dd W,
    \]
    \[
    (rs(r,\Theta))(t\wedge\mathfrak{t}) = (rs(r,\Theta))(0)- \int_{0}^{t\wedge \mathfrak{t}}\mathrm{div}_x(rs(r,\Theta)\vec U)\,\dd t,
    \]
    where $s$ is the total entropy given by (\ref{entropy}).
\end{itemize}
\end{dfn}

\begin{rmk}
We expect blow up in finite time for strong solutions as in the deterministic case \cite{Sm}. 
\end{rmk}

\subsection{The approximate system}\label{Bsec}
To begin, we introduce a cut-off function 

\[
\chi \in C^{\infty}(\R), \chi(z)= \begin{cases}
1 \, \text{for}\, z \leq 0,\\
\chi'(z) \leq 0\, \text{for}\, 0 < z<1,\\
\chi(z)=0 \, \text{for}\, z\geq 1,
\end{cases}
\]
together with the operator
\begin{equation}\label{cut}
\phi_{\varepsilon}=\chi \left(|\vec v|-\frac{1}{\varepsilon}\right)\phi,\quad \varepsilon >0.
\end{equation}
Let $Q = (0,T)\times \T$ be the periodic space-time cylinder. We consider a stochastic variant of a system introduced in \cite{KrZa}, and further refined in \cite{HoFeBr}.  That is, the \textit{complete} Euler system (\ref{Euler}) is approximated by:

\begin{equation}\label{eq:Euler}
\begin{cases}
\dd \varrho + \mathrm{div}(\varrho \vec u )\dd t =0 \quad\qquad\qquad\qquad\qquad\,\,\,\,\text{in $Q$,}\\
\dd \varrho \vec u + \mathrm{div}(\varrho \vec u \otimes \vec u)\dd t+\nabla_x p(\varrho, s)\,\dd t  =\varepsilon \mathcal{L}\vec u \,\dd t +\varrho\phi_{\varepsilon}\dd W\,\,
 \text{in $Q$,}\\
\dd \varrho s +\mathrm{div}(\varrho s \vec u )\,\dd t \geq 0 \quad\qquad\qquad\qquad\qquad\,\,\,\,\text{in $Q$,}
\end{cases}
\end{equation}

with initial conditions

\[
\varrho(0.\cdot) =\varrho_{0,\varepsilon}\, \vec u (0,\cdot) = \vec u_{0,\varepsilon},\, s(0,\cdot)= s_{0,\varepsilon}.
\]

Here, the unknown fields are: the fluid density $\varrho =\varrho(t,x)$, the velocity $\vec u = \vec u(t,x)$ and the total entropy ($S=\varrho s$). We denote by $\mathcal{L}$, the suitable `viscosity' operator.

\begin{rmk}[Viscosity operator]Let $W^{3,2}(\T)$ be a separable Hilbert space
\[
W^{3,2}(\T) =\bigg\{\vec u\in W^{3,2}(\T)\bigg\},
\]
 complemented with $((\,;\,))$, a scalar product on $W^{3,2}(\T)$,i.e.

\[
((\vec v; \vec w)) = \sum_{|\alpha|=3}\int_{\T}\partial_{x}^{\alpha} \vec v\cdot \partial_{x}^{\alpha} \vec w \, \dd x +\int_{\T}\vec{v}\cdot \vec{w}\, \dd x, \,\, \vec v,\vec w \in W^{3,2}(\T).
\]
\end{rmk}

In reference to \cite{Kato}, we consider a self-adjoint  operator $\mathcal{L}$  on $W^{3,2}(\T)$ associated with the bilinear form $((\,;\,))$ given by
\[
\mathcal{L}\vec u = \Delta^{3}\vec u-\vec u = \sum_{|\alpha|=3}(\partial_{x}^{\alpha})\partial_{x}^{\alpha}\vec u -\vec u.
\]

In view of the viscosity operator $\mathcal{L}$ considered, the weak formulation associated with the momentum equation in (\ref{eq:Euler}) reads
\begin{equation*}
\begin{split}
    \left[ \int_{\T}\varrho \vec u \cdot \varphi \, \dd x\right]_{t=0}^{t =\tau} = \int_{0}^{T}\int_{\T}\left[\varrho\vec u \otimes \vec u :\nabla_x\varphi + \varrho^{\gamma}\exp{\left(\frac{ s}{c_v}\right)}\mathrm{div}\varphi\right]\,\dd x\dd t \\
    -\varepsilon\int_{0}^{T}((\vec u;\varphi))\,\dd t +\int_{0}^{T}\int_{\T}\varrho \phi\cdot \varphi \,\dd x \dd W,
\end{split}
\end{equation*}
for any $\tau>0$, and any $\varphi \in W^{3,2}(\T)$. The continuity equation and total entropy in (\ref{eq:Euler}) are solved strongly, while the momentum equation is solved in the weak sense.  We expect the approximate system (\ref{eq:Euler}) to have stochastically strong solutions, but in the present work martingale solutions are sufficient for our purposes. In the following we state the existence theorem of martingale solutions to the approximate system (\ref{eq:Euler}).

\begin{thm}\label{ExthmA}
Assume (\ref{eq:HS}) holds. Let $\Lambda_{\varepsilon}$ be a Borel probability measure on $L^{\gamma}(\T)\times L^{\gamma}(\T)\times L^{\frac{2\gamma}{\gamma +1}}(\T)$ such that
\[
\Lambda\bigg\{(\varrho,S,\vec m)\in L^{\gamma}(\T)\times L^{\gamma}(\T)\times L^{\frac{2\gamma}{\gamma +1}}(\T): 0<\underline{\varrho}<\varrho<\overline{\varrho},0<\underline{\vartheta}<\vartheta<\overline{\vartheta} \bigg\}=1,
\]
where $\underline{\vartheta}, \overline{\vartheta},\underline{\varrho}, \overline{\varrho} $ are deterministic constants. Moreover, the moment estimate
\[
\int_{L^{\gamma}(\T)\times L^{\gamma}(\T)\times L^{\frac{2\gamma}{\gamma +1}}(\T)}\left\|\frac{1}{2}\frac{|\vec m|^2}{\varrho}+c_v\varrho^{\gamma}\exp{\left(\frac{S}{c_v\varrho}\right)}\right\|_{L^1(\T)}^p\,\dd \Lambda_{\varepsilon} < \infty,
\]
holds for all $p\geq 1$. Then there exists a martingale solution to the approximate problem (\ref{eq:Euler}) subject to initial law $\Lambda_{\varepsilon}$. 
\end{thm}

\subsection{Measure-valued solutions}\label{MVS}
In order to introduce the concept of stochastic measure-valued martingale solutions, we reformulate the \textit{complete} Euler system using the variables $\vec m =\varrho\vec u $ and $S = \varrho s$ so that (\ref{Euler}) reads

\begin{eqnarray}
\dd \varrho + \mathrm{div}_x\vec m\,\dd t &=&0\label{eq:aEuler}\\
\dd  \vec m + \mathrm{div}_x\left(\frac{ \vec m \otimes \vec m}{\varrho}\right)\dd t+\nabla_x p(\varrho, s)\,\dd t  &=& \varrho\phi\dd W\,\,
 \text{in $Q$,}\label{eq:bEuler}\\
\dd S +\mathrm{div}_x\left(\frac{S\vec m}{\varrho} \right)\,\dd t &\geq&0 \quad\qquad\qquad\,\,\text{in $Q$.}\label{eq:cEuler}
\end{eqnarray}

We note that, in general, the admissibility criterion for physically possible solutions is the energy equality and it is the only \textit{tool} of establishing a \textit{priori} bounds. However, the `a \textit{priori}' bounds deduced do not guarantee weak convergences of nonlinear terms $\frac{ \vec m \otimes \vec m}{\varrho}, p(\varrho,s) \in L^1(\T)$ due to the presence of \textit{oscillations} and \textit{concentrations}. Given such scenario, we  adopt the characterisation of (nonlinear terms) in the weak formulation as combination of Young measures and defect measures. 

\begin{itemize}
    \item Young measures are probability measures on the phase space, they capture the oscillations of the solution.
    \item Defect measures are measures on physical space-time and they account for the `blow up' type collapse due to possible concentration points.
\end{itemize}

We are now ready to introduce the concept of measure-valued martingale solutions to the \textit{complete} stochastic Euler system (\ref{eq:aEuler})-(\ref{eq:cEuler}). From here onward, we denote by $\mathcal{M}^+$ the space of non-negative radon measures, and we denote by  $A$ the space of ``dummy variables":
\[
A = \bigg\{[\varrho', \vec m',{S'}]\bigg|\varrho'\geq 0, \vec m' \in \R^3,S'\in \R\bigg\}
\]
By $\mathcal{P}(A)$ we denote the space of probability measures on $A$. 

\begin{dfn}[Dissipative measure-valued martingale solution]\label{E:dfn}Let $\Lambda$ be a borel probability measure on $L^{\gamma}\times L^{\frac{2\gamma}{\gamma +1}}\times L^{\gamma}$ and $\phi \in L_2(\UU;L^2(\T))$. Then 
\[
((\Omega,\FF, (\FF_t)_{t\geq 0},\p),\varrho,\vec m,S, \mathcal{R}_{\text{conv}},\mathcal{R}_{\text{press}},\mathcal{V}_{t,x}, W)
\]
is called a dissipative measure-valued solution to (\ref{eq:aEuler})-(\ref{eq:cEuler}) with initial law $\Lambda$, provided\footnote{Some of our variables are not stochastic processes in the classical sense and we use their adaptedness in the sense of random distributions as introduced in \cite{FrBrHo} (Chap. 2.2).}:

\begin{enumerate}
    \item [(a)] $(\Omega,\FF, (\FF_t)_{t\geq 0},\p)$ is a stochastic basis with complete right-continuous filtration;
    \item[(b)]$W$ is a $(\FF_t)_{t\geq 0}$-cylindrical Wiener process;
    \item[(c)] The density $\varrho$ is $(\FF_t)_{t\geq 0}$-adapted and satisfies $\p$-a.s.
    \[
    \varrho \in C_{\text{loc}}([0,\infty), W^{-4,2}(\T))\cap L_{\text{loc}}^{\infty}(0,\infty;L^{\gamma}(\T));
    \]
    \item[(d)] The momentum $\vec m$ is $(\FF_t)_{t\geq 0}$-adapted and satisfies $\p$-a.s.
    \[
    \vec m \in C_{\text{loc}}([0,\infty), W^{-4,2}(\T))\cap L_{\text{loc}}^{\infty}(0,\infty;L^{\frac{2\gamma}{\gamma +1}}(\T));
    \]
    \item[(e)] The total entropy $S$ is $(\FF_t)_{t\geq 0}$-adapted and satisfies $\p$-a.s.
    \[
    S \in L^\infty([0,\infty),L^{\gamma}(\T))\cap BV_{w,\text{loc}}(0,\infty;W^{-l,2}(\T)), l>\frac{5}{2};
    \]
    \item[(f)] The parametrised measures $(\mathcal{R}_{\text{conv}},\mathcal{R}_{\text{press}},\mathcal{V})$ are $(\FF_t)_{t\geq 0}$-progressively measurable and satisfy $\p$-a.s.
    \begin{eqnarray}\label{eq:para}
    t\mapsto \mathcal{R}_{\mathrm{conv}}(t) &\in &L_{\text{weak-(*)}}^{\infty}(0,\infty;\mathcal{M}^+(\T, \mathbb{R}^{3\times 3}));\\
    t\mapsto \mathcal{R}_{\mathrm{press}}(t) &\in &L_{\text{weak-(*)}}^{\infty}(0,\infty;\mathcal{M}^+(\T, \mathbb{R}));\\
    (t,x)\mapsto \mathcal{V}_{t,x} &\in &L_{\text{weak-(*)}}^{\infty}(Q;\mathcal{P}(A));
    \end{eqnarray}
    \item[(g)]$\Lambda=\p\circ(\varrho (0),\vec m (0),S{0})^{-1}$;
    \item[(h)]For all $\varphi \in C^{\infty}(\T)$ and all $\tau > 0$ there holds $\p$-a.s.
    \begin{equation}\label{eq:cont}
       \left[\int_{\T}\varrho\varphi\, \dd x\right]_{t=0}^{\tau =0} = \int_{0}^{\tau}\int_{\T} \vec m \cdot \nabla \varphi \, \dd x\dd t;
    \end{equation}
    \item[(i)]For all $\boldsymbol{\varphi} \in C^{\infty}(\T)$ and all $\tau > 0$ there holds $\p$-a.s.
    \begin{eqnarray}\label{eq:mcxs}
         \left[\int_{\T}\vec m \cdot \boldsymbol{\varphi}\right]_{t=0}^{t=\tau}&=&\int_{0}^{\tau}\int_{\T}\left[\frac{\vec m \otimes\vec m}{\varrho}:\nabla\boldsymbol{\varphi}+\varrho\exp{\left(\frac{S}{c_v\varrho}\mathrm{div}\boldsymbol{\varphi}\right)}\right]\, \dd x \dd t\\
         &&+\int_{0}^{\tau}\nabla \varphi:\dd \mathcal{R}_{\text{conv}} \dd t+\int_{0}^{\tau}\int_{\T}\mathrm{div} \boldsymbol{\varphi}\,\dd \mathcal{R}_{\text{press}} \dd t\nonumber\\
         &&+\int_{0}^{\tau}\boldsymbol{\varphi}\cdot \varrho \phi \, \dd x \dd W;\nonumber
    \end{eqnarray}
    \item[(j)] The total entropy holds in the sense that 
    \begin{equation}\label{eq:entr}
        \int_{0}^{\tau}\int_{\T}\left[\langle\mathcal{V}_{t,x};Z(S')\rangle\partial_t \varphi+\langle\mathcal{V}_{t,x},Z(S')\frac{\vec m'}{\varrho'}\rangle\cdot \varphi\right]\, \dd x\dd t \leq \left[\int_{\T}\langle \mathcal{V}_{t,x};Z(S')\rangle \varphi\, \dd x\right]_{t=0}^{t=\tau},
    \end{equation}
    for any $\varphi \in C^1([0,\infty)\times \T), \varphi \geq 0, \p$-a.s.,{and any $Z$, $Z \in BC(\R)$ non-decreasing.} 
    \item[(k)] The total energy satisfies
    \begin{equation}\label{eq:Ener}
        \mathbf{E}_t=\mathbf{E}_s+\frac{1}{2}\int_{s}^{t}\|\sqrt{\varrho}\phi\|_{L_2(\UU;L^2(\T))}^2\, \dd \sigma + \int_{s}^{t}\int_{\T}\vec m \cdot \phi \, \dd x \dd W,
    \end{equation}
    $\p$-a.s. for a.a. $0\leq s<t$, where
    \[
    \mathbf{E}= \int_{\T}\left[\frac{1}{2}\frac{|\vec m |^2}{\varrho}+c_v\varrho^{\gamma}\exp{\left(\frac{S}{c_v\varrho}\right)}\right]\, \dd x + \frac{1}{2}\int_{\T}\dd \mathrm{tr}\mathcal{R}_{\text{conv}}(t) +c_v\int_{\T}\dd \mathrm{tr}\mathcal{R}_{\text{press}}(t)
    \]
    for $\tau \geq 0$ and 
    \[
   E_0= \int_{\T}\left[\frac{1}{2}\frac{|\vec m_0 |^2}{\varrho_0}+c_v\varrho_0^{\gamma}\exp{\left(\frac{S_0}{c_v\varrho_0}\right)}\right]\, \dd x.
    \]
\end{enumerate}
\end{dfn}

\begin{rmk}
The use of cut-off function $Z$ in (\ref{eq:entr}) is inspired by Chen and Frid \cite{ChFr}.
\end{rmk}

\subsection{Main results}
In light of the above discussion, we are now ready to state the main results of the paper. The existence of dissipative measure-valued martingale solutions follows from the the following theorem.
\begin{thm}\label{ExMainr}
Assume (\ref{eq:HS}) holds. Let $\Lambda$ be a Borel probability measure on $L^{\gamma}(\T)\times L^{\gamma}(\T)\times L^{\frac{2\gamma}{\gamma +1}}(\T)$ such that
\[
\Lambda\bigg\{(\varrho,S,\vec m)\in L^{\gamma}(\T)\times L^{\gamma}(\T)\times L^{\frac{2\gamma}{\gamma +1}}(\T): 0<\underline{\varrho}<\varrho<\overline{\varrho},0<\underline{\vartheta}<\vartheta<\overline{\vartheta} \bigg\}=1,
\]
where $\underline{\vartheta}, \overline{\vartheta},\underline{\varrho}, \overline{\varrho} $ are deterministic constants. Moreover, the moment estimate

\[
\int_{L^{\gamma}(\T)\times L^{\gamma}(\T)\times L^{\frac{2\gamma}{\gamma +1}}(\T)}\left\|\frac{1}{2}\frac{|\vec m|^2}{\varrho}+c_v\varrho^{\gamma}\exp{\left(\frac{S}{c_v\varrho}\right)}\right\|_{L^1(\T)}^p\,\dd \Lambda < \infty,
\]
holds for all $p\geq 1$. Then there exists a dissipative measure-valued martingale solution to the \textit{complete} stochastic Euler system  (\ref{eq:aEuler})-(\ref{eq:cEuler}) in the sense of Definition \ref{E:dfn} subject to initial law $\Lambda$. 
\end{thm}

Furthermore, in view of results in \cite{FEJB}, additional to the natural physical principles prescribed in (\ref{caloric})-(\ref{entropy}) we need a purely technical hypothesis

\begin{equation}\label{pressure}
|p(\varrho,\vartheta)|\lesssim (1+\varrho +\varrho e(\varrho,\vartheta)+\vartheta|s(\varrho,\vartheta)|),
\end{equation}
to establish bounds. Consequently, we then establish the following weak (measure-valued)-strong uniqueness principle:
\begin{thm}\label{thm_a} The pathwise weak-strong uniqueness holds true for the system (\ref{eq:aEuler})-(\ref{eq:cEuler}) in the following sense.
Let the thermodynamics functions $ e =e(\varrho,\vartheta), s = s(\varrho,\vartheta),$ and $p=p(\varrho,\vartheta)$ satisfy the Gibbs' relation (\ref{gibbs}), and the technical hypothesis  (\ref{pressure}).
 let 
\[[((\Omega , \FF , (\FF)_{t \geq 0}, \mathbb{P} ),\varrho,\vec m,  S, \mathcal{R}_{\mathrm{conv}},\mathcal{R}_{\mathrm{press}},\mathcal{V}_{t,x},W)
\]
 be a dissipative measure-valued martingale solution to (\ref{eq:aEuler})-(\ref{eq:cEuler}) in the sense of Definition (\ref{E:dfn}), let $[r,\Theta,\vec U]$  and a stopping time $\mathfrak{t}$ be a strong solution of the same problem defined on the stochastic basis with the same Wiener process and with initial data
 \begin{equation}\label{Idata}
 \varrho(0,\cdot) =r(0,\cdot),\quad\vec u (0,\cdot)={\vec U}(0,\cdot),\qquad \vartheta (0,\cdot)=\vec \Theta(0,\cdot)\quad\p\text{-a.s.}
 \end{equation}
Then
 \[ [{\varrho},{\vartheta},{\vec u}](\cdot \wedge \mathfrak{t}) \equiv [r,{\Theta},{\vec U}] (\cdot \wedge \mathfrak{t}),
\]
and
\[ \mathcal{R}_{\mathrm{conv}}=\mathcal{R}_{\mathrm{press}}=0,   \]
 $\p$-a.s., and for any $(t,x)\in (0,T)\times\T$
 \[
  \mathcal{V}_{t\wedge \mathfrak{t},x}= \delta_{ s(r,\Theta)},
\]
$\p$-a.s.
\end{thm}

\section{Basic approximate problem}\label{Bapp}
This section of the paper is devoted to the sketch proof of Theorem \ref{ExthmA}, that is, existence of martingale solutions to (\ref{eq:Euler}). The sketch proof of the theorem follows from the ideas presented in [4] to which we refer to for further details. We construct these solutions via a multi-level approximation scheme. The idea here is to start with a finite dimensional approximation of Galerkin type. However, as a consequence of maximum principle (usually incompatible with Galerkin type approximation) this can only be applied to the momentum equation since the we need the density $\varrho$ and temperature $\vartheta$ to be positive at the first level of approximation. Adopting the the approximation scheme introduced in \cite{EA} and adapted to the stochastic setting in \cite{BF} with appropriate adjustments, we regularize our system as follows:\newline

Let $\Delta$ be the Laplace operator defined on the periodic domain $\T$. Let $\{\vec w_n\}_{n\geq 1}$ be the orthonormal system of the associated eigenfunctions. We consider the Galerkin method given by the family of finite-dimensional spaces;

\[
H_m=\left(\rm{span}\bigg\{\vec w_n\,\,\,\bigg|\,n\leq m\bigg\}\right)^3, m=1,2,\ldots
\]
endowed with the Hilbert structure of the Lebesgue space $L^2(Q,\R^3)$. Let 
\[
\Pi_{m}:L^2(Q,\R^3)\to H_m,
\]
be the associated $L^2$-orthogonal projection, and we have $W^{2,2}(\T,\R^3)\hookrightarrow\hookrightarrow C({\T},\R^3)$. Indeed,

\begin{equation}\label{eq:prest}
    \|\Pi_m[\vec f]\|_{L^{\infty}(\T)}\lesssim\|\Pi_m[\vec f]\|_{W^{2,2}(\T)}\lesssim \|\vec f\|_{W^{2,2}(\T)},
\end{equation}
where the associated embedding constants are independent of $m$. Furthermore, since $H_m$ is finite dimensional, all norms are equivalent on $H_m$ for any fixed $m$- (a property that will be frequently used at the first level of approximation). Finally, we introduce the operator
\[
[\vec v]_R =\chi(\|\vec v\|_{H_m}-R)\vec v,
\]
defined for $\vec v \in H_m$. Let $Q =(0,T)\times \T$ be the space-time cylinder, we seek to solve the basic approximate system:

\begin{align}\label{eq:Ba}
  \dd \varrho&+ \mathrm{div}(\varrho[\vec u]_R) \,\dd t =0,\\
  \dd \Pi_m[\varrho\vec u]&+\Pi_m[\mathrm{div}(\varrho[\vec u]_R\otimes\vec u)]\,\dd t+\Pi_m\Big[\chi(\|\vec u\|_{H_m}-R)\nabla\left(p(\varrho,\vartheta)\right)\Big]\label{eq:Bb}\\
  &=\Pi_m\Big[\varepsilon\mathcal{L}\vec u\Big]\,\dd t +\varrho\Pi_m[(\phi_{\varepsilon})]\,\dd W, \nonumber\\
  \label{eq:Bc}
  \dd S&+\big[\mathrm{div}(S[\vec u]_R)\big]\,\dd t=0,
\end{align}
subject to initial law $\Lambda$,
prescribed with  random initial data
\begin{equation}\label{eq:ID}
\begin{split}
\varrho(0, \cdot) &= \varrho_0 \in C^{2 + \nu}(\T), \, \varrho_0 > 0,
\vartheta(0, \cdot) = \vartheta_{0},\, \vartheta_{0} \in 
C^{2+\nu}(\T), \, \vartheta_{0} > 0, \\
\vec u (0, \cdot) &= \vec u_{0} \in H_m.
\end{split}
\end{equation}
In our basic approximate system (\ref{eq:Ba})-(\ref{eq:Bc}), equations (\ref{eq:Ba}) and (\ref{eq:Bc}) are deterministic, that is, they can be solved pathwise, and equation (\ref{eq:Bb}) involves stochastic integration. 
The Galerkin projection applied above reduces the problem to a variant of ordinary stochastic differential equation. We solve the system (\ref{eq:Ba})-(\ref{eq:ID}) using an iteration scheme.

\subsection{Iteration Scheme}
We construct solutions to (\ref{eq:Ba})-(\ref{eq:ID}) using a modification of the Cauchy collocation method. Thus, fixing a time step $h>0$ we set 
\begin{equation}\label{eq:id}
    \varrho(t,\cdot)=\varrho_0, \quad \vartheta(t,\cdot)=\vartheta_0, \quad \vec u (t,\cdot)=\vec u_0, \quad \text{for $t\leq 0$},
\end{equation}
 and define recursively, for $t\in [nh,(n+1)h)$
 
 \begin{equation}
    \label{eq:Baa}
  \dd \varrho+ \mathrm{div}(\varrho[\vec u(nh,\cdot)]_R) \,\dd t =0,\quad \varrho(nh,\cdot)=\varrho(nh{-},\cdot),
 \end{equation}
 \begin{equation}\label{eq:Bcc}
  \dd S+\big[\mathrm{div}(S[\vec u(nh,\cdot)]_R)\big]\,\dd t=0,\quad \vartheta(nh,\cdot)=\vartheta(nh-,\cdot),
 \end{equation}
 Here, the unknown quantities $\varrho,\vartheta$ are uniquely deduced from the deterministic equations (\ref{eq:Baa}) and (\ref{eq:Bcc}) in terms of $\vec u$ and initial data. Now given $\varrho,\vartheta$ we solve
 \begin{align}
  \dd \Pi_m[\varrho\vec u]&+\Pi_m\left[\mathrm{div}\Big(\varrho[\vec u(nh,\cdot)]_R\otimes\vec u(nh,\cdot)\Big)\right]\,\dd t+\Pi_m\Big[\chi(\|\vec u(nh,\cdot)\|_{H_m}-R)\nabla\left(p(\varrho,\vartheta)\right)\Big]\, \dd t\label{eq:Bbb}\\
  &=\Pi_m\Big[\varepsilon\mathcal{L}\vec u\Big]\,\dd t +\Pi_m[\varrho(\phi_{\varepsilon})]\,\dd W,\quad t\in [nh,(n+1)h),\quad \vec u(nh,\cdot)=\vec u(nh-). \nonumber
\end{align}
To solve (\ref{eq:Bbb}), it is convenient to reformulate the system using $\dd \vec u$.  We observe that
\[
\dd\Pi_{m}(\varrho\vec u) = \Pi_{m}(\dd \varrho\vec u)+\Pi_{m}(\varrho\dd \vec u)=\Pi_{m}(\partial_t\varrho\vec u)+\Pi_{m}(\varrho\dd \vec u).
\]
We adopt the linear mapping $\mathcal{M}[\varrho]$,
\[
\mathcal{M}[\varrho]:H_m\to H_m,\quad \mathcal{M}[\varrho](\vec u)=\Pi_m(\varrho \vec u),
\]
or, equivalently,
\[
\int_{Q}\mathcal{M}[\varrho]\vec u\cdot \varphi \, \dd x \equiv \int_{Q}\varrho\vec u\cdot \varphi \, \dd x\quad\text{for all $\varphi\in H_m$},
\]
and its properties as introduced in (\cite{FeNoPe}, section 2.2). To be specific, using maximum principle, we take $\varrho$ to be bounded from below away from zero so that the operator $\mathcal{M}[\varrho]$ is invertible. Then we reformulate the relation in (\ref{eq:Bbb}) to obtain

\begin{align}\label{eq:It}
 \vec u(t)-\vec u (nh-)&+\mathcal{M}^{-1}[\varrho(t)]\int_{nh}^{t}\Pi_m\left[\mathrm{div}\Big(\varrho[\vec u(nh,\cdot)]_R\otimes\vec u(nh,\cdot)\Big)\right]\,\dd t\nonumber\\
 &+\mathcal{M}^{-1}[\varrho(t)]\int_{nh}^{t}\Pi_m\Big[\chi(\|\vec u(nh,\cdot)\|_{H_m}-R)\nabla\left(p(\varrho,\vartheta)\right)\Big]\, \dd t\\
 &=\mathcal{M}^{-1}[\varrho(t)]\int_{nh}^{t}\Pi_m\Big[\varepsilon\mathcal{L}\vec u\Big]\,\dd t\nonumber\\
 &+\mathcal{M}^{-1}[\varrho(t)]\int_{nh}^{t}\varrho\Pi_m[(\phi_{\varepsilon})]\,\dd W,\quad  nh<t<(n+1)h. \nonumber
\end{align}

The constructed iteration scheme (\ref{eq:Baa})-(\ref{eq:Bbb}) gives a unique solution $[\varrho,\vartheta,\vec u]$ for any initial data (\ref{eq:id}), and the variables $\varrho,\vartheta$ and $\vec u$ are continuous in time $\p$-a.s. Indeed, we find solution $\varrho$ and $\vartheta$ such that
\[\varrho \in C([0,T];C^{2+\nu}(\T), \vartheta \in C([0,T];C^{2+\nu}(\T)\, \text{a.s.} \]
by applying standard results (see, e.g \cite{AlZa}) pathwise. Finally, given $\varrho$ and $\vartheta$ we can find the velocity
\[
\vec u \in C([0,T]; H_m), \p\text{-a.s.}
\]
solving (\ref{eq:Bbb}) recursively.
 \subsubsection{The limit for vanishing time step}
 The solution $[\varrho,\vartheta,\vec u]$ provided by the iteration scheme (\ref{eq:Baa})-(\ref{eq:Bbb}) exists for any time step $h$. Next, we show that as $h\to 0$ the iteration scheme yields the basic approximate system (\ref{eq:Ba})-(\ref{eq:Bc}).  This essentially follows from establishing uniform bounds for (\ref{eq:Baa})-(\ref{eq:Bbb}) independent of $h$ following the arguments presented in (\cite{BF}, section 3.2). In particular,
 we assume the initial data satisfies the bounds
 \[
0<\underline{\varrho}\leq\varrho_0, \|\varrho_0\|_{C^{2+\nu}(\T)}\leq \overline{\varrho}, \quad
0<\underline{\vartheta}\leq\vartheta_0, \|\vartheta_0\|_{C^{2+\nu}(\T)}\leq \overline{\vartheta},
\]
for deterministic constants $\underline{\varrho}$ and $\overline{\varrho}$ with $\nu > 0$. Taking into account the standard results on compressible transport equations and that $\vec [\vec u]_R$ is bounded in any Sobolev space in terms of $R$, and the equivalence of norms we have:

 \begin{itemize}
     \item A priori bound for density $\varrho$ is given by 
\begin{equation}\label{eq:da}
    \text{ess}\sup(\|\varrho(t,\cdot)\|_C^{2+\nu}+\|\partial_t\varrho(t,\cdot)\|_{C^\nu}+ \|\varrho^{-1}(t)\|_{C(\overline{Q})}\lesssim c(m,R,T,\underline{\varrho},\overline{\varrho}),\, \p\text{-a.s}
\end{equation}
 uniformly in $h$ for deterministic constants $\underline{\varrho}$ and $\overline{\varrho}$ with $\nu > 0$.
\item Similarly, a priori bound for total entropy is 
\begin{equation}\label{eq:ea}
    \text{ess}\sup(\|S(t,\cdot)\|_C^{2+\nu}+\|\partial_t S(t,\cdot)\|_{C^\nu}+ \|S^{-1}(t)\|_{C(\overline{Q})}\lesssim c(m,R,T,\underline{S},\overline{S}),
\end{equation}
$\p$-a.s., where the same bound of $\vartheta$ follows immediately from using $S=\varrho(\log \vartheta^{c_v}-\log(\varrho))$, for deterministic constants $\underline{\vartheta}$ and $\overline{\vartheta}$ uniform in $h$.
\item

Following the lines in \cite{BF} (Section 3.2), that is, establishing bounds for
relation (\ref{eq:Bbb}), we use a test function $\varphi \in H_m$ and take a supremum over $\varphi$, pass to expectations and apply Burkholder Davis-Gundy inequality to control the noise term noise. Finally, applying Gronwall's lemma we deduce the estimate

\begin{equation}\label{eq:enr}
\E\left[\sup_{\tau \in [0,T]}\|\Pi_m[\varrho\vec u](\tau,\cdot)\|_{H_m}^{r}+ \varepsilon \sup_{\tau \in [0,T]}\|\vec u(\tau,\cdot)\|_{H_m}^{r}\right]\lesssim c(r,T)\E[1+\|\vec u_0 \|_{H_m}^{r}],\quad r>1.
\end{equation}
To pass to the limit $h \to 0$ in the momentum equation (\ref{eq:Bbb}) we require 
  the uniform bound (\ref{eq:enr}) and  compactness on the velocities in the space $C([0,T],H_m)$. Furthermore, we need to control the difference $(\vec u-\vec u(nh,\cdot))$ uniformly in time. Similarly, following closely the presentations in \cite{BF} with appropriate modifications to our particular case we infer 
  \begin{equation}\label{eq:Enest}
    \E\left[\|\vec u\|_{C^{\beta}([0,T];H_m)}\right] \lesssim \E\left[\|\vec u_0\|_{H_m}^{r}+1\right], \quad r> 2 ,\beta \in \left(0,\frac{1}{2}-\frac{1}{r}\right),
  \end{equation}
  uniformly in $h$. The `a priori' bounds (\ref{eq:da})-(\ref{eq:Enest}) are sufficient to take the limit $h\to 0$ in the iteration scheme (\ref{eq:Baa})-(\ref{eq:Bbb}).
 \end{itemize}

We consider the joint law of the basic state variables $[\varrho,\vartheta,\vec u, W]$ in the pathspace

\begin{equation}\label{eq:psc}
    \mathfrak{G} \equiv C^l([0,T];C^{2+l}(\T))\times C^l([0,T];C^{2+l}(\T))\times C^l([0,T];H_m)\times C([0,T];\UU_0),\quad l\in (0,\overline{\nu}),
\end{equation}
where $\overline{\nu}$ is the minimum H\"older exponent in (\ref{eq:Enest}). Now let $[\varrho_h,\vartheta_h,\vec u_h, W]$ be the unique solution to the iteration scheme (\ref{eq:Baa})-(\ref{eq:Bbb}), with initial data being $\FF_0$ measurable and satisfying 
\[
0<\underline{\varrho}\leq\varrho_0, \|\varrho_0\|_{C^l([0,T];C^{2+l}(\T))}\leq \overline{\varrho}, \quad
0<\underline{\vartheta}\leq\vartheta_0, \|\vartheta_0\|_{C^l([0,T];C^{2+l}(\T))}\leq \overline{\vartheta},
\]
as well as 
\begin{equation}
    \E\bigg[\|\vec u_0\|_{H_m}^r\bigg]\leq \overline{u}\quad \text{for some $r>2$.}
\end{equation}
$\p$-a.s., 
Denote by $\mathcal{L}[\varrho_h,\vartheta_h,\vec u_h, W]$ the joint law of  $[\varrho_h,\vartheta_h,\vec u_h, W]$ on $\mathfrak{G}$, and by $\mathcal{L}[\varrho_h],\mathcal{L}[\vartheta_h],\mathcal{L}[\vec u_h]$ and $ \mathcal{L}[W]$ the corresponding marginals, respectively. As a consequence of bounds established (\ref{eq:da})-(\ref{eq:Enest}), the joint law $\mathcal{L}[\varrho_h,\vartheta_h,\vec u_h, W]$ is \textit{tight} on the Quasi-Polish space $\mathfrak{G}$. By applying Jabubowski-Skorokhod's representation theorem \cite{Jaku} we get a new probability space with new random variables, a.s convergence of new variables on the pathspace (\textit{w.l.o.g} we keep the same notation). Performing the limit $h\to 0$ in the new probablity space yields 

\begin{align}\label{eq:Ba:lm}
  \partial_t \varrho&+ \mathrm{div}(\varrho[\vec u]_R) \, =0,\\
  \dd \Pi_m[\varrho\vec u]&+\Pi_m[\mathrm{div}(\varrho[\vec u]_R\otimes\vec u)]\,\dd t+\Pi_m\Big[\chi(\|\vec u\|_{H_m}-R)\nabla\left(p(\varrho,\vartheta)\right)\Big]\label{eq:Bb:lm}\\
  &=\Pi\Big[\varepsilon\mathcal{L}\vec u\Big]\,\dd t +\varrho\Pi_m[(\phi_{\varepsilon})]\,\dd W, \nonumber\\
  \label{eq:Bc:lm}
  \partial_t S&+\mathrm{div}(S[\vec u]_R)\,=0.
\end{align}

The system (\ref{eq:Ba:lm})-(\ref{eq:Bc:lm}) is still depended on $R$ and $m$. Now our goal is to perform the limits $R\to \infty$ and $m\to \infty$, respectively. The procedure is similar to the above discussion for the limit $h\to 0$. To proceed as discussed above, we start off by deducing uniform bounds enforced by random initial data and the energy balance.

\subsubsection{Energy balance}
A solution to the approximate system (\ref{eq:Ba})-(\ref{eq:Bc}) satisfies a variant of energy balance. Derivation of total energy balance to the system consist of testing (\ref{eq:Bb}) with the test function $\vec u$ and integrating by parts the resultant formulation. Observe that the scalar product 

\[
\int_{\T}\Pi_m(\varrho\vec u)\cdot \vec u \dd x= \int_{\T} \varrho |\vec u|^2 \dd x
\]
and the linear mapping $\mathcal{M}$ yields
\[
\int_{\T}\mathcal{M}^{-1}[\varrho]\Pi_m[\vec u]\cdot \Pi_m[\varrho\vec u]\, \dd x = \int_{\T}\varrho \mathcal{M}^{-1}[\varrho]\Pi_m[\vec u]\cdot \vec u\,\dd x = \int_{\T}\vec u\cdot\vec u \,\dd x.
\]
Now we are ready to derive the total energy balance, for this, we consider the following proposition.

\begin{prop}\label{energy_p}
let $[\varrho,\vartheta,\vec u,W]$ be a martingale solution of the basic approximate system  (\ref{eq:Ba})-(\ref{eq:Bc}). Then the following total energy balance equations

\begin{equation}\label{energy}
    \dd \int_{\T}\left[\frac{1}{2}\varrho|\vec u|^2+ \varrho e\right]\, \dd x +\varepsilon((\vec u, \vec u))\, \dd t =\frac{1}{2}\sum_{k=1}^{\infty}\int_{\T}\varrho|\Pi_m[\phi_\varepsilon{e_k}]|^{2}\, \dd x \dd t +\int_{\T}\varrho\Pi_m[\phi_\varepsilon]\cdot \vec u \,\dd x \dd W.
\end{equation}
holds $\p$-a.s.
\end{prop}

\begin{proof}
Applying It\^o's formula to the functional
\[
f(\varrho,\varrho \vec u) = \frac{1}{2}\int_{\T} \varrho |\vec u|^2 \,  \dd x = \frac{1}{2}\int_{\T}\frac{|\vec m|^2}{\varrho}\, \dd x,
\]

from (\ref{eq:Bb}) we obtain,

\begin{eqnarray}\label{ito}
\dd \int_{\T}\frac{1}{2}\varrho |\vec u|^2 \, \dd x&=&-\int_{\T}\bigg[\mathrm{div}(\varrho[\vec u]_R \otimes \vec u)+\chi(\|\vec u\|_{H_m}-R)\nabla_xp(\varrho,\vartheta)\bigg ] \cdot \vec u \, \dd x \dd t \nonumber\\
&&+ \int_{\T} \varepsilon \mathcal{L}\vec u \cdot \vec u \, \dd x \dd t  -\frac{1}{2}\int_{\T}|\vec u|^2 \dd \varrho \dd x  \\
&&+ \frac{1}{2}\sum_{k=1}^{\infty}\int_{\T}\varrho |\Pi_m[(\phi_{\varepsilon}){e_k}]|^2 \, \dd x \dd t + \int_{\T}\varrho\phi_{\varepsilon}\cdot \vec u \, \dd x \dd W. \nonumber
\end{eqnarray}

Furthermore, from the continuity equation (\ref{eq:Ba}), we deduce that;
\[
\frac{1}{2}\int_{\T}|\vec u|^2 \dd \varrho \dd x = -\frac{1}{2}\int_{\T}\mathrm{div}(\varrho[\vec u]_R)\cdot|\vec u|^2 \,\dd x \dd t,
\]
such that the integral with convective term simplifies to
\[
\int_{\T}\mathrm{div}(\varrho [\vec u]_R \otimes \vec u)\cdot \vec u \, \dd x   =\frac{1}{2}\int_{\T}\mathrm{div}(\varrho [\vec u]_R )\cdot  |\vec u|^2\, \dd x,
\]

and 
\begin{eqnarray*}
\int_{\T}\chi(\|\vec u\|_{H_m}-R)\nabla_xp(\varrho,\vartheta)\cdot \vec u \, \dd x &=&-\int_{\T}p(\varrho,\vartheta)\mathrm{div} \,[\vec u]_R \, \dd x.
\end{eqnarray*}
In view of the above observations, (\ref{ito}) reduces to

\begin{eqnarray}\label{ito_0}
\dd \int_{\T}\frac{1}{2}\varrho |\vec u|^2 \, \dd x+ \varepsilon((\vec u; \vec u)) &=&
\int_{\T}p(\varrho,\vartheta)\mathrm{div} \,[\vec u]_R \, \dd x \dd t \\
&&+ \frac{1}{2}\sum_{k=1}^{\infty}\int_{\T}\varrho |\Pi_m[(\phi_{\varepsilon})_{e_k}]|^2 \, \dd x \dd t + \int_{\T}\varrho\phi_{\varepsilon}\cdot \vec u \, \dd x \dd W. \nonumber
\end{eqnarray}

Finally, re-writing the entropy equation as an expression of internal energy using Gibb's relation (\ref{gibbs}) the followings holds

\[
\int_{\T}p(\varrho, s)\mathrm{div}[\vec u]_R \, \dd x= -\dd \int_{\T} \varrho e \, \dd x.
\]

Consequently, re-writing  (\ref{ito_0}) yields
\begin{eqnarray}
\dd \int_{\T}\frac{1}{2}\varrho |\vec u|^2 + \varrho e \, \dd x+ \varepsilon((\vec u; \vec u))\,\dd t
&=& \frac{1}{2}\sum_{k=1}^{\infty}\int_{\T}\varrho |\Pi_m[(\phi_{\varepsilon})_{e_k}]|^2 \, \dd x \dd t + \int_{\T}\varrho\phi_{\varepsilon}\cdot \vec u \, \dd x \dd W, \nonumber
\end{eqnarray}
as required.
\end{proof}

\subsubsection{Uniform Bounds}
Keeping $\varepsilon>0$ fixed, we derive bounds independent of the parameters $R$ and $m$. We note, the projections $\Pi_m$ are bounded by (\ref{eq:prest}). In view of (\ref{energy}), we deduce the estimate

\begin{equation}\label{eq:estimate}
    \int_{\T}\left[\frac{1}{2}\varrho|\vec u|^2+ \varrho e\right]\, \dd x +\varepsilon\int_{0}^{T}((\vec u,\vec u))\, \dd t
    \lesssim \bigg(\mathbf{E}_0 +c(T,\phi_{\varepsilon},\overline{\varrho}) + M_t\bigg),
\end{equation}
where
\[
\mathbf{E}_0 =\int_{\T}\left[\frac{1}{2}\varrho_0|\vec u_0|^2+ \varrho_0 e_0\right]\, \dd x,\quad M_t = \int_{0}^{T}\int_{\T}\varrho \Pi_m[\phi_\varepsilon]\cdot \vec u \,\dd x\dd W .
\]
Furthermore, taking the exponent of (\ref{eq:estimate}) and the expectation of the resultant exponent formulation we obtain 
\begin{equation}\label{eq:exp}
    \E\left[\exp{\left(\lambda \mathbf{E}_t+\lambda\int_{0}^{T}((\vec u,\vec u))\, \dd t\right)}\right]\leq c \,\E\bigg[\exp{\left( \lambda M_t\right)}\bigg]\lesssim c(\lambda) \quad \forall \lambda >0,
\end{equation}
 $\p$-a.s, the bound follows from applying exponential version of Burkholder-Davis-Gundy inequality to the r.h.s of (\ref{eq:exp}), and using $\chi(|\vec u|-\frac{1}{\varepsilon})\vec u \leq 1/\varepsilon$ to deduce
 \begin{eqnarray*}
 \langle \langle M_t \rangle \rangle &=& \sum_{k}\int_{0}^{T}\bigg( \int_{\T}\underbrace{\varrho \Pi_{m}(\phi_{\varepsilon})_{e_k}\cdot \vec u}_{=\varrho\Pi_m\phi{e_k}\chi(|\vec u|-\frac{1}{\varepsilon})\vec u}\,\dd x\bigg)^2\dd t\\
 &\leq&c(\varepsilon)\sum_{k}\int_{0}^{T} \bigg(\underbrace{\int_{\T}\varrho \,\dd x}_{\leq c(\overline{\varrho})}\bigg)^2\|\Pi_m\phi{e_k}\|_{L_{x}^{\infty}}^{2}\, \dd t\\
 &\lesssim& c(\varepsilon,\phi,\overline{\varrho}).
 \end{eqnarray*}
 
\textbf{Limit $R\to \infty$}. We assume the parameter $m$ is fixed. The approximate problem (\ref{eq:Ba})-(\ref{eq:Bc}) admits a martingale solution $[\varrho_R, \vartheta_R, \vec u_{R}]$ with initial law $\Lambda$ for any fixed $R>0$. To perform the limit $R\to \infty$, we establish compactness of the phase variables and use Jakubowski's variant of the Skorokhod representation theorem \cite{Jaku}.\newline

\textbf{Compactness}. We recall the standard regularity estimates of compressible transport equations in \cite{AlZa} applied to (\ref{eq:Ba:lm}):
\begin{equation}\label{eq:maxr}
\begin{split}
      \|\varrho(t,\cdot)\|_{L^{\infty}(0,T;W^{2,2}(\T))}\lesssim \|\varrho_0\|_{W^{2,2}(\T)}\exp{\left(\int_{0}^{T}\|[\vec u]_R\|_{W^{3,2}}\,\dd t\right)},\\
     \lesssim \|\varrho_0\|_{W^{2,2}(\T)}\exp{\left(\int_{0}^{T}\|\vec u\|_{W^{3,2}}\,\dd t\right)},
\end{split}
\end{equation}
and 
\begin{equation}\label{eq:maxrr}
\begin{split}
      \|\nabla\varrho(t,\cdot)\|_{L^{\infty}(0,T;W^{1,2}(\T))}\lesssim \|\varrho_0\|_{W^{2,2}(\T)}\exp{\left(\int_{0}^{T}\|[\vec u]_R\|_{W^{3,2}}\,\dd t\right)},\\
      \lesssim \|\varrho_0\|_{W^{2,2}(\T)}\exp{\left(\int_{0}^{T}\|\vec u\|_{W^{3,2}}\,\dd t\right)}.
\end{split}
\end{equation}
  We control the right-hand side of (\ref{eq:maxr}) and (\ref{eq:maxrr}) in expectation by using (\ref{eq:exp}) to deduce the estimate
\begin{equation}\label{eq:maaxr}
\E\bigg[\|\varrho\|_{L^\infty(0,T;L^{\infty}(\T))}\bigg]\lesssim c,\quad
\E\bigg[\|\nabla \varrho\|_{L^\infty(0,T;L^{6}(\T))}\bigg]\lesssim c,
\end{equation}
where $c>0$ is dependent on initial data. In view of (\ref{eq:estimate}), (\ref{eq:maxrr}),(\ref{eq:maaxr})  and (\ref{eq:Ba:lm}) we deduce that
 
 \[
 \E\bigg[\|\partial_t\varrho\|_{L^{2}(0,T;L^{\infty}(\T))}\bigg]\lesssim c .
 \]

 {Finally, we obtain the estimate}
\begin{equation}\label{eq:mxr}
  \E\bigg[\|\partial_t\varrho\|_{L^{2}(0,T;L^{\infty}(\T)} \bigg] +\E\bigg[\|\varrho\|_{L^{\infty}(0,T;L^{\infty}(\T)}\bigg]\lesssim c,
\end{equation}
where $c>0$ is dependent on initial data. The standard regularity estimate of the total entropy for the variable $S$ follows same arguments as shown for $\varrho$ and we obtain the $\vartheta$ estimate via the  relation $s=\log \vartheta^{c_v}-\log(\varrho)$.
Consequently, using (\ref{eq:Bb:lm}), (\ref{eq:estimate}) and (\ref{eq:prest}), the compactness of $\varrho \vec u$ with respect to the time variable follows from the bound
 
 \begin{equation}\label{Estimate}
\E \left[\|\varrho\vec{u}\|_{C^{\alpha}([0,T];W^{-3,2}(\T))}^r \right] \lesssim c(r),
 \end{equation}
 
for all $0<\alpha(r) < 1/2$. Accordingly, with established uniform bounds necessary to perform the limit, we proceed as in $h$-limit.
We consider the joint law of the basic state variables $[\varrho,S,\vec u, W]$ in the pathspace

\begin{equation}\label{eq:Rpsc}
    \mathfrak{G} \equiv L^{2}(0,T;W^{1,2}(\T))\times L^{2}(0,T;W^{1,2}(\T))\times C([0,T];W^{-4,2}(\T))\times L^2(0,T;W^{3,2}(\T))\times C([0,T];\UU_0).
\end{equation}
 Let $[\varrho_R,\vartheta_R,\vec m_R, W]$ be the unique solution to the iteration scheme (\ref{eq:Baa})-(\ref{eq:Bbb}) with respect to initial law $\Lambda$ and assume
\[
0<\underline{\varrho}\leq\varrho_0, \|\varrho_0\|_{W^{2,2}(\T))}\leq \overline{\varrho}, \quad
0<\underline{\vartheta}\leq\vartheta_0, \|\vartheta_0\|_{W^{2,2}(\T))}\leq \overline{\vartheta},
\]
as well as 
\begin{equation}
    \E\bigg[\|\vec u_0\|_{H_m}^r\bigg]\leq \overline{u}\quad \text{for some $r>2$,}
\end{equation}
$\p$-a.s.
Arguing similarly as in the {h}-limit (with obvious modifications):
We apply the Jakubowski's variant of Skorokhod representation theorem \cite{Jaku}, create new probability space with new sequence of random variables that are a.s convergent (\textit{w.l.o.g} we keep the same notation). Thus, passing  the limit $R\to \infty$ in (\ref{eq:Ba:lm})-(\ref{eq:Bc:lm}) yields

\begin{align}\label{eq:Baa:lm}
  \partial_t \varrho&+ \mathrm{div}(\varrho\vec u)  =0,\\
  \dd \Pi_m[\varrho\vec u]&+\Pi_m[\mathrm{div}(\varrho\vec u\otimes\vec u)]\,\dd t+\Pi_m\Big[\nabla\left(p(\varrho,\vartheta)\right)\Big]\label{eq:Bbb:lm}\\
  &=\Pi_m\Big[\varepsilon\mathcal{L}\vec u\Big]\,\dd t +\varrho\Pi_m[(\phi)]\,\dd W, \nonumber\\
  \label{eq:Bcc:lm}
  \partial_t S&+\mathrm{div}(S\vec u)=0.
\end{align}

\subsubsection{Galerkin Limit }
\textbf{Limit} $m\to \infty$. The approximate problem (\ref{eq:Baa:lm})-(\ref{eq:Bcc:lm}) admits a martingale solution $[\varrho_m, \vartheta_m, \vec u_{m}]$ with initial law $\Lambda$ for any fixed $m>0$. We proceed step by step as in the $R$-limit, that is,  following preceding parts, we establish uniform bounds (compactness) and perform the limit. In this case, the density estimate (\ref{eq:mxr}) and similarly the temperature estimate continue to hold for $m\to \infty$, and {we can weaken the regularity of initial data in (\ref{eq:Rpsc}) by considering a sequence of initial laws that lose regularity when $m\to \infty$}. Performing the limit $m\to \infty$ yields a martingale solution in to the system
\begin{align*}
  \dd \varrho&+ \mathrm{div}(\varrho\vec u) \,\dd t =0,\\
  \dd \varrho\vec u&+\mathrm{div}(\varrho\vec u]\otimes\vec u)\,\dd t+\nabla\left(p(\varrho,\vartheta)\right)\\
  &=\varepsilon\mathcal{L}\vec u\,\dd t +\varrho\phi_{\varepsilon}\,\dd W, \nonumber\\
  \dd S&+\mathrm{div}(S\vec u)\,\dd t=0.
\end{align*}
this completes the proof of Theorem \ref{ExthmA}.

\section{Existence results}\label{mvsproof}

The proof Theorem \ref{ExMainr} consists of establishing `a priori bounds' from the energy inequality, compactness arguments in space-time variables, and application of Jakubowski's version of Skorokhod representation theorem \cite{Jaku} to deal with Quasi-Polish spaces.

\begin{rmk}
For any $\varepsilon > 0$ Theorem \ref{ExthmA} yields the existence of martingale solution 
\[
((\Omega_{\varepsilon},\FF_{\varepsilon},(\FF_{t}^{\varepsilon}),\p_{\varepsilon}), \varrho_{\varepsilon},\vec m_{\varepsilon},S_{\varepsilon},W^{\varepsilon})
\]
to (\ref{eq:Euler}). We can assume without loss of generality that the probability space does not depend on $\varepsilon$ (see, e.g \cite{Jaku}), that is, the solution is given by
\[
((\Omega,\FF,(\FF_{t}^{\varepsilon}),\p), \varrho_{\varepsilon},\vec m_{\varepsilon},S_{\varepsilon},W^{\varepsilon}).
\]
\end{rmk}
We are now ready to consider the following proposition of `a priori bounds'.

 \begin{prop}\label{propE} Let $p \in [1,\infty) $. Then the functions $\varrho,\vec u$  and $s$ satisfy the following
\begin{gather}\label{priori}
    \E\left(\sup_{t \in (0,T)}\int_{\T}\bigg[ \frac{1}{2}\frac{|\vec m_{\varepsilon}|^2}{\varrho_{\varepsilon}}+ c_v\varrho_{\varepsilon}^{\gamma}\exp\left(\frac{ S_{\varepsilon}}{c_v\varrho_{\varepsilon}}\right) \bigg]\,\dd x + \varepsilon \int_{0}^{T}((\vec u_{\varepsilon};\vec u_{\varepsilon}))\, \dd x \dd t \right)^p\\
\leq C\left(1 +\E\bigg[ \int_{\T}\bigg( \frac{1}{2}\varrho_{0,\varepsilon} |\vec u_{0,\varepsilon}|^2 + c_v\varrho_{0,\varepsilon}^{\gamma}\exp\left(\frac{ S_{0,\varepsilon}}{c_v\varrho_{0,\varepsilon}}\right) \bigg)\,\dd x\bigg]^p \right)
\leq  C(T,\overline{\varrho},\phi, \Lambda), \nonumber
\end{gather}
uniformly in $\varepsilon$, where $\Lambda$ is the initial law.
\end{prop}

\begin{proof} 
 First, we observe that the energy formulation of the approximate system (\ref{eq:Euler}) is of the form 
\begin{gather*}
\hspace{-7cm}
    \int_{\T}\left[\frac{1}{2}\frac{|\vec m|^2}{\varrho}+ \varrho e\right]\, \dd x +\varepsilon\int_{0}^{T}((\vec u,\vec u))\, \dd t\\
    =\int_{\T}\left[\frac{1}{2}\frac{|\vec m|^2}{\varrho}+ \varrho_0 e_0\right]\, \dd x+
    \frac{1}{2}\sum_{k=1}^{\infty}\int_{0}^{T}\int_{\T}\varrho|\phi{e_k}|^{2}\, \dd x \dd t +\int_{0}^{T}\int_{\T}\varrho \phi_{\varepsilon}\cdot \vec u \,\dd x \dd W
\end{gather*}
For the estimate to hold, we  take the supremum in time first, and complete the proof by taking the expectations. Accordingly, splitting terms and proving them individual in separate steps yields:

\begin{itemize}
    \item firstly, considering the correction term we deduce
    \begin{equation*}
      \frac{1}{2}\sum_{k}\int_{0}^{T}\int_{\T}\varrho|\phi_{\varepsilon}{e_k}|^2 \dd x \dd t = \frac{1}{2}\sum_k\int_{0}^{T}\int_{\T}|\sqrt{\varrho}\phi_{\varepsilon}{e_k}|^2 \dd x \dd t \leq \frac{1}{2}\int_{0}^{T}\|\sqrt{\varrho}\phi\|_{L_2(\UU,L^2(\T))}^{2} \dd t < \infty.
    \end{equation*}
 The bound follows from the assumptions of $\phi$ in (\ref{eq:HS}) and using 
 \[
 \|\varrho\|_{L_x^1}=\|\varrho_0\|_{L_x^1}\leq \overline{\varrho}. 
 \]

\item Next, the noise term. Here we take supremum in time and build expectations. Furthermore, we use the Burgholder-Davis-Gundy inequality to obtain
\begin{eqnarray*}
\E \left( \sup_{t \in (0,T)}\left| \int_{0}^{t}\int_{\T} \varrho \phi\varepsilon \vec u \, \dd x\dd W \right|\right) &=&\E \left( \sup_{t \in (0,T)}\left|\sum_{k} \int_{0}^{t}\int_{\T} \varrho [\phi_\varepsilon]{e_k} \vec u \, \dd x\dd \beta_k \right|\right)\\
&\leq& c \E\left(\int_{0}^{T} \sum_k\left[ \int_{\T}\varrho [\phi_\varepsilon]{e_k} \vec u\, \dd x \right]^2 \, \dd t\right)^{1/2}\\
&\leq& c \E\left(\int_{0}^{T}\sum_k\|\sqrt{\varrho}\phi{e_k}\|_{L_2(\UU,L^2(\T))}^{2}\int_{\T}|\sqrt{\varrho} \vec u|^2 \, \dd x \dd t\right)^{1/2}\\
&\leq& c(\phi)\E \Bigg(\sup_{t \in (0,T)}\underbrace{\int_{\T}\varrho \dd x}_{\leq\, \overline{\varrho}} \int_{\T}\varrho |\vec u|^2\,\dd x \Bigg)^{1/2}\\
&\leq& \frac{\delta}{2}\E \left(\sup_{t \in (0,T)} \int_{\T}\varrho |\vec u|^2\,\dd x \right)+ c^2(\phi,\overline{\varrho},\delta),
\end{eqnarray*}
where the last line follows from Young's inequality. Now taking delta $\delta$ small enough, we can absorb the supremum term from the right.\newline

\item We note that expectation on initial data is bounded by assumption i.e.
\[
\E\bigg[ \int_{\T}\bigg( \frac{1}{2}\varrho_0 |\vec u_0|^2 + c_v\varrho_{0}^{\gamma}\exp\left(\frac{\vec S_0}{c_v\varrho_0}\right) \bigg)\,\dd x\bigg]^p < \infty.
\]

Hence combining the correction and stochastic terms we deduce (\ref{priori}).
\end{itemize}
\end{proof}

In view of Proposition \ref{propE}, we establish the following bounds:\newline

Firstly, we consider $Z\in BC(\R)$ such that
\[
Z' \geq 0, Z(s)
\begin{cases}
< 0 \quad \text{for}\, s < s_0,\\
=0 \quad \text{for}\, s \geq s_0,
\end{cases}
\]
then the total entropy in (\ref{eq:Euler}) satisfies the minimum principle provided that
    \[
     S_0 \geq \varrho_0 \vec s_0 > -\infty\, \, \text{a.a in }\, \T,
    \]
 see \cite{FEJB} for details. Since the entropy is bounded below, using (\ref{priori}) we deduce

\begin{equation}\label{BA}
    \E \left(\sup_{t \in [0,T]}\int_{\T} \varrho^{\gamma}\,\dd x \right) \lesssim C(T,\overline{\varrho},\phi,\Lambda),
\end{equation}
for any $t \in [0,T]$. Now using $\vec m = \varrho\vec u$, we observe
\[
|\vec m|^{\frac{2\gamma}{\gamma+1}} =|\varrho|^{\frac{\gamma}{\gamma+1}}\left|\frac{\vec m}{\sqrt{\varrho}}\right|^{\frac{2\gamma}{\gamma+1}} \lesssim \varrho^{\gamma}+\frac{|\vec m|^2}{\varrho},
\]
 and we obtain

\begin{equation}\label{BB}
    \E \left( \sup_{t\in [0,T]}\int_{\T}|\vec m|^{\frac{2\gamma}{\gamma+1}} \, \dd x\right)\lesssim C(T,\overline{\varrho},\phi,\Lambda),
\end{equation}
for any $t \in [0,T]$. Bounds on the total entropy $S$. Since $ S \geq s_0\varrho$, for $ S \leq  0$
\[
| S|=- S \leq - s_0\varrho .
\]

If $ S \geq 0$, we note

\[
\varrho^{\gamma} \exp{\left(\frac{ S}{c_v \varrho}\right)} = c_{v}^{-\gamma} \frac{\exp{\left(\frac{ S}{c_v \varrho}\right)}}{(\frac{S}{c_v \varrho})^{\gamma}} S^{\gamma} \gtrsim  S^{\gamma};
\]
hence 

\begin{equation}\label{BC}
 \E\left(\sup_{t\in [0,T]}\int_{\T}| S|^{\gamma}\, \dd x \right)   \lesssim C(T,\overline{\varrho},\phi,\Lambda),
\end{equation}
for any $t \in [0,T]$.
Finally, we derive an estimate for the quantity $S/\sqrt{\varrho}$. For $ S \leq 0$, we obtain 

\[
\left|\frac{ S}{\sqrt{\varrho}}\right| \leq - s_0\sqrt{\varrho}.
\]
If $ S> 0$, we have

\[
\varrho^{\gamma}\exp{\left(\frac{ S}{c_v\varrho}\right)} = \varrho^{\gamma}\exp{\left(\frac{ S}{\sqrt{\varrho}}\frac{1}{c_v\sqrt{\varrho}}\right)} =c_{v}^{-2\gamma}\frac{\exp{\left(\frac{S}{\sqrt{\varrho}}\frac{1}{c_v\sqrt{\varrho}}\right)}}{\left(\frac{ S}{\sqrt{\varrho}}\frac{1}{c_v\sqrt{\varrho}}\right)^{2\gamma}}\left(\frac{ S}{\sqrt{\varrho}}\right)^{2\gamma} \gtrsim \left(\frac{ S}{\sqrt{\varrho}}\right)^{2\gamma},
\]
and in view of this result we deduce the bound
\begin{equation}\label{BD}
    \E\left(\sup_{t\in [0,T]}\int \left|\frac{ S}{\sqrt{\varrho}}\right|^{2\gamma}\, \dd x\right) \lesssim C(T,\overline{\varrho},\phi,\Lambda).
\end{equation}
In view of the above bounds (\ref{BA})-(\ref{BD}) and the energy inequality (\ref{energy}), we deduce the following (uniform) bounds

\begin{eqnarray}
    \sqrt{\varepsilon}\vec u_{\varepsilon} &\in & L^p(\Omega;L^2([0,T];W^{3,2}(\T)))\\
    \varrho_{\varepsilon}&\in& L^p(\Omega;L^{\infty}([0,T];L^{\gamma}(\T))), \label{A}\\
    \vec m_{\varepsilon} &\in& L^p(\Omega;L^{\infty}([0,T];L^{\frac{2\gamma}{\gamma +1}}(\T))),\\
    \frac{\vec m_{\varepsilon}}{\sqrt{\varrho_{\varepsilon}}}&\in& L^p(\Omega;L^{\infty}([0,T];L^{2}(\T))),\\
    \frac{\vec m_{\varepsilon} \otimes \vec m_{\varepsilon}}{\varrho}&\in& L^p(\Omega;L^{\infty}([0,T];L^{1}(\T))),\\
    S_{\varepsilon}&\in& L^p(\Omega;L^{\infty}([0,T];L^{\gamma}(\T))),\\
    \frac{ S_{\varepsilon}}{\sqrt{\varrho_{\varepsilon}}}&\in& L^p(\Omega;L^{\infty}([0,T];L^{2\gamma}(\T))),
\end{eqnarray}

\subsection{A priori estimates}
Next, we seek to pass to the limit in the nonlinear convective term and this procedure requires compactness arguments. The balance of momentum is given by

\begin{eqnarray*}
\int_{\T}\varrho\vec{u}_{\varepsilon}\cdot \varphi\, \dd x &=& \int_{\T}\varrho_0\vec{u}_0 \cdot \varphi \, \dd x + \int_{0}^{t}\int_{\T}{\varrho_{\varepsilon}\vec{u}_{\varepsilon}\otimes \vec{u}_{\varepsilon}}: \nabla \varphi\, \dd x \dd s \\
&&-\varepsilon \int_{0}^{t}\int_{\T} \nabla \Delta  \vec{u}_{\varepsilon}\cdot\nabla \Delta \varphi \, \dd x \dd s -\varepsilon\int_{0}^{t}\vec u_{\varepsilon}\varphi\,\dd x\dd s\\
&&+\int_{0}^{t}\int_{\T}\varrho_{\varepsilon}^{\gamma}\exp{\left(\frac{ S_{\varepsilon}}{c_v\varrho_{\varepsilon}}\right)}\cdot \mathrm{div} \varphi \, \dd x \dd t + \int_{\T}\int_{0}^{t}\varrho_{\varepsilon} \phi_\varepsilon dW_s \cdot \varphi \dd x,
\end{eqnarray*}
 for all $\varphi \in C^{\infty}(\T)$. We show boundedness of the system by considering deterministic and stochastic parts separately. For the deterministic case, we consider the functional
 
\[
\mathcal{H}_{\varepsilon}(t, \varphi):=
\int_{0}^{t}\int_{\T}{\varrho_{\varepsilon}\vec{u}_{\varepsilon}\otimes \vec{u}_{\varepsilon}}: \nabla \varphi\, \dd x \dd s 
-\varepsilon \int_{0}^{t}((\vec u_{\varepsilon};\varphi)) \dd s 
+\int_{0}^{t}\int_{\T}\varrho_{\varepsilon}^{\gamma}\exp{\left(\frac{ S_{\varepsilon}}{c_v\varrho_{\varepsilon}}\right)}\cdot \mathrm{div} \varphi \, \dd x \dd s.
\]

We observe that 
\begin{equation*}
\begin{cases}
\partial_t\mathcal{H}_{\varepsilon}(t,\varphi)   \in L^1(\Omega; L^2(0,T;W^{-3,2}(\T)),\\ 
\mathcal{H}_{\varepsilon}(t,\varphi) \in L^1(\Omega; W^{1,2}(0,T;W^{-3,2}(\T)),
\end{cases}
\end{equation*}
uniformly in $\varepsilon$. Then we deduce the estimate

\begin{equation*}
\mathbb{E}\left[\|\mathcal{H}_{\varepsilon}\|_{W^{1,2}([0,T];W_{\text{div}}^{-3,2}(\T))}\right] \leq C.
\end{equation*}
 
 The stochastic term yields 

\[
\E \left[ \left \|\int_{0}^{\cdot} \varrho_{\varepsilon} \phi_\varepsilon \,\dd W^{\varepsilon}\right\|_{C^{\alpha}([0,T];L^2(\T))}^{p}\right] \leq c \E \left[ \int_{0}^{T}\|\sqrt{\varrho} \phi\|^{p}_{L_2(\UU,L^2(\T))} \, \dd t \right] \leq c(\overline{\varrho},p,\phi,T),
\]
 for all $\alpha \in (1/p,1/2)$ and $p>2$, see [\cite{Brt_1}, Lemma 9.1.3. b)] or [\cite{Hofm}, Lemma 4.6]). Now combining the deterministic and stochastic estimates, and using the embedding $W^{1,2}_t \hookrightarrow C^{\alpha}_{t}$ and $L^2_x \hookrightarrow W^{-3,2}_x$ shows 
 
 \begin{equation}\label{estimate}
\E \left[\|\varrho_{\varepsilon}\vec{u}_{\varepsilon}\|_{C^{\alpha}([0,T];W^{-3,2}(\T))} \right] \leq c(T),
 \end{equation}
 
for all $\alpha < 1/2$. On the regularity of mass continuity we have
\begin{equation*}
\int_{0}^{T}\int_{\T} \partial_t \rho \varphi\,\dd x  =\int_{0}^{T}\int_{\T} [ \varrho u\nabla_x\varphi] \, \dd x \dd t,
\end{equation*}
so that
\[
\sup_t\|\partial_t \varrho \|_{W^{-3,2}} 
\leq C \sup_t \|\varrho_{\varepsilon} u_{\varepsilon}\|_1.
\]
Hence, $\partial_t \varrho_{\varepsilon} \in L^{\infty}(0,T;W^{-3,2}(\T))$ a.s. and in view of (\ref{priori}) we obtain

\[
\E\left[\| \varrho_{\varepsilon}\|_{C^{\alpha}([0,T];W^{-3,2}(\T))}\right] \leq C.
\]

 Similarly, for the entropy balance we have
\begin{equation*}
\int_{0}^{T}\int_{\T} \partial_t  S \varphi\,\dd x  =\int_{0}^{T}\int_{\T} \left[ S\frac{\vec m}{\varrho} \nabla_x\varphi\right] \, \dd x \dd t,
\end{equation*}
and we deduce
\[
\E\left[\|  S_{\varepsilon}\|_{C^{\alpha}([0,T];W^{-3,2}(\T))}\right] \leq C.
\]

\subsection{Compactness Argument}

Our goal here is to show tightness of the approximate solutions using the following compact embeddings. 

\begin{equation}\label{eqnA}
C^{\alpha}([0,T];W^{-3,2}(\T)) \cap L^{\infty}(L^{\frac{2\gamma}{\gamma +1}}(\T)) \hookrightarrow \hookrightarrow C([0,T];W^{-4,2}(\T))\cap C_w (L^{\frac{2\gamma}{\gamma +1}}(\T)).
\end{equation}

\begin{equation}\label{eqnB}
 C^{\alpha}([0,T];W^{-3,\frac{2\gamma}{\gamma+1}})\cap L^{\infty}(0,T;L^{\gamma}(\T)) \hookrightarrow\hookrightarrow  C([0,T];W^{-4,2}(\T))\cap C_w(0,T; L^{\gamma}(\T)). 
\end{equation}

We set the spaces:

\begin{eqnarray*}
\mathscr{X}_{\varrho_0}: = L^{\gamma}(\T), & \quad& \mathscr{X}_{\vec m_0} = L^{\frac{2\gamma}{\gamma+1}}(\T), \\
\mathscr{X}_{\Vec{m}}: = C([0,T];W^{-4,2}(\T))\cap C_w (L^{\frac{2\gamma}{\gamma +1}}(\T)),&\quad& \mathscr{X}_{W} :=C([0,T];\UU_0),\\
\mathscr{X}_{\varrho}:=C([0,T];W^{-4,2}(\T))\cap C_w([0,T]; L^{\gamma}(\T)), &\quad& \mathscr{X}_{\text{C}}:=L^{\infty}(0,T;\mathcal{M}^{+}(\T,\R^{3\times3})),\\
\mathscr{X}_{\text{P}}:=L^{\infty}(0,T;\mathcal{M}^{+}(\T,\R), &\qquad& \mathscr{X}_{\vec U}:=L^2(0,T;W^{3,2}(\T))),\\
\mathscr{X}_{ S}:=C([0,T];W^{-4,2})\cap C_w([0,T]; L^{\gamma}(\T)), &\qquad& \mathscr{X}_{ S_0}: = L^{\gamma}(\T), \\
\mathscr{X}_{\text{Q}}:=L_{w*}^{\infty}(Q;\mathcal{P}(A)),\\
\end{eqnarray*}
 with respect to weak-* topology for all spaces with $L^{\infty}(\cdot,\mathcal{M}^{\cdot}(\cdot))$. Furthermore, for $T>0$, we choose the product path space

\begin{equation}
  \mathfrak{X}_{T}:=\mathscr{X}_{\varrho_0} \times \mathscr{X}_{\vec m_0}\times \mathscr{X}_{\vec S_0} \times\mathscr{X}_{\varrho }\times \mathscr{X}_{ \Vec{m}} \times\mathscr{X}_{ S} \times\mathscr{X}_{\text{prss}} \times\mathscr{X}_{\text{conv}}\times \mathscr{X}_{W},
\end{equation}

with the following laws:
\begin{equation*}
\begin{cases}
\mu_{{(\varrho \vec{u})}_{\varepsilon}} \ \text{ is the law of }\ \varrho_{\varepsilon} \vec{u}_{\varepsilon} \ \text{on} \   C([0,T];W^{-4,2}(\T))\cap C_w (L^{\frac{2\gamma}{\gamma +1}}(\T)),\\
\mu_{\varrho_{\varepsilon}} \ \text{ is the law of }\ \varrho_{\varepsilon} \ \text{on} \  C([0,T];W^{-4,2})\cap C_w(0,T; L^{\gamma}(\T)),\\
\mu_{ S_{\varepsilon}} \ \text{ is the law of }\ S_{\varepsilon} \ \text{on} \  C([0,T];W^{-4,2})\cap C_w(0,T; L^{\gamma}(\T)),\\
\mu_W \ \text{ is the law of } \ W \ \text{on} \ C([0,T],\UU_0). \\ 
\end{cases}
\end{equation*}
In addition, let $\mu_{\vec U_{\varepsilon}},\mu_{\mathrm{C}_{\varepsilon}}, \mu_{\mathrm{P}_{\varepsilon}} $ and $\mu_{\mathrm{Q}_{\varepsilon}}$ denote the laws of
\[\vec U_{\varepsilon}:=\sqrt{\varepsilon}\vec u\qquad
\mathrm{C}_{\varepsilon}:=\varrho_{\varepsilon} \vec u_{\varepsilon}\otimes \vec u_{\varepsilon}, \qquad \mathrm{P}_{\varepsilon}:=\varrho_{\varepsilon}^{\gamma} \exp{\left(  \frac{ S_{\varepsilon}}{c_v \varrho_{\varepsilon}}\right)},\qquad Q_{\varepsilon}:= S_{\varepsilon} \frac{\vec m_{\varepsilon}}{\varrho_{\varepsilon}},
\]
respectively. Let $\bfr_T$ be the restriction operator which restricts measurable functions (or space-time distributions) defined on $(0,\infty)$ to $(0,T)$. We denote by  $\mathcal{L}_T[\varrho_0,\varrho_0\vec{u}_0, S_0,\varrho_{\varepsilon},\varrho_{\varepsilon}\vec{u}_{\varepsilon}, S_{\varepsilon},\vec U_{\varepsilon}, P_{\varepsilon}, C_{\varepsilon}, Q_{\varepsilon}, W]$ the probability law on $\mathfrak{X}_{T}$. Note, tightness on $\mathcal{L}_T[\varrho_0,\varrho_0\vec{u}_0, S_0,\varrho_{\varepsilon},\varrho_{\varepsilon}\vec{u}_{\varepsilon}, S_{\varepsilon},\vec U_{\varepsilon}, P_{\varepsilon}, C_{\varepsilon}, Q_{\varepsilon}, W]$  for  any $T>0$ implies tightness of $\mathcal{L}[\varrho_0,\varrho_0\vec{u}_0, S_0,\varrho_{\varepsilon},\varrho_{\varepsilon}\vec{u}_{\varepsilon}, S_{\varepsilon},\vec U_{\varepsilon}, P_{\varepsilon}, C_{\varepsilon}, Q_{\varepsilon}, W]$ on $\mathfrak{X}=\cap_T\mathfrak{X}_T$. For $\varrho \vec{u}$, we fix $T>0$ and consider the ball $B_{R_1}$ in the space

\[
C^{\alpha}([0,T];W^{-3,2}(\T))\cap L^{\infty} (L^{\frac{2\gamma}{\gamma +1}}(\T)).
\]

In view of (\ref{priori}), (\ref{estimate}), and using Markov inequality for the complement $B_{R_1}^{C}$ we deduce that

\begin{eqnarray*}
\mu_{(\varrho \vec u )_{\varepsilon}}(\BB_{R_1}^{C})&=&\p \left(\|\varrho \vec{u}_{\varepsilon}\|_{C^{\alpha}([0,T];W^{-3,2}(\T))}  + \|\varrho \vec{u}_{\varepsilon}\|_{L^{\infty} (L^{\frac{2\gamma}{\gamma +1}}(\T))}\geq R\right)\\
&\leq&\frac{\E}{R_1}\left(\|\varrho \vec{u}_{\varepsilon}\|_{C^{\alpha}([0,T];W^{-3,2}(\T))}  + \|\varrho \vec{u}_{\varepsilon}\|_{L^{\infty} (L^{\frac{2\gamma}{\gamma +1}}(\T))}\right)\\
&\leq& \frac{C}{R_1}.
\end{eqnarray*}

Therefore, for a fixed $\eta>0$ we find $R_1(\eta)$ with 
\[
\mu_{(\varrho \vec{u})_{\varepsilon}}(\BB_{R_1}) \geq 1 -\eta.
\]
Hence, the law $\mu_{(\varrho \vec{u})_{\varepsilon}}$ is tight. Using similar arguments as shown above, we deduce that the laws:  $[\mu_{\varrho_\varepsilon},\mu_{ S_\varepsilon},\mu_{\vec U_\varepsilon}] $ are tight.

\begin{prop}
The law $\mu_{C_\varepsilon}$ is tight.
\end{prop}
\begin{proof}
 We consider a ball $B_{R} \in L^{\infty}(0,T;\mathcal{M}^{+}(\T,\R^{3\times3}))$ that is relatively compact with respect to weak-* topology. Now taking the complement ( $B_{R}^{c}$) of the ball and using Markov-inequality we deduce
 \begin{eqnarray*}
 \mathcal{L}[C_{\varepsilon}]( B_{R}^{c}) &=& \p\left( \int_{0}^{T}\int_{\T}\,\dd |C_{\varepsilon}|\dd t > R\right)\\
 &=&\p\left( \int_{0}^{T}\int_{\T}\left|\frac{\vec m_{\varepsilon} \otimes \vec m_{\varepsilon}}{\varrho_\varepsilon}\right|\,\dd x \dd t > R \right)\\
 &\leq& \frac{1}{R}\E\left\|\frac{\vec m_{\varepsilon} \otimes \vec m_{\varepsilon}}{\varrho_\varepsilon}\right\|_{L^{\infty}(0,T;L^1(\T))}\leq \frac{C}{R},
 \end{eqnarray*}
 where the last line follows from Proposition \ref{propE}. Therefore, for a fixed $\eta>0$ we find $R(\eta)$ with 
\[
\mathcal{L}[C_{\varepsilon}]( B_{R}) \geq 1 -\eta.
\]
The proof is complete.
\end{proof}
Similarly, arguing as shown above, the laws: 
$\mu_{\mathrm{P}_{\varepsilon}}$ and $\mu_{\mathrm{Q}_{\varepsilon}}$ are tight.
The laws $\mu_{W}, \mu_{\varrho_0}, \mu_{\varrho_0\vec u_0}$ and $\mu_{ S_0}$ are tight since they are Radon measures on the Polish spaces. Therefore, we can infer that the law $\mathcal{L}_{T}[\varrho_0,\varrho_0\vec{u}_0,S_0,\varrho_{\varepsilon},\varrho_{\varepsilon}\vec{u}_{\varepsilon}, S_{\varepsilon},\vec U_{\varepsilon},P_{\varepsilon},C_{\varepsilon},Q_{\varepsilon}, W]$ is a sequence of tight measures on $(\mathfrak{X}_{T})$.{ Consequently, its weak-* limit is tight as well and hence a Radon measure}. Since $T$ was arbitrary chosen we deduce that $\mathcal{L}[\varrho_0,\varrho_0\vec{u}_0,S_0,\varrho_{\varepsilon},\varrho_{\varepsilon}\vec{u}_{\varepsilon}, S_{\varepsilon},\vec U_{\varepsilon},P_{\varepsilon},C_{\varepsilon},Q_{\varepsilon}, W]$ is tight on $\mathfrak{X}$.  In view of the Jakubowksi's version of Skorokhod representation theorem \cite{Jaku} (see also Brzezniak et al.\cite{Brzez}), we have the following proposition.

\begin{prop} \label{skorokhod}
There exists a nullsequence $(\varepsilon_m)_{m \in \N}$, a complete probability space  $(\Tilde{\Omega}, \Tilde{\FF},\Tilde{\p})$ with $(\mathfrak{X},\mathscr{B}_{\mathfrak{X}})$-valued random variables \\ $(\tilde{\varrho}_{0\varepsilon_m},\tilde{\varrho}_{0,\varepsilon_m}\tilde{\vec{u}}_{0,\varepsilon_m},\tilde{ S}_{0,\varepsilon_m},\tilde{\varrho}_{\varepsilon_m},\tilde{\varrho}_{\varepsilon_m}\tilde{\vec{u}}_{\varepsilon_m},\tilde{ S}_{\varepsilon_m},\tilde{\vec U}_{\varepsilon_m}, \tilde{P}_{\varepsilon_m}, \tilde{C}_{\varepsilon_m},\tilde{Q}_{\varepsilon_m}, \tilde{W}_{\varepsilon_m})$, $m \in \N$,\\ and $(\tilde{\varrho}_0,\tilde{\varrho}_0\tilde{\vec{u}}_0,\tilde{S}_0,\tilde{\varrho},\tilde{\varrho}\tilde{\vec{u}}, \tilde{ S},\tilde{\vec U},\tilde{P},\tilde{C},\tilde{Q}, \tilde{W})$ such that
\begin{itemize}
    \item [(a)] For all $m \in \N$ the law of $(\tilde{\varrho}_{0\varepsilon_m},\tilde{\varrho}_{0,\varepsilon_m}\tilde{\vec{u}}_{0,\varepsilon_m},\tilde{ S}_{0,\varepsilon_m},\tilde{\varrho}_{\varepsilon_m},\tilde{\varrho}_{\varepsilon_m}\tilde{\vec{u}}_{\varepsilon_m},\tilde{ S}_{\varepsilon_m},\tilde{\vec U}_{\varepsilon_m}, \tilde{P}_{\varepsilon_m}, \tilde{C}_{\varepsilon_m},\tilde{Q}_{\varepsilon_m}, \tilde{W}_{\varepsilon_m})$ on $\mathfrak{X}$ is given by (coincides with) $\mathcal{L}[\varrho_0,\varrho_0\vec{u}_0, S_0,\varrho_{\varepsilon},\varrho_{\varepsilon}\vec{u}_{\varepsilon},S_{\varepsilon},\vec U_{\varepsilon},P_{\varepsilon},C_{\varepsilon},Q_{\varepsilon}, W];$ 
    \item[(b)] The law of $(\tilde{\varrho}_0,\tilde{\varrho}_0\tilde{\vec{u}}_0,\tilde{S}_0,\tilde{\varrho},\tilde{\varrho}\tilde{\vec{u}}, \tilde{ S},\tilde{\vec U},\tilde{P},\tilde{C},\tilde{Q}, \tilde{W})$ is a Radon measure on $(\mathfrak{X},\mathscr{B}_{\mathfrak{X}});$
    \item[(c)] $(\tilde{\varrho}_{0\varepsilon_m},\tilde{\varrho}_{0,\varepsilon_m}\tilde{\vec{u}}_{0,\varepsilon_m},\tilde{ S}_{0,\varepsilon_m},\tilde{\varrho}_{\varepsilon_m},\tilde{\varrho}_{\varepsilon_m}\tilde{\vec{u}}_{\varepsilon_m},\tilde{ S}_{\varepsilon_m},\tilde{\vec U}_{\varepsilon_m}, \tilde{P}_{\varepsilon_m}, \tilde{C}_{\varepsilon_m},\tilde{Q}_{\varepsilon_m}, \tilde{W}_{\varepsilon_m})$, $m \in \N$, converges $\tilde{\p}$-almost surely to  $(\tilde{\varrho}_0,\tilde{\varrho}_0\tilde{\vec{u}}_0,\tilde{S}_0,\tilde{\varrho},\tilde{\varrho}\tilde{\vec{u}}, \tilde{ S},\tilde{\vec U},\tilde{P},\tilde{C},\tilde{Q}, \tilde{W})$ in the topology of $\mathfrak{X}$, i.e.
    \begin{equation}
        \begin{cases}
        \tilde{\varrho}_{0,\varepsilon_m} \to \tilde{\varrho}_0 \,\, \text{in}\, L^{\gamma}(\T), \\
        \tilde{\varrho}_{0,\varepsilon_m}\tilde{\vec{u}}_{0,\varepsilon_m} \to \tilde{\varrho}_{0}\tilde{\vec{u}}_{0} \, \, \text{in}\, \, L^{\frac{2\gamma}{\gamma+1}}(\T),\\
        \tilde{ S}_{0,\varepsilon_m} \to \tilde{S_0} \,\, \text{in }\,\, C([0,T];W^{-4,2})\cap C_w(0,T; L^{\gamma}(\T)),\\
        \tilde{\varrho}_{\varepsilon_m} \to \tilde{\varrho} \,\, \text{in }\,\, C([0,T];W^{-4,2})\cap C_w(0,T; L^{\gamma}(\T)),\\
        \tilde{\vec S}_{\varepsilon_m} \to \tilde{\vec S} \,\, \text{in }\,\, C([0,T];W^{-4,2})\cap C_w(0,T; L^{\gamma}(\T)),\\
        \tilde{\vec U}_{\varepsilon_m} \to \tilde{\vec 0} \,\, \text{in }\,\, L^2([0,T];W^{3,2}(\T))),\\
        \tilde{\varrho}_{\varepsilon_m}\tilde{\vec{u}}_{\varepsilon_m} \to \tilde{\varrho}\tilde{\vec{u}} \, \, \text{in}\, \, C([0,T];W^{-4,2}(\T))\cap C_w (L^{\frac{2\gamma}{\gamma +1}}(\T)),\\
        \tilde{P}_{\varepsilon_m} \to \overline{\tilde{P}} \,\, \text{in} \,\, L_{w^*}^{\infty}(0,T;\mathcal{M}^{+}(\T,\R),\\
        \tilde{C}_{\varepsilon_m} \to \overline{\tilde{C}} \,\, \text{in} \,\, L_{w^*}^{\infty}(0,T;\mathcal{M}{+}(\T,\R^{3\times 3}), \\
        \tilde{Q}_{\varepsilon_m} \to \overline{\tilde{Q}} \,\, \text{in} \,\, L_{w^*}^{\infty}(Q;\mathcal{P}(A)), \\
         \tilde{W}_{\varepsilon_m} \to \tilde{W} \,\, \text{in} \,\, C([0,T];\UU_0),
        \end{cases}
    \end{equation}
    $\tilde{\p}$-a.s.
\end{itemize}
\end{prop}

To guarantee adaptedness of random variables and to ensure that the stochastic integral continues to hold in the new probability space we introduce filtration for correct measurability. We simplify notation as follows, set
\[
\mathcal{X}_0: =\left[\tilde{\varrho}_0,\tilde{\rho}_0\tilde{\vec{u}}_0,\tilde{ S}_0\right], \mathcal{X}:=\left[\tilde{\varrho},\tilde{\varrho}\tilde{\vec{u}},\tilde{ S},\tilde{\vec U},\right].
\]
Let $\tilde{\FF}_t$ and $\tilde{\FF}_{t}^{\varepsilon_m}$ be the $\tilde{\p}$-augmented filtration of random variables $(\tilde{\varrho}_0,\tilde{\varrho}_0\tilde{\vec{u}}_0,\tilde{S}_0,\tilde{\varrho},\tilde{\varrho}\tilde{\vec{u}}, \tilde{ S},\tilde{\vec U},\tilde{P},\tilde{C}, \tilde{W})$ and 
$(\tilde{\varrho}_{0\varepsilon_m},\tilde{\varrho}_{0,\varepsilon_m}\tilde{\vec{u}}_{0,\varepsilon_m},\tilde{\vec S}_{0,\varepsilon_m},\tilde{\varrho}_{\varepsilon_m},\tilde{\varrho}_{\varepsilon_m}\tilde{\vec{u}}_{\varepsilon_m},\tilde{\vec S}_{\varepsilon_m},\tilde{\vec U}_{\varepsilon_m}, \tilde{P}_{\varepsilon_m}, \tilde{C}_{\varepsilon_m},\tilde{Q}_{\varepsilon_m}, \tilde{W}_{\varepsilon_m})_{m \in \N}$, respectively, i.e.

\begin{eqnarray*}
\tilde{\FF}_t &=& \sigma (\sigma(\mathcal{X}_0,\rr \mathcal{X},\rr\tilde{W}) \cup \sigma_t( \tilde{P},\tilde{C},\tilde{Q})\cup\{ \mathcal{N} \in \tilde{\FF};\tilde{\p}(\mathcal{N})=0\}), t\geq 0,\\
\tilde{\FF}_{t}^{\varepsilon_m}&=&\sigma(\sigma(\mathcal{X}_{0,\varepsilon_m},\rr\mathcal{X}_{\varepsilon_m},\rr\tilde{W}_{\varepsilon_m})\cup \sigma_t (\tilde{P}_{\varepsilon_m},\tilde{C}_{\varepsilon_m},\tilde{Q}_{\varepsilon_m},)
\cup\{ \mathcal{N} \in \tilde{\FF};\tilde{\p}(\mathcal{N})=0\}), t\geq 0.
\end{eqnarray*}

Here $\rr$ denotes the restriction operator to the interval $[0,t]$ on the path space and $\sigma_t$\footnote{The family of $\sigma$-fields $(\sigma_{t}[\vec V])_{t\geq 0}$ given as random distribution history of 
 \begin{equation*}
     \sigma_t[\vec V]:= \bigcap_{s>t}\sigma\left(\bigcup_{\varphi\in C_c^{\infty}(Q;\R^3)}\{\langle \vec V, \varphi \rangle <1 \}\cup \{N\in \FF, \p(N)=0\} \right)
 \end{equation*}
 is called the history of $\vec V$. In fact, any random distribution is adapted to its history, see\cite{FrBrHo} (Chap. 2.2).} denotes the history of a random distribution.

\subsection{The new probability space}

In this section we will use the elementary method from \cite{BZ} to show that the approximated equations hold in the new probability space. The essence of this elementary method is to identify the quadratic and cross variations corresponding to the martingale with limit Wiener process obtained via compactness. Now in view of proposition {\ref{skorokhod}}, we note that $\tilde{W}$ has the same law as $W$. And as result, there exist a collection of independent real-valued $(\tilde{\FF}_t)_{t\geq 0}$ - Wiener process $\tilde{\beta}_{k}^{\varepsilon_m}$ such that $\tilde{W}^N = \sum_{k}\tilde{\beta}_{k}^{\varepsilon_m}e_k.$  To be specific, there exist a collection of independent real-valued $(\tilde{\FF}_t)_{t\geq 0}$ - Wiener process $\tilde{\beta}_{k}$ such that $\tilde{W} = \sum_{k}\tilde{\beta}_{k}e_k$. For all $t \in [0,T]$ and $\varphi \in C_{c}^{\infty}(\T) $ define the functionals:

\begin{eqnarray*}
\mathscr{M}^{\varepsilon_m}(\varrho_{0},\vec m_{0}, \varrho,\vec m,\vec U, C,P)_t &=&\int_{\T}(\vec{m} -\vec{m}_{0}) \cdot \varphi \, \dd x - \int_{0}^{t}\int_{\T}\nabla \varphi\, \dd C \dd s\\
&&  \sqrt{\varepsilon_m} \int_{0}^{t}\int_{\T} \nabla \Delta  \vec{U}\cdot\nabla \Delta \varphi \, \dd x \dd s -\sqrt{\varepsilon_m}\int_{0}^{t}\vec U\varphi\,\dd x\dd s\\
&&-\int_{0}^{t}\int_{\T} \mathrm{div} \varphi \, \dd P \dd s,\\
\end{eqnarray*}
\[
\Psi_t =\sum_{k=1}\int_{0}^{t}\left(\int_{\T}\varrho\phi_{e_k}\cdot \boldsymbol{\varphi} \ \text{d}x\right)^2 \ \text{d}s,
\]

\[
(\Psi_k)_t =\int_{0}^{t}\int_{\T}\varrho\phi_{e_k}\cdot \boldsymbol{\varphi} \ \text{d}x \text{d}s.
\]

Now, let $\mathscr{M}^{\varepsilon_m}(\tilde{\varrho}_{0,\varepsilon_m},\tilde{\vec m}_{0,\varepsilon_m}, \tilde{\varrho}_{\varepsilon_m},\tilde{\vec m}_{\varepsilon_m},\tilde{\vec U}_{\varepsilon_m}, \tilde{C}_{\varepsilon_m},\tilde{P}_{\varepsilon_m})_{s,t}$ denote the increment\\ $\mathscr{M}^{\varepsilon_m}(\tilde{\varrho}_{0,\varepsilon_m},\tilde{\vec m}_{0,\varepsilon_m}, \tilde{\varrho}_{\varepsilon_m},\tilde{\vec m}_{\varepsilon_m},\tilde{\vec U}_{\varepsilon_m}, \tilde{C}_{\varepsilon_m},\tilde{P}_{\varepsilon_m})_{t}-\mathscr{M}^{\varepsilon_m}(\tilde{\varrho}_{0,\varepsilon_m},\tilde{\vec m}_{0,\varepsilon_m}, \tilde{\varrho}_{\varepsilon_m},\tilde{\vec m}_{\varepsilon_m},\tilde{\vec U}_{\varepsilon_m}, \tilde{C}_{\varepsilon_m},\tilde{P}_{\varepsilon_m})_{s}$ and similarly for $\Psi_{s,t}$ and 
$(\Psi_k)_{s,t}$. In the new probability space, completeness of proof follows from showing that 

\begin{equation}
\mathscr{M}^{\varepsilon_m}(\tilde{\varrho}_{0,\varepsilon_m},\tilde{\vec m}_{0,\varepsilon_m}, \tilde{\varrho}_{\varepsilon_m},\tilde{\vec m}_{\varepsilon_m},\tilde{\vec U}_{\varepsilon_m}, \tilde{C}_{\varepsilon_m},\tilde{P}_{\varepsilon_m})_{t} = \int_{0}^{t}\int_{\T} \tilde{\varrho}_{\varepsilon_m} \phi \cdot \boldsymbol{\varphi} \ \mathrm{d}x\mathrm{d}\tilde{W}_{s}^{\varepsilon_m}.
\label{DS}
\end{equation}

For (\ref{DS}) to hold, it  suffices  to show that $\mathscr{M}^{\varepsilon_m}(\tilde{\varrho}_{0,\varepsilon_m},\tilde{\vec m}_{0,\varepsilon_m}, \tilde{\varrho}_{\varepsilon_m},\tilde{\vec m}_{\varepsilon_m},\tilde{\vec U}_{\varepsilon_m}, \tilde{C}_{\varepsilon_m},\tilde{P}_{\varepsilon_m})_{t}$ is an  $(\FF_t^{\varepsilon_m})_{t\geq 0}$-martingale process and its  corresponding quadratic and cross variations satisfy, respectively,

\begin{equation}\label{variations}
 \bigg \langle \bigg \langle \mathscr{M}^{\varepsilon_m}(\tilde{\varrho}_{0,\varepsilon_m},\tilde{\vec m}_{0,\varepsilon_m}, \tilde{\varrho}_{\varepsilon_m},\tilde{\vec m}_{\varepsilon_m},\tilde{\vec U}_{\varepsilon_m}, \tilde{C}_{\varepsilon_m},\tilde{P}_{\varepsilon_m}) \bigg \rangle \bigg\rangle =\Psi ,
\end{equation}

\begin{equation}\label{variations_1}
    \bigg\langle \bigg\langle \mathscr{M}^{\varepsilon_m}(\tilde{\varrho}_{0,\varepsilon_m},\tilde{\vec m}_{0,\varepsilon_m}, \tilde{\varrho}_{\varepsilon_m},\tilde{\vec m}_{\varepsilon_m},\tilde{\vec U}_{\varepsilon_m}, \tilde{C}_{\varepsilon_m},\tilde{P}_{\varepsilon_m}), \tilde{\beta}_k\bigg \rangle \bigg \rangle = \Psi_k,
\end{equation}

and consequently

\begin{equation}\label{cross_var}
\bigg \langle \bigg \langle \mathscr{M}^{\varepsilon_m}(\tilde{\varrho}_{0,\varepsilon_m},\tilde{\vec m}_{0,\varepsilon_m}, \tilde{\varrho}_{\varepsilon_m},\tilde{\vec m}_{\varepsilon_m},\tilde{\vec U}_{\varepsilon_m}, \tilde{C}_{\varepsilon_m},\tilde{P}_{\varepsilon_m}) -\int_{0}^{t}\int_{\T} \tilde{\varrho}_{\varepsilon_m}\phi \cdot \boldsymbol{\varphi} \ \mathrm{d}x\mathrm{d}\tilde{W}_{s}^{\varepsilon_m} \bigg \rangle \bigg\rangle = 0,    
\end{equation}
which implies the desired equation on the new probability space. We note that (\ref{variations}) and (\ref{variations_1}) hold based on the following observation: the mapping
\[
(\varrho_{0},{\vec m}_{0}, {\varrho},{\vec m},{\vec U}, {C},{P}) \mapsto \mathscr{M}(\varrho_{0},{\vec m}_{0}, {\varrho},{\vec m},{\vec U}, {C},{P})_t
\]
is {well-defined} and continuous on the path space. Using proposition \ref{skorokhod} we infer that 

\[
\mathscr{M}^{\varepsilon_m}(\varrho_{0,\varepsilon_m},{\vec m}_{0,\varepsilon_m}, {\varrho}_{\varepsilon_m},{\vec m}_{\varepsilon_m},{\vec U}_{\varepsilon_m}, {C}_{\varepsilon_m},{P}_{\varepsilon_m}) \sim^d \mathscr{M}^{\varepsilon_m}(\tilde{\varrho}_{0,\varepsilon_m},\tilde{\vec m}_{0,\varepsilon_m}, \tilde{\varrho}_{\varepsilon_m},\tilde{\vec m}_{\varepsilon_m},\tilde{\vec U}_{\varepsilon_m}, \tilde{C}_{\varepsilon_m},\tilde{P}_{\varepsilon_m}).
\]

Fixing times $s,t \in [0,T]$, with $s<t$ we consider a continuous function $h$ such that
\[
h:V|_{[0,s]} \to [0,1].
\]
The process
\begin{eqnarray*}
 \mathscr{M}^{\varepsilon_m}(\varrho_{0,\varepsilon_m},{\vec m}_{0,\varepsilon_m}, {\varrho}_{\varepsilon_m},{\vec m}_{\varepsilon_m},{\vec U}_{\varepsilon_m}, {C}_{\varepsilon_m},{P}_{\varepsilon_m})&=& \int_{0}^{t}\int_{\T}\varrho_{\varepsilon_m} \phi  \cdot \boldsymbol{\varphi} \ \text{d}x \text{d}W_{s}^{\varepsilon_m}\\
 &=& \sum_{k=1}\int_{0}^{t}\int_{\T}\varrho_{\varepsilon_m} \phi_{e_k}\cdot  \boldsymbol{\varphi}\ \text{d}x\ \text{d}\beta_k^{\varepsilon_m},
\end{eqnarray*}
is a square integrable $(\FF_t)_{t \geq 0}$-martingale, consequently, we infer

\[
[\mathscr{M}^{\varepsilon_m}(\varrho_{0,\varepsilon_m},{\vec m}_{0,\varepsilon_m}, {\varrho}_{\varepsilon_m},{\vec m}_{\varepsilon_m},{\vec U}_{\varepsilon_m}, {C}_{\varepsilon_m},{P}_{\varepsilon_m})]^2 - \Psi,\]
\[ \mathscr{M}^{\varepsilon_m}(\varrho_{0,\varepsilon_m},{\vec m}_{0,\varepsilon_m}, {\varrho}_{\varepsilon_m},{\vec m}_{\varepsilon_m},{\vec U}_{\varepsilon_m}, {C}_{\varepsilon_m},{P}_{\varepsilon_m})\beta_k -  \Psi_k,
\]
are $(\FF_t)_{t \geq 0}$-martingales. Now we set
\[
\vec X:=[\varrho_{0},{\vec m}_{0}, {\varrho},{\vec m},{\vec U}, {C},{P}], \qquad {\vec X}_{\varepsilon_m}: =[{\varrho}_{0,\varepsilon_m},{\vec m}_{0,\varepsilon_m}, {\varrho}_{\varepsilon_m},{\vec m}_{\varepsilon_m},{\vec U}_{\varepsilon_m}, {C}_{\varepsilon_m},{P}_{\varepsilon_m}],
\]
and
\[
\tilde{\vec X}:=[\tilde{\varrho}_{0},\tilde{{\vec m}}_{0}, \tilde{\varrho},\tilde{\vec m},\tilde{\vec U}, \tilde{C},\tilde{P}], \qquad \tilde{\vec X}_{\varepsilon_m}: =[\tilde{\varrho}_{0,\varepsilon_m},\tilde{\vec m}_{0,\varepsilon_m}, \tilde{\varrho}_{\varepsilon_m},\tilde{\vec m}_{\varepsilon_m},\tilde{\vec U}_{\varepsilon_m}, \tilde{C}_{\varepsilon_m},\tilde{P}_{\varepsilon_m}].
\]
 Let $\vec{r}_s$ be a restriction  function  to the interval $[0,s]$.   In view of proposition \ref{skorokhod} and the equality of laws we obtain:
 
 \begin{equation}\label{eq:hmA}
\tilde{\E}\bigg[h(\vec{r}_s\tilde{\vec{X}}_{\varepsilon_m},\vec{r}_s\tilde{W}^{\varepsilon_m})\mathscr{M}^{\varepsilon_m}(\tilde{\vec{X}}_{\varepsilon_m})_{s,t}={\E}\bigg[h(\vec{r}_s{\vec{X}}_{\varepsilon_m},\vec{r}_s{W}^{\varepsilon_m})\mathscr{M}^{\varepsilon_m}({\vec{X}}_{\varepsilon_m})_{s,t}\bigg]=0
\end{equation}

\begin{equation}\label{eq:hmB}
\tilde{\E}\bigg[h(\vec{r}_s\tilde{\vec{X}}_{\varepsilon_m},\vec{r}_s\tilde{W}^{\varepsilon_m})([\mathscr{M}^{\varepsilon_m}(\tilde{\vec{X}}_{\varepsilon_m})]^2 - \Psi)_{s,t}\bigg]={\E}\bigg[h(\vec{r}_s{\vec{X}}_{\varepsilon_m},\vec{r}_s{W}^{\varepsilon_m})([\mathscr{M}^{\varepsilon_m}({\vec{X}}_{\varepsilon_m})]^2 - \Psi)_{s,t}\bigg]=0
\end{equation}

\begin{equation}\label{eq:hmC}
\tilde{\E}\bigg[h(\vec{r}_s\tilde{\vec{X}}_{\varepsilon_m},\vec{r}_s\tilde{W}^{\varepsilon_m})(\mathscr{M}^{\varepsilon_m}(\tilde{\vec{X}}_{\varepsilon_m})\beta_k -  (\Psi_k))_{s,t})={\E}\bigg[h(\vec{r}_s{\vec{X}}_{\varepsilon_m},\vec{r}_s{W}^{\varepsilon_m})(\mathscr{M}^{\varepsilon_m}({\vec{X}}_{\varepsilon_m})\beta_k -  (\Psi_k))_{s,t})\bigg]=0
\end{equation}

Therefore, (\ref{variations}) and (\ref{variations_1}) hold, and consequently, (\ref{cross_var}) follows. Thus the momentum formulation:
\begin{eqnarray*}
\int_{\T}(\tilde{\vec{m}}_{\varepsilon_m} ) \cdot \varphi \, \dd x  &=&\int_{\T}(\tilde{\vec{m}}_{0,\varepsilon_m}) \cdot \varphi \, \dd x + \int_{0}^{t}\int_{\T}\nabla \varphi\, \dd \tilde{C}_{\varepsilon_m} \dd s\\
&&- \sqrt{\varepsilon_m} \int_{0}^{t}\int_{\T} \nabla \Delta  \tilde{\vec{U}}_{\varepsilon_m}\cdot\nabla \Delta \varphi \, \dd x \dd s -\sqrt{\varepsilon_m}\int_{0}^{t}\tilde{\vec U}_{\varepsilon_m}\varphi\,\dd x \dd s\\
&&
+\int_{0}^{t}\int_{\T}\mathrm{div} \varphi \, \dd \tilde{P}_{\varepsilon_m} \dd s
+\int_{\T}\int_{0}^{t}\tilde{\varrho}_{\varepsilon_m} \phi d\tilde{W}_{s}^{\varepsilon_m} \cdot \varphi \, \dd x,\\
\end{eqnarray*}

holds $\tilde{\p}$-a.s in new probability space $(\Tilde{\Omega}, \Tilde{\FF},\Tilde{\p})$. We note that, the terms in the continuity equation and entropy balance are continuous on the path-space and as such, both equations continue to hold on the new probability space $\tilde{\p}$-a.s as well.

\subsection{Passage to the limit}

To identify the limits in the nonlinear terms we first introduce defect measures. For this,  we adopt notion of measures as presented in \cite{DBrt}.  In view of Proposition \ref{skorokhod}  we have
\[
p(\tilde{\varrho}_{\varepsilon_m},\tilde{S}_{\varepsilon_m})\to \overline{p(\tilde{\varrho}, \tilde{S})}\,\,\text{weakly-(*) in} \,\, L^{\infty}(0,T;\mathcal{M}^{+}(\T,\R) .
\]
 Noting that $p(\tilde{\varrho}, \tilde{S}) = \tilde{\varrho}^{\gamma}\exp{\left(\frac{\tilde{\vec S}}{c_v\tilde{\varrho}}\right)}$ is a convex functional, we deduce
\[
0\leq p(\tilde{\varrho}, \tilde{S}) \leq \overline{p(\tilde{\varrho}, \tilde{S})},\quad \tilde{\mathcal{R}}_{\mathrm{press}}\equiv \overline{p(\tilde{\varrho}, \tilde{S})}- p(\tilde{\varrho},\tilde{S}) \in L^{\infty}(0,T;\mathcal{M}^{+}(\T,\R) .
\]

Arguing similarly for the convective term,
\[
  \frac{\tilde{\Vec{m}_{\varepsilon_m}} \otimes \tilde{\vec{m}}_{\varepsilon_m}}{\tilde{\varrho}_{\varepsilon_m}}\to \overline{\frac{\tilde{\Vec{m}} \otimes \tilde{\vec{m}}}{\tilde{\varrho}}} \,\,\text{weakly-(*) in} \,\, L^{\infty}(0,T;\mathcal{M}^{+}(\T,\R^{3\times3}),
 \]
setting
 \[
 \tilde{\mathcal{R}}_{\mathrm{conv}} \equiv \overline{\frac{\tilde{\Vec{m}} \otimes \tilde{\vec{m}}}{\tilde{\varrho}}} - \frac{\tilde{\Vec{m}} \otimes \tilde{\vec{m}}}{\tilde{\varrho}},
 \]
 for $\xi \in \R^3$, convexity implies  
 \begin{eqnarray*}
 \tilde{\mathcal{R}}_{\mathrm{conv}}:(\xi \otimes \xi)&=&\lim_{\varepsilon_m\to 0}\left[\frac{\tilde{\vec m}_{\varepsilon_m}\otimes \tilde{\vec m}_{\varepsilon_m}}{\rho_{\varepsilon_m}}:(\xi \otimes \xi)\right]- \frac{\tilde{\Vec{m}} \otimes \tilde{\vec{m}}}{\varrho}:(\xi \otimes \xi)\\
 &=&\lim_{\varepsilon_m \to 0}\left[\frac{|\tilde{\vec m}_{\varepsilon_m}\cdot \xi|^2}{\tilde{\varrho}_{\varepsilon_m}} -\frac{|\tilde{\vec m}\cdot \xi|^2}{\tilde{\varrho}} \right] \geq 0,
 \end{eqnarray*}
 so that $\tilde{\mathcal{R}}_{\mathrm{conv}} \in L^{\infty}(0,T;\mathcal{M}^{+}(\T,\R^{3\times3}))$.
To perform the stochastic limit term we use Lemma 2.1 in \cite{Debussche}. On the account of convergences in Proposition \ref{skorokhod}, Lemma 2.1 in \cite{Debussche} and the higher moments from (\ref{eq:hmA})-(\ref{eq:hmC})  we can pass to the limit $\varepsilon_m \to 0$ in the momentum equation in (\ref{eq:Euler}) and obtain

\begin{eqnarray}
\label{m_system}
\tilde{\E}\left[\int_{0}^{T}\int_{\T}\tilde{\vec{m}} \cdot \varphi \, \dd x \dd t \right] &=&\tilde{\E}\Bigg[\int_{0}^{T}\Bigg(\int_{\T}\tilde{\vec{m}}_{0}  \cdot \varphi \, \dd x  +
\int_{0}^{t}\int_{\T}\overline{\tilde{\frac{\Vec{m}_{} \otimes \vec{m}_{}}{\rho_{}}}} : \nabla \varphi\, \dd x \dd s \\
&&
+\int_{0}^{t}\int_{\T}\overline{\tilde{\varrho}^{\gamma}\exp{\left(\frac{\tilde{\vec S}}{c_v\tilde{\varrho}}\right)}} \cdot \mathrm{div} \varphi \, \dd x \dd s +\int_{\T}\int_{0}^{t}\tilde{\varrho} \phi d\tilde{W}_{s}  \cdot \varphi \,\dd x \Bigg) \dd t\Bigg].\nonumber
\end{eqnarray}

Consequently, the momentum equation in the sense of (\ref{eq:mcxs}) follows from rewriting (\ref{m_system}) using defect measures. Similarly, using Proposition \ref{skorokhod} we perform  $\varepsilon_m\to 0$ limit in the mass continuity and total entropy to deduce the equivalence of  (\ref{eq:cont}) and (\ref{eq:entr}) in the new probability space, respectively.

\subsubsection{On the Energy inequality}
Finally, we consider the energy equality. In the original probability space, the approximate system (\ref{eq:Euler}) has an energy equality of the form

\begin{eqnarray*}
E_{t}^{\varepsilon_m}= E_{s}^{\varepsilon_m}+\frac{1}{2}\int_{s}^{t}\|\sqrt{\varrho_{\varepsilon_m}}\phi\|_{L_2(\UU,L^2(\T))}^{2}\,\dd \sigma+ \int_{s}^{t}\int_{\T}\vec m_{\varepsilon_m}\phi \,\dd x\dd W^{\varepsilon_m},
\end{eqnarray*}
$\p$-a.s for a.a $0\leq s<t$, where

\[
E_{t}^{\varepsilon_m} =\int_{\T}\left[\frac{1}{2}\frac{|\vec m_{\varepsilon_m}|^2}{\varrho_{\varepsilon_m}}+ c_v\varrho_{\varepsilon_m}^{\gamma}\exp{\left(\frac{\vec S_{\varepsilon_m}}{c_v\varrho_{\varepsilon_m}}\right)} \right]\, \dd x +\varepsilon_m\int_{0}^{T}((\vec u_{\varepsilon_m},\vec u_{\varepsilon_m}))\, \dd t,
\]
for a.a $t \geq 0$. For any fixed $s$ this is equivalent to

\[
-\int_{s}^{\infty}\partial_t \varphi E_t^{\varepsilon_m} \, \dd t- \varphi(s)E_{s}^{\varepsilon_m} = \frac{1}{2}\int_{s}^{\infty}\varphi \|\sqrt{\varrho_{\varepsilon_m}}\phi\|_{L_2(\UU,L^2(\T))}^{2}\,\dd t+ \int_{s}^{\infty}\varphi\int_{\T}\vec m_{\varepsilon_m} \cdot \phi \,\dd x\dd W^{\varepsilon_m},
\]
$\p$-a.s for all $\varphi\in C_{0}^{\infty}([s,\infty))$. By virtue of Theorem 2.9.1 in \cite{DEH}, and in view of Proposition \ref{skorokhod} the energy equality continues to hold in the new probability space and reads

\[
\tilde{E}_{t}^{\varepsilon_m}=\tilde{E}_{s}^{\varepsilon_m}+\frac{1}{2}\int_{s}^{t}\|\sqrt{\tilde{\varrho}_{\varepsilon_m}}\phi\|_{L_2(\UU,L^2(\T))}^{2}\,\dd \sigma+ \int_{s}^{t}\int_{\T}\tilde{\vec m}_{\varepsilon_m}\phi \,\dd x\dd \tilde{W}^{\varepsilon_m},
\]
$\tilde{\p}$-a.s. for a.a $s$(including $s=0$) and all $t \geq s$. Averaging in $t$ and s, as in \cite{BrMo}, the above expression becomes continuous on the path space.  Furthermore, fixing $s=0$ and we use Lemma 2.1 in \cite {Debussche}, the bounds established in Proposition \ref{skorokhod}, and higher moments to perform the limit $\varepsilon_m\to 0$ and obtain

\begin{equation}\label{eq:newEner}
    \tilde{\mathbf{E}}_t\leq\tilde{\mathbf{E}}_0+\frac{1}{2}\int_{s}^{t}\|\sqrt{\tilde{\varrho}}\phi\|_{L_2(\UU;L^2(\T))}^2\, \dd \sigma + \int_{s}^{t}\int_{\T}\tilde{\vec m} \cdot \phi \, \dd x \dd \tilde{W},
\end{equation}
    $\p$-a.s. for a.a. $t\in [0,T]$, where
    \[
    \tilde{\mathbf{E}}_t= \int_{\T}\left[\frac{1}{2}\frac{|\tilde{\vec m} |^2}{\tilde{\varrho}}+c_v\tilde{\varrho}^{\gamma}\exp{\left(\frac{\tilde{S}}{c_v\tilde{\varrho}}\right)}\right]\, \dd x + \frac{1}{2}\int_{\T}\dd \mathrm{tr}\mathcal{R}_{\text{conv}}(t) +c_v\int_{\T}\dd \mathcal{R}_{\text{press}}(t),
    \]
     and 
    \[
   \tilde{E}_0= \int_{\T}\left[\frac{1}{2}\frac{|\tilde{\vec m}_0 |^2}{\tilde{\varrho}_0}+c_v\tilde{\varrho}_0^{\gamma}\exp{\left(\frac{S_0}{c_v\varrho_0}\right)}\right]\, \dd x.
    \]
Performing the limit $\varepsilon_m\to 0$ yields an energy inequality. Our goal now is to convert (\ref{eq:newEner}) to equality, for this, we argue as in \cite{HoFeBr}. The entropy balance in the approximate system (\ref{eq:Euler}) holds as an equality. Hence, to convert (\ref{eq:newEner}) to equality, it is sufficient to augment  the term contributing to the internal energy ($\mathcal{R}_{\text{press}}(t)$) by $h(t)\dd x$ with suitable spatially homogeneous $h\geq 0$. And $\mathcal{R}_{\text{press}}(t)$ acts on $\mathrm{div}_x\varphi$  in a periodic domain $\T$, therefore,

\[
\int_{\T}h(t)\mathrm{div}_x \varphi\,\dd x=0.
\]
 Finally, for any s we have
 \[
-\int_{s}^{\infty}\partial_t \varphi \tilde{E}_t \, \dd t- \varphi(s)\tilde{E}_{s} = \frac{1}{2}\int_{s}^{\infty}\varphi \|\sqrt{\varrho}\phi\|_{L_2(\UU,L^2(\T))}^{2}\,\dd t+ \int_{s}^{\infty}\varphi\int_{\T}\tilde{\vec m}\cdot \phi \,\dd x\dd W,
\]
$\p$-a.s for all $\varphi\in C_{0}^{\infty}([s,\infty))$.

\section{Weak-strong Uniqueness}\label{WK}

In this section we prove that a stochastic measure-valued martingale solution to (\ref{eq:aEuler})-(\ref{eq:cEuler}) coincides with a strong solution so long as the later exists. In order to do this, we need to introduce a \textit{relative entropy inequality}; a tool that allows us to compare two solutions. In the following analysis, it is more convenient to express the variable $S$ as $\varrho s(\varrho,E)$ where $E =\varrho e(\varrho,\vartheta)$ and to work with new state variables: the density $\varrho$, the momentum $\vec m$ and the internal energy $E$, see \cite{FEJB}.\newline

We begin by stating the following auxiliary results, a variant of (\cite{FrBrHo}, Thm. A.4.1), to which we refer to for more details.

\begin{lem}\label{prod}
Let $q$ be a stochastic process on $(\Omega,\FF,(\FF_t)_{t\geq 0},\p)$ such that for some $\alpha \in \R$,
\[
q \in C_{\rm weak}([0,T]; W^{-\alpha,p}(\T))\cap L^{\infty}(0,\infty; L^{1}(\T))\quad\p\text{-a.s},
\]
\begin{equation}\label{eq:RA}
\E\left[\sup_{t\in [0,T]}\|q\|_{L^1(\T)}^{p}\right]<\infty \quad \text{for all }\, 1\leq p <\infty,
\end{equation}
\begin{equation}\label{eq:RB}
   \dd q = D_{t}^{d}q\,\dd t + \mathbb{D}_t^s q \,\dd W, 
\end{equation}
where $D_t^d$, $\mathbb{D}_t^d$ are progressively measurable with
\begin{equation}\label{eq:RC}
    \begin{split}
    D_t^d q \in L^p(\Omega; L^2(0,T;W^{-\alpha,k}(\T))), \qquad \mathbb{D}_t q \in L^p(\Omega; L^2(0,T;L_2(\UU; W^{-m,2}(\T)))),\\
    \sum_{k\geq 1}\int_{0}^{T}\|\mathbb{D}_t^s q (e_k)\|_1^2\, \dd t\in L^p(\Omega)\quad 1\leq p <\infty,
    \end{split}
\end{equation}
for some $k>1$ and some $m \in \N$.\\
Let $w$ be a stochastic process on $(\Omega,\FF,(\FF_t)_{t\geq 0},\p)$ satisfying 
\[
w \in C([0,T]; W^{\alpha, k'}\cap C(\T)\quad\p\text{-a.s.},
\]
\begin{equation}\label{eq:RD}
\E \left[\sup_{t\in [0,T]}\|w\|_{W^{\alpha,k'}\cap C(\T)}^{p}\right]<\infty, \quad 1\leq p<\infty,
\end{equation}
\begin{equation}\label{eq:RE}
   dw =D_t^d w + \mathbb{D}_t^s w \dd W 
\end{equation}
where $D_t^d$, $\mathbb{D}_t^d$ are progressively measurable with
\begin{equation}\label{eq:RF}
    \begin{split}
    D_t^d w \in L^p(\Omega; L^2(0,T;W^{\alpha,k'}\cap C(\T)), \qquad \mathbb{D}_t w \in L^p(\Omega; L^2(0,T;L_2(\UU; W^{m,2}(\T)))),\\
    \sum_{k\geq 1}\int_{0}^{T}\|\mathbb{D}_t^s q (e_k)\|_{W^{\alpha,k'}\cap C(\T)}^{2} \, \dd t\in L^p(\Omega)\quad 1\leq p <\infty.
    \end{split}
\end{equation}

Let $Q$ be $[\alpha+2]$-continuously differentiable function satisfying
\begin{equation}\label{eq:RG}
 \E\left[\sup_{t\in [0,T]}\|Q^{(j)}(w)\|_{W^{\alpha,k'}\cap C(\T)}^{p} \right] <\infty \quad j=0,1,2, \quad 1\leq p<\infty. 
\end{equation}
Then 
\begin{eqnarray}\label{eq:RH}
\dd \left (\int_{\T}qQ(w)\,\dd x\right)&=& \left(\int_{\T}\left[q\left(Q'(w)D_t^dw+\frac{1}{2}\sum_{k\geq 1}Q''(w)|\mathbb{D}_t^s(e_k)|^2\right)\right]\, \dd x+ \bigg\langle Q(w),D_s^d q\bigg\rangle\right)\, \dd t\nonumber\\
&&+\left(\sum_{\geq 1}\int_{\T}\mathbb{D}_t^sq(e_k)\mathbb{D}_t^q w(e_k)\,\dd x \right)\dd t + \dd \mathbb{M},
\end{eqnarray}
where 
\begin{equation}\label{eq:RI}
 \mathbb{M}=    \sum_{k\geq 1}\int_{0}^{t}\int_{\T} [q Q'(w)\mathbb{D}_t^sw(e_k) +Q(r)\mathbb{D}_t^sq (e_k)]\,\dd x \dd W_k.
\end{equation}
\end{lem}
 Now following the presentation in \cite{EFAN}, we introduce the (thermodynamic potential) \textit{ballistic free energy}

\begin{equation}\label{ballistic}
 H_{{\Theta}}(\varrho,\vartheta)=\varrho e(\varrho,\vartheta)-\Theta \varrho s (\varrho,\vartheta),
\end{equation}
introduced by Gibbs and more recently by Erickson \cite{EJL}. In addition to Lemma \ref{prod}, 
we consider the \textit{relative energy} functional in the context of measure-valued martingale solutions to the complete Euler system given by
\begin{eqnarray}\label{relative}
\mathcal{E}\bigg(\varrho,E,\vec m\bigg|r,\Theta,\vec U\bigg) &=& \int_{\T}\left[\frac{1}{2}\frac{|\vec m|^2}{\varrho} +E \right]\, \dd x+\frac{1}{2}\int_{\T}\dd \mathrm{tr}[\mathcal{R}_{\text{conv}}]+c_v \int_{\T}\dd \mathcal{R}_{\text{press}}\nonumber\\
&&-\int_{\T}\vec m \cdot \vec U \, \dd x+\int_{\T}\frac{1}{2}\varrho\left|\vec U\right|^2 \, \dd x\\
&&-\int_{\T}\Theta\varrho {s}(\varrho,E)\, \dd x-\int_{\T}\varrho\partial_{\varrho}H_{\Theta}\,\dd x+\int_{\T}\partial_{\varrho}H_{\Theta}(r,\Theta)(r)-H_{\Theta}(r,\Theta)\, \dd x,\nonumber
\end{eqnarray}
where the relative functional (\ref{relative}) is defined for all $t \in [0,T]$. Now, having stated Lemma \ref{prod} and  the relative energy functional, we are in a position to derive the \textit{relative entropy inequality}.

\begin{prop}[Relative Entropy Inequality]\label{propA}
Let $
((\Omega , \FF , (\FF_t)_{t \geq 0}, \mathbb{P} ),\varrho,\vec m, S, \mathcal{R}_{\mathrm{conv}},\mathcal{V}_{t,x},\mathcal{R}_{\mathrm{entr}},W)
$ be a dissipative measure-valued martingale solution to the system (\ref{eq:aEuler})-(\ref{eq:cEuler}). Let $(r,\Theta, \vec U)$ be a trio of stochastic processes defined on the same probability space and adapted to the filtration $(\FF_t)_{t\geq 0}$ such that

\begin{eqnarray}\label{eq:sass}
    \dd r &=& D_t^dr\,\dd t,\nonumber\\
    \dd \vec U &=& D_t^d \vec U \,\dd t+ \mathbb{D}_t^s \vec U\, \dd W,\nonumber\\
    \dd \Theta &=& D_t^d \Theta\,\dd t,\\
    \dd [\partial_{\varrho}H_{\Theta}(r,\Theta)] &=& D_t^d [\partial_{\varrho}H_{\Theta}(r,\Theta)]\,\dd t,\nonumber
\end{eqnarray}

and\footnote{Note, the moment bound for $\Theta$ below implies the same for $S(r,\Theta)$ by (\ref{entropy}) since $r$ and $\Theta$ are bounded below and above.}
\begin{equation*}
\begin{split}
 r \in C([0,T]; C^{1} (\T)),\quad\Theta \in C([0,T]; C^{1} (\T)),\quad \vec U \in C([0,T]; C^{1}(\T)), \quad\p \text{-a.s.,}  \\
 \E\left[\sup_{t\in [0,T]}\|r\|_{W^{1,q}(\T)}^{2}\right]^k+\E\left[\sup_{t\in [0,T]}\|\vec U\|_{W^{1,q}(\T)}^{2}\right]^q \leq c(q),\quad \text{for all }\, 2\leq q < \infty,
\end{split}
\end{equation*}
\[
0<\underline{r} \leq r(t,x) \leq \overline{r} \quad \p\text{-a.s.},
\]

\begin{equation*}
\E\left[\sup_{t\in [0,T]}\|\Theta\|_{W^{1,q}(\T)}^{2}\right]^k \leq c(q),\quad \text{for all }\, 2\leq q < \infty,
\end{equation*}
\[
0<\underline{\Theta} \leq \Theta(t,x) \leq \overline{\Theta} \quad \p\text{-a.s.}.
\]
Furthermore, $r,\Theta, \vec U$, satisfy 
\[
D^dr,D^d \Theta,D^d\vec U \in L^{q}(\Omega; C(0,T;C^{1}(\T)))\qquad  \mathbb{D}^s\vec U \in L^2(\Omega; L^2(0,T;L_2(\UU;L^2(\T))),
\]
\begin{equation}\label{eq:prop}
 \left(\sum_{k\geq 1}|\mathbb{D}^s\vec U(e_k)|^q\right)^{\frac{1}{q}}   \in L^{q}(\Omega; L^q(0,T;L^q(\T))),
\end{equation}
respectively.
Then the relative entropy inequality:
\begin{equation}\label{REI}
   \mathcal{E}\bigg(\varrho,E,\vec m\bigg|r,\Theta,\vec U\bigg)\leq \mathcal{E}\bigg(\varrho,E,\vec m\bigg|r,\Theta,\vec U\bigg)(0)+\int_{0}^{\tau}\mathcal{Q}\bigg(\varrho,E,\vec m\bigg|r,\Theta,\vec U\bigg)\, \dd t+\mathbb{M},
\end{equation}
holds $\p$-a.s for all $\tau \in (0,T)$, where
\begin{align*}
\mathcal{Q}\bigg(\varrho,\vartheta,\vec u\bigg|r,\Theta,\vec U\bigg)=
&\int_{\T}\varrho\left(\frac{\vec m}{\varrho}-\vec U\right)\cdot\nabla_x\vec U\cdot \left(\vec U-\frac{\vec m}{\varrho}\right)\,\dd x\nonumber\\
&+\int_{\T}[(D_t^d \vec U + \vec U\cdot\nabla_x \vec U )\cdot(\varrho\vec U- \vec m) -p(\varrho,\vartheta)\mathrm{div}_x\vec U] \, \dd x  \\
&-\int_{0}^{\tau}\int_{\T}[\langle \CV; \varrho s(\varrho,E)\rangle D_t^d \Theta + \vec \langle \CV; s(\varrho,E)\vec m\rangle\cdot \nabla_x \Theta ] \, \dd x \dd t\nonumber\\
&+\int_{0}^{T}\int_{\T}[\varrho s (r,\Theta) \partial_{t}\Theta+ \vec m s(r,\Theta)\cdot\nabla_x\Theta]\, \dd x\dd t \nonumber\\
&+\int_{\T} \left(\left(1-\frac{\varrho}{r}\right)\partial_{t}p(r,\Theta)-\frac{\vec m}{r}\cdot\nabla_x p(r,\Theta)\right)\,\dd x.\nonumber,\\
&-\sum_{k\geq 1}\int_{\T}\mathbb{D}_t^s\vec U(e_k)\cdot \varrho\phi(e_k)\, \dd x - \int_{\T}\nabla \vec U:\dd \mathcal{R}_{\text{conv}}-\int_{\T}\mathrm{div} \vec U \dd \mathcal{R}_{press}\\
&+\frac{1}{2}\|\sqrt{{\varrho}}\phi\|_{L_2(\UU,L^2(\T))}^{2} +\frac{1}{2}\sum_{k\geq 1}\int_{\T}\varrho|\mathbb{D}_t^s\vec U(e_k)|^2\, \dd x, 
\end{align*}
and
\begin{eqnarray*}
\mathbb{M}&=& \int_{0}^{\tau}\int_{\T}{\vec m}\phi \,\dd x\dd {W} \\
&& -\int_{0}^{t}\int_{\T}\bigg[\vec m\mathbb{D}_t^s\vec U+ \vec U\varrho \phi\bigg]\,\dd x \dd W +\int_{0}^{t}\int_{\T}\varrho\vec U\cdot \mathbb{D}_t^s\vec U\, \dd x\dd W,
\end{eqnarray*}
here we set $Z(s(\varrho,E))=s(\varrho,E)$ for convenience, see \cite{FEJB} section 3.2 for properties of $Z$.
\end{prop}

\begin{proof}
We note that, the right-hand-side of the formulation (\ref{relative}) follows from energy inequality. Therefore, using the energy inequality and Lemma \ref{prod}, we proceed in several steps as follows:\newline

\textbf{Step 1}:
To compute $\dd \int_{\T} \vec m\cdot\vec U \,\dd x$ we recall that $q =\vec m$ satisfies hypotheses (\ref{eq:RA}) , (\ref{eq:RC}) with some $k< \infty$. Applying Lemma \ref{prod} we deduce
\begin{eqnarray}\label{eq:PA}
\dd \left( \int_{\T}\vec m\cdot\vec U \,\dd x \right)&=&\left(\int_{\T}\bigg[\vec m\cdot D_t^d\vec U+ \left(\frac{\vec m\otimes\vec m}{\varrho}\right)\cdot\nabla \vec U+ p(\varrho, s)\mathrm{div}\vec U\bigg]\,\dd x \right)\, \dd t\nonumber\\
&&+\sum_{k\geq 1}\int_{\T}\mathbb{D}_t^s\vec U(e_k)\cdot \varrho\phi(e_k)\, \dd x \dd t+ \int_{\T}\nabla \vec U:\dd \mathcal{R}_{\text{conv}}\dd t +\int_{\T}\rm{div}\vec U\dd \mathcal{R}_{\rm{press}}\dd t\\
&&+\dd M_1,\nonumber
\end{eqnarray}

where 
\[
M_1 =\int_{0}^{t}\int_{\T}\bigg[\vec m\mathbb{D}_t^s\vec U+ \vec U\varrho \phi\bigg]\,\dd x \dd W.
\]

Similarly to (\ref{eq:PA}), we compute
\begin{eqnarray}\label{eq:PB}
\dd\left(\int_{\T}\frac{1}{2}\varrho |\vec U|^2\,\dd x\right)&=&\int_{\T}\varrho\vec u \cdot \nabla \vec U\cdot \vec U\, \dd x \dd t+ \int_{\T}\varrho\vec U\cdot D_t^d\vec U\, \dd x \dd t\\
&&+\frac{1}{2}\sum_{k\geq 1}\varrho|\mathbb{D}_t^s\vec U(e_k)|^2\, \dd x \dd t +\dd M_2, \nonumber
\end{eqnarray}
where 
\[
M_2 = \int_{0}^{t}\int_{\T}\varrho\vec U\cdot \mathbb{D}_t^s\vec U\, \dd x\dd W.
\]

Testing the entropy balance (\ref{eq:entr}) with $\Theta$ we deduce
\begin{equation}\label{eq:PD}
\dd \left(\int_{\T}\langle \mathcal{V}_{t,x};  \varrho s \rangle\cdot\Theta \dd x\right) 
\geq\int_{\T} \varrho\langle\CV;\vec m s \rangle\cdot \nabla_x \Theta\, \dd x\dd t+\int_{\T}\langle \CV; \varrho s\rangle D_t^d \Theta\, \dd x\dd t,
\end{equation}
where $D_t^d\Theta =\partial_t\Theta$. Similarly, testing the continuity equation (\ref{eq:cont}) with $\partial_{\varrho}H_{\Theta}(r,\Theta)$ yields

\begin{equation}\label{eq:PF}
 \dd\left(\int_{\T}\varrho \partial_{\varrho}H_{\Theta}(r,\Theta)\, \dd x \right) =\int_{\T} \vec m\cdot \nabla_x(\partial_{\varrho}H_{\Theta}(r,\Theta)) \,\dd x+ \int_{\T}\varrho D_t^d(\partial_{\varrho}H_{\Theta}(r,\Theta))\, \dd x.
\end{equation}
where $D_t^d(\partial_{\varrho}H_{\Theta}(r,\Theta)) =\partial_t(\partial_{\varrho}H_{\Theta}(r,\Theta))$.\newline
\textbf{Step 2}:
Finally, we collect and sum the resulting expressions (\ref{eq:PA})-(\ref{eq:PF}), and add (\ref{eq:Ener}) to the sum to obtain (\ref{REI}), see \cite{EFAN} for more details.
\end{proof}

\begin{rmk} The \textit{ relative entropy inequality} is satisfied for any trio $[r,\Theta,\vec U]$ provided $p,e$ and $s$ satisfy the Gibbs' relation.
\end{rmk}

We are now ready to prove Theorem \ref{thm_a}, accordingly, we use Proposition \ref{propA} and a Gronwall type argument to prove pathwise weak-strong uniqueness claim as follows.

\begin{proof}[Proof of the claim]:
\newline
\textbf{Step 1}

We begin by introducing a stopping time

\[
\tau_{M}=\inf\bigg \{t\in (0,T)|\quad\|\vec U(s,\cdot)\|_{W^{1,2}(\T)}>M\bigg\}
\]
Since $[r,\Theta,\vec U]$ is a strong solution,
\[
\p\left[\lim_{M\to \infty}\tau_M=\mathfrak{t}\right]=1;
\]
therefore, it is enough to show results for a fixed M. Furthermore, $[r,\Theta, \vec U] \equiv[\varrho,\vartheta,\vec u]$ satisfies an equation of the form (\ref{eq:sass}), with
\[
D_t^d =-\vec U \cdot \nabla_x \vec U -\frac{1}{r}\nabla_x p(r,\Theta),\quad \mathbb{D}_t^s\vec U =\phi,\quad D_t^d r =-\mathrm{div}_x(r\vec U). 
\]

\begin{rmk}
Note that the It\"o correction term in (\ref{REI}) vanishes for our choice of $D_t^d\vec U.$ 
\end{rmk}
\textbf{Step 2}
\newline
For $M>0$, we have 
\begin{equation}\label{eq:Tstop}
\sup_{t\in [0,\tau_M]}\|\nabla\vec U\|_{L^{\infty}(\T)}\leq c(M).
\end{equation}

Since $r$ satisfies the continuity equation and hypothesis (\ref{Idata}), then from maximum and minimum principle we have
\[
0<\underline{r}_M\leq r(t\wedge \mathfrak{t})\leq \overline{r}_M
\]
for some deterministic constants $\underline{r}_M,\overline{r}_M$. Similarly, for $\Theta$ we have 
\[
0<\underline{\Theta}_M\leq \Theta(t\wedge \mathfrak{t})\leq \overline{\Theta}_M.
\]
The relative energy (\ref{relative}) can be re-written as

\begin{eqnarray}\label{Nvar}
\mathcal{E}\bigg(\varrho,E,\vec m\bigg|r,\Theta,\vec U\bigg)&=&\int_{\T}\left(\frac{1}{2}\varrho\left|\frac{\vec m}{\varrho}-\vec U\right|^2 +E -\Theta\varrho \vec{s}(\varrho,E)-\partial_{\varrho}H_{\Theta}(r,\Theta)(\varrho-r)-H_{\Theta}(r,\Theta)\right)\, \dd x\nonumber\\
&&+ \frac{1}{2}\int_{\T}\dd \text{tr} \mathcal{R}_{\text{conv}}(t)+c_v\int_{\T} \dd \mathcal{R}_{\text{press}}(t).
\end{eqnarray}
where $E =\varrho e(\varrho,\vartheta)$. 

Moreover, we consider a function $\Phi(\varrho,E)$,
\[
\Phi \in C_c^{\infty}(0,\infty)^2, 0\leq \Phi\leq 1.
\]
For a measurable function $G(\varrho,E, \vec m)$, we set 
\begin{equation}\label{ess}
 G = G_{\text{ess}}  + G_{\text{res}}, \quad G_{\text{ess}}=\Phi(\varrho,E)G(\varrho,E,\vec m),\quad G_{\text{res}}=(1-\Phi(\varrho,E))G(\varrho,E,\vec m).
\end{equation}
Following the presentation in \cite{FEJB, EA}, $G_{\text{ess}}$ accounts for the `essential part' that describes the behaviour of the non-linearity in the non-degenerate area where both $\varrho$ and $\vartheta$ are bounded below and above. On the other hand, $G_{\text{res}}$ accounts for the `residual part' that captures the behaviour in the singular regime $\varrho,\vartheta \to 0$ or/and $\varrho,\vartheta \to \infty$.\newline

In view of (\ref{ess}), we recall the coercivity properties of $\mathcal{E}$ proved in (\cite{EA}, Chapter 3, Proposition 3.2),

\begin{eqnarray}\label{coercivity}
\mathcal{E}\bigg(\varrho,E,\vec m|r,\Theta,\vec U\bigg) &\gtrsim& \int_{\T}\bigg[|\varrho-r|^2+|E-re(r,\Theta)|^2+\left|\frac{\vec m}{\varrho}-\vec U\right|^2\bigg]_{\mathrm{ess}}\, \dd x\\
&& +\int_{\T}\bigg[1+\varrho +\varrho|s(\varrho,E)|+E +\frac{|\vec m|}{\varrho}\bigg]_{\mathrm{res}}\,\dd x\nonumber.
\end{eqnarray}

\textbf{Step 3:}\newline
In view of (\ref{Nvar}) and  Proposition \ref{propA}, we apply the  relative entropy inequality (\ref{REI}) on the time interval $[0,\tau_M]$
\begin{equation}\label{SREI}
   \mathcal{E}\bigg(\varrho,E,\vec m\bigg|r,\Theta,\vec U\bigg)(t\wedge\tau_M)\leq \mathcal{E}\bigg(\varrho,E,\vec m\bigg|r,\Theta,\vec U\bigg)(0)+\int_{0}^{\tau\wedge \tau_M}\mathcal{Q}\bigg(\varrho,E,\vec m\bigg|r,\Theta,\vec U\bigg)\, \dd t+\mathbb{M}(t\wedge \tau_M),
\end{equation}
with
\begin{align}\label{qfn}
\mathcal{Q}\bigg(\varrho,E,\vec m\bigg|r,\Theta,\vec U\bigg)=
&\int_{\T}\varrho\left(\frac{\vec m}{\varrho}-\vec U\right)\cdot\nabla_x \vec U\cdot\left(\vec U-\frac{\vec m}{\varrho}\right) \,\dd x\nonumber\\
&+\int_{\T}[(\varrho\vec U- \vec m)(D_t^d \vec U + \vec U\cdot\nabla_x \vec U ) -p(\varrho,E)\mathrm{div}_x\vec U] \, \dd x \nonumber\\
&-\int_{0}^{\tau}\int_{\T}[\langle \CV; \varrho s(\varrho,E)\rangle D_t^d \Theta + \vec \langle \CV; s(\varrho,E)\vec m\rangle\cdot \nabla_x \Theta ] \, \dd x \dd t\nonumber\\
&+\int_{0}^{T}\int_{\T}[\varrho s (r,\Theta) \partial_{t}\Theta+ \vec m s(r,\Theta)\cdot\nabla_x\Theta]\, \dd x\dd t \nonumber\\
&+\int_{\T} \left(\left(1-\frac{\varrho}{r}\right)\partial_{t}p(r,\Theta)-\frac{\vec m}{r}\cdot\nabla_x p(r,\Theta)\right)\,\dd x\nonumber\\
&- \int_{\T}\nabla \vec U:\dd \mathcal{R}_{\text{conv}}-\int_{\T}\mathrm{div} \vec U \dd \mathcal{R}_{\text{press}}.\nonumber\\
\end{align}

To apply Gronwall's argument, we need to show the estimate 

\begin{equation}\label{Qest}
\mathcal{Q}\bigg(\varrho,E,\vec m\bigg|r,\Theta,\vec U\bigg)\lesssim c\, \mathcal{E}\bigg(\varrho,E,\vec m\bigg|r,\Theta,\vec U\bigg),
\end{equation}
for some constant $c > 0$. In view of (\ref{eq:Tstop}), the estimate on defect measures is given by

\begin{equation}
 \int_{0}^{\tau\wedge \tau_m}\int_{\T}\nabla_x\vec U:\dd[\mathcal{R}_{\text{conv}} + \mathcal{R}_{\text{press}}\mathbb{I}]\,\dd t  
\lesssim c(M) \frac{1}{2} \int_{0}^{\tau\wedge \tau_m}\int_{\T}\dd \text{trace}[\mathcal{R}_{\text{conv}} + \mathcal{R}_{\text{press}}\mathbb{I}]\,\dd t.  
\end{equation}

Similarly, using (\ref{eq:Tstop}) we obtain 
\begin{eqnarray*}
\left|\int_{\T}\varrho\left(\frac{\vec m}{\varrho}-\vec U\right)\cdot\nabla_x \vec U\cdot\left(\vec U-\frac{\vec m}{\varrho}\right) \,\dd x\right|&\leq& \int_{\T}\varrho\left|\frac{\vec m}{\varrho}-\vec U\right|^2|\nabla_x \vec U| \, \dd x,\\
&\lesssim& c(M)\int_{\T}\varrho\left|\frac{\vec m}{\varrho}-\vec U\right|^2 \, \dd x.
\end{eqnarray*}

Furthermore, we observe that
\[
(\varrho\vec U- \vec m)(D_t^d \vec U + \vec U\cdot\nabla_x \vec U )-\frac{ \vec m}{r}\cdot\nabla_x p(r,\Theta)=-\frac{\varrho\vec U}{r}\cdot\nabla_x p(r,\Theta).
\]
Consequently, (\ref{qfn}) reduces to
\begin{align}\label{qfnn}
\mathcal{Q}\bigg(\varrho,E,\vec m\bigg|r,\Theta,\vec U\bigg)\lesssim
&-\int_{0}^{\tau}\int_{\T}[\langle \CV; \varrho s(\varrho,\vartheta)\rangle D_t^d \Theta + \vec \langle \CV; s(\varrho,\vartheta)\vec m\rangle\cdot \nabla_x \Theta ] \, \dd x \dd t\nonumber\\
&+\int_{0}^{T}\int_{\T}[\varrho s (r,\Theta) \partial_{t}\Theta+ \vec m s(r,\Theta)\cdot\nabla_x\Theta]\, \dd x\dd t \nonumber\\
&+\int_{\T}[p(r,\Theta)\mathrm{div}_x\vec U -p(\varrho,E)\mathrm{div}_x\vec U] \, \dd x \nonumber\\
&+\int_{\T} \left((r-\varrho)\frac{1}{r}\partial_{t}p(r,\Theta)-\frac{\varrho\vec U}{r}\cdot\nabla_x p(r,\Theta)-p(r,\Theta)\mathrm{div}_x\vec U\right)\,\dd x\nonumber\\
&+c_1\, \mathcal{E}\bigg(\varrho,E,\vec m\bigg|r,\Theta,\vec U\bigg).\nonumber\\
\end{align}

Finally, adopting the notation in (\ref{ess}) and manipulating terms in (\ref{qfnn}), as in \cite{FEJB}, we obtain

\begin{align}\label{eq:EA}
\mathcal{Q}\bigg(\varrho,E,\vec m\bigg|r,\Theta,\vec U\bigg)\lesssim
c\, \mathcal{E}\bigg(\varrho,E,\vec m\bigg|r,\Theta,\vec U\bigg).\nonumber\\
\end{align}

\textbf{Step 4}\newline
In view of the above estimates, the relative entropy inequality (\ref{SREI}) reduces to
\begin{equation}\label{SREE}
   \mathcal{E}\bigg(\varrho,E,\vec m\bigg|r,\Theta,\vec U\bigg)(t\wedge\tau_M)\lesssim \mathcal{E}\bigg(\varrho,E,\vec m\bigg|r,\Theta,\vec U\bigg)(0)+\int_{0}^{\tau\wedge \tau_M}c\,\mathcal{E}\bigg(\varrho,E,\vec m\bigg|r,\Theta,\vec U\bigg)(t)\, \dd t+\mathbb{M}(t\wedge\tau_M).
\end{equation}
Now, taking the expectation in $[t\wedge \tau_M]$ and applying Gronwall's lemma yields
\begin{equation}\label{eq:ExpA}
   \E\left[\mathcal{E}\bigg(\varrho,E,\vec m\bigg|r,\Theta,\vec U\bigg)(t\wedge\tau_M)\right]\leq c(M)\E\left[\mathcal{E}\bigg(\varrho,E,\vec m\bigg|r,\Theta,\vec U\bigg)(0)\right], 
\end{equation}
where 
\[
\E\left[\mathcal{E}\bigg(\varrho,E,\vec m\bigg|r,\Theta,\vec U\bigg)(0)\right]=0
\]
by assumptions. Therefore, we observe that 
\[
\E\left[\mathcal{E}\bigg(\varrho,E,\vec m\bigg|r,\Theta,\vec U\bigg)(t\wedge\tau_M)\right]=0
\]
for all $t \in (0,T)$,  yielding the claim.
\end{proof}

\section{Martingale solutions as measures on the space of trajectories}\label{markov}
Firstly, we observe that from the proof of Theorem [existence], the natural filtration associated to a dissipative measure-valued martingale solution in the sense of Definition \ref{E:dfn} is the joint canonical filtration of $[\varrho, \vec m, S, \mathcal{R}_{\text{conv}},\mathcal{R}_{\text{press}},\mathcal{V}_{t,x},W]$. However, the canonical processes $[S, \mathcal{R}_{\text{conv}},\mathcal{R}_{\text{press}},\mathcal{V}_{t,x}]$ are class of equivalences in time and not a stochastic processes in the classical sense. Therefore, it is not obvious as to how one should formulate the Markovianity of the system (\ref{eq:aEuler})-(\ref{eq:cEuler}). To circumvent this problem, we shall introduce new variables $\Ss,\vec R$ (time integrals) such that 
\[
\Ss = \int_{0}^{\cdot}S\,\dd s, \quad \vec R = \int_{0}^{\cdot}\left(\mathcal{R}_{\text{conv}},\mathcal{R}_{\text{press}},\mathcal{V}_{t,x}\right)\,\dd s.
\]
Consequently, the notion of new variables allows us to establish the Markov selection for the joint law of $[\varrho,\vec m, \Ss, \vec R]$. In this case, the stochastic process has continuous trajectories and contains all necessary information. The initial data for $[\Ss,\vec R]$ is superfluous and only needed for technical reasons in the selection process.
To study Markov selection, it is desirable to consider the martingale solutions as probability measures $\PP \in $ Prob $[\Omega]$ such that
\[
\Omega= C_{\text{loc}}([0,\infty); W^{-k,2}(\T)),
\]
where $k>3/2$. Adopting the set-up of Section \ref{path} we set $X= W^{-k,2}(\T))$. Accordingly, let $\B$ denote the Borel $\sigma$-field on $\Omega$. Let $\boldsymbol{\xi}=(\xi^1,\xi^2,\xi^3,\xi^4)$ denote the canonical process of projections such that
\[
\boldsymbol{\xi}:\Omega=(\xi^1,\xi^2,\xi^3,\xi^4):\Omega\to \Omega, \quad \boldsymbol{\xi}_t\omega=(\xi_t^1,\xi_t^2,\xi_t^3,\xi_t^4)(\omega)=\omega_t \in W^{-k,2}(\T)),\text{for any $t\geq 0$}
\]
where the notation $\omega_t$ indicates that our random variable is time dependent. In addition, let $(\B_t)_{t\geq 0}$ be the filtration associated to canonical process given by
\[
\B_t:= \sigma(\boldsymbol{\xi}|_[0,t]),\quad t\geq 0,
\]
which coincides with the Borel $\sigma$-field on $\Omega^{[0,t]}=([0,t];W^{-k,2}(T))$. For analysis, the dissipative martingale solutions
\[
((\Omega,\FF, (\FF_t)_{t\geq 0},\p),\varrho,\vec m,S, \mathcal{R}_{\text{conv}},\mathcal{R}_{\text{press}},\mathcal{V}_{t,x}, W),
\]
in the sense of Definition \ref{E:dfn} is considered a probability law $U$, that is,
\[
\PP =\LL\left[\varrho, \vec m = \varrho\vec u, \int_{0}^{\cdot}S\,\dd s,\int_{0}^{\cdot} (\mathcal{R}_{\text{conv}},\mathcal{R}_{\text{press}},\mathcal{V}_{t,x})\, \dd s\right]\in\text{ Prob$[\Omega]$}.
\]
Consequently, we obtain the probability space $(\Omega,\B, (\B_t)_{t\geq 0}, U)$. Furthermore, let $\mathcal{P}$ be probability measure and introduce the space
 \begin{eqnarray*}
 F &=& \left\{[\varrho,\vec m, \mathcal{S},\vec{R}] \in\Tilde{F}\bigg|\int_{\T}\frac{|\vec m|^2}{|\varrho|}\dd x<\infty\right\},\\
 \Tilde{F}&=&L^{\gamma}(\T)\times L^{\frac{2\gamma}{\gamma+1}}(\T)\times (L^{\gamma}(\T))\cap BV_{w, \text{loc}}^{2}(W^{-l,2}(\T))\times (W^{-k,2}(\T, \cdot))^2
 \times (W^{-k,2}(\T ,A).
 \end{eqnarray*}
 where $A=\R\times\R^3\times\R$. We augment $F$ with the points of the form $(0,\vec0,\Ss,\vec R)$ for $\Ss \in W^{-l,2}(\T)\cap L^{\gamma}(\T)$ and $\vec R \in W^{-k,2}(\T,\R^{15})$. Therefore, $F$ is Polish space with metric
 \begin{equation}\label{metric}
    \dd_{F}(y,z)=\dd_{Y}((y^1,\vec y^2, \vec y^3,\vec y^4),(z^1,\vec z^2,\vec z^3,\vec z^4))=\|y-z\|_{\tilde F}+\left\|\frac{\vec y^2}{\sqrt{|y^1|}}-\frac{\vec z^2}{\sqrt{|z^1|}}\right\|_{L_x^2}.
 \end{equation}
 \newline \\
Moreover, inclusion $F \hookrightarrow X$ is dense. Accordingly, we define a subset
 \[
 Y =\left\{[\varrho, \vec m, \mathcal{S},\vec{R}]\in X\bigg|\varrho\not\equiv 0, \varrho \geq 0, \int_{\T}\frac{|\vec m|^2}{\varrho}\, \dd x< \infty\right\}.
 \]
 We observe that $(Y,d_{F})$ is not complete because $\varrho\not\equiv0$, and the inclusion $Y\hookrightarrow X$ is not dense since $\varrho \geq 0$. The probability law $\PP(t,\cdot)$ continues to hold(supported) in $Y$, and consequently determines \textit{the set of admissible initial conditions}. 
 
 \begin{dfn}[Dissipative measure-valued martingale solution]\label{dfnMarkov}A borel probability measure $\PP$ on $\Omega$ is called a solution to the martingale problem associated to (\ref{eq:aEuler})-(\ref{eq:cEuler}) provided:
 \begin{itemize}
     \item [(a)] it holds
     \begin{align*}
          \PP&\left(\xi^1\in C_{\mathrm{loc}}[0,\infty);(L^{\gamma}(\T),w)),\xi^1\geq 0\right)=1,\\
           \PP&\left(\xi^2\in C_{\mathrm{loc}}[0,\infty);(L^{\frac{2\gamma}{\gamma+1}}(\T),w))\right)=1,\\
            \PP&\left(\xi^3\in W^{1,\infty}[0,\infty);(L^{\gamma}(\T)))\cap BV_{w,\mathrm{loc}}^{2}(0,\infty; W^{-l,2}(\T))\right)=1,\\
             \PP&\left(\xi^4\in (W_{\mathrm{weak}-(*)}^{1,\infty}(0,\infty;\mathcal{M}^+(\T, \cdot)))^2\right)=1;\\
             \PP&\left(\xi^4\in (W_{\mathrm{weak}-(*)}^{1,\infty}((0,\infty)\times\T; \p(\R^6)))\right)=1;\\
     \end{align*}
     \item[(b)] the total energy
     \[
    \mathfrak{E}= \int_{\T}\left[\frac{1}{2}\frac{|\boldsymbol{\xi}^2|^2}{\xi^1}+c_v(\xi^1)^{\gamma}\exp{\left(\frac{\xi^3}{c_v\xi^1}\right)}\right]\, \dd x + \frac{1}{2}\int_{\T}\dd \mathrm{tr}\boldsymbol{\xi}^4_{\text{conv}}(t) +c_v\int_{\T}\dd \mathrm{tr}\boldsymbol{\xi}^4_{\text{press}}(t)
    \]
    belongs to the space $L^\infty_{\rm{loc}}(0,\infty)$ $\PP\mbox{-a.s.}$;
    \item[(c)] it holds $\PP$-a.s.
\[
\left[ \int_{\T}\xi_{t}^1 \psi  \right]_{t = 0}^{t = \tau} -
\int_0^\tau \int_{\T}{ \boldsymbol{\xi}_{t}^{2} \cdot \nabla \psi  } \,\dd x\dt = 0
\]
for any $\psi \in C^1(\T)$ and  $\tau \geq 0$;
\item[(d)] for any $\varphi \in C^1(\T, \R^3)$,
the stochastic process

\begin{eqnarray*}
         \mathscr M(\varphi): [\omega, \tau] \mapsto\left[\int_{\T}\boldsymbol{\xi}_t^2 \cdot \boldsymbol{\varphi}\right]_{t=0}^{t=\tau}&-&\int_{0}^{\tau}\int_{\T}\left[\frac{\boldsymbol{\xi}_t^2 \otimes\boldsymbol{\xi}_t^2}{\xi_t^1}:\nabla\boldsymbol{\varphi}+\xi_t^1\exp{\left(\frac{\xi_t^3}{c_v\xi_t^1}\right)\mathrm{div}\boldsymbol{\varphi}}\right]\, \dd x \dd t\\
         &&-\int_{0}^{\tau}\nabla \varphi:\dd \boldsymbol{\xi_t}_{\mathrm{conv}}^4 \dd t-\int_{0}^{\tau}\int_{\T}\mathrm{div} \boldsymbol{\varphi}\,\dd \boldsymbol{\xi_t}_{\mathrm{press}}^4 \dd t\nonumber
    \end{eqnarray*}
is a square integrable $((\mathfrak{B}_{t})_{t\geq0},\PP)$-martingale with quadratic variation
\[
\frac{1}{2} \int_0^\tau \sum_{k=1}^\infty \left( \int_{\T}{\xi_{t}^1\varphi e_k \cdot \varphi } \right)^2\dd t;
\]
\item[(e)] It holds $\PP$-a.s.
    \begin{align}\label{eq:Entr}
    \begin{aligned}
        \int_{0}^{\tau}\int_{\T}&\left[\langle(\boldsymbol{\xi}^4_\nu)_{t,x};Z(\tilde S)\rangle\partial_t \varphi+\langle(\boldsymbol{\xi}^4_\nu)_{t,x},Z(\tilde S)\tilde{\vec m}/\tilde\varrho\rangle\cdot \varphi\right]\, \dd x\dd t \\&\qquad\qquad\leq \left[\int_{\T}\langle (\boldsymbol{\xi}^4_\nu)_{t,x};Z(\tilde S)\rangle \varphi\, \dd x\right]_{t=0}^{t=\tau}
        \end{aligned}
    \end{align}
    for any $\varphi \in C^1([0,\infty)\times \T), \varphi \geq 0,$ 
    \item[(f)] The stochastic process
    \begin{equation}\label{eq:Enerr}
         \mathscr E: [\omega, \tau] \mapsto\mathfrak{E}_\tau-\mathfrak{E}_0-\frac{1}{2}\int_{0}^{\sigma}\|\sqrt{\xi^1}\phi\|_{L_2(\UU;L^2(\T))}^2\, \dd \sigma 
    \end{equation}
    is a square integrable $((\B_{t})_{t\geq0},U)$-martingale with quadratic variation
    \[
\frac{1}{2} \int_0^\tau \sum_{k=1}^\infty \left( \int_{\T}{\boldsymbol{\xi}_{t}^2\cdot\varphi e_k } \right)^2\dt
    \]
    for $\tau \geq 0$.
 \end{itemize}
 \end{dfn}
 
 In the following we state the relation between Definition \ref{E:dfn} and Definition \ref{dfnMarkov}.
 
 \begin{prop}\label{equdfns}
 The following statement holds true
 \begin{enumerate}
     \item Let $
((\Omega,\FF, (\FF_t)_{t\geq 0},\p),\varrho,\vec m,S, \mathcal{R}_{\text{conv}},\mathcal{R}_{\text{press}},\mathcal{V}_{t,x}, W)$ be a dissipative martingale solution to (\ref{eq:aEuler})-(\ref{eq:cEuler}) in the sense of Definition \ref{E:dfn}. Then for every $\FF_0$-measurable random variables $[\mathcal{S}_0, \mathbf{R}_0]$ with values in $\Ss \in W^{-l,2}(\T)\cap L^{\gamma}(\T)$ and $\vec R \in W^{-k,2}(\T,\R^{15})$ we have that 
\begin{equation}\label{marklawA}
\PP= \LL\left[\varrho, \vec m=\varrho\vec u, \mathcal{S}_0+\int_{0}^{\cdot}\mathcal{S}\,\dd s,\vec R_0+\int_{0}^{\cdot}\vec{R}\,\dd s\right]\in \mathrm{Prob}[\Omega]
\end{equation}
is a solution to the martingale problem associated to (\ref{eq:aEuler})-(\ref{eq:cEuler}) in the sense of Definition \ref{dfnMarkov}.
\item Let $\PP$ be a solution to the martingale problem associated to (\ref{eq:aEuler})-(\ref{eq:cEuler}) in the sense of Definition \ref{dfnMarkov}. Then there exists a dissipative martingale solution $
((\Omega,\FF, (\FF_t)_{t\geq 0},\p),\varrho,\vec m,S, \mathcal{R}_{\text{conv}},\mathcal{R}_{\text{press}},\mathcal{V}_{t,x}, W_1)$  to the system (\ref{eq:aEuler})-(\ref{eq:cEuler}) in the sense of Definition \ref{E:dfn} satisfying properties (a)-(j), furthermore, there exists  $((\Omega,\FF, (\FF_t)_{t\geq 0},\p),\varrho,\vec m,S, \mathcal{R}_{\text{conv}},\mathcal{R}_{\text{press}},\mathcal{V}_{t,x},W_2)$ in the sense of Definition \ref{E:dfn}  satisfying property (k) and an $\FF_0$-measurable random variables $[\mathcal{S}_0, \mathbf{R}_0]$ for $\Ss \in W^{-l,2}(\T)\cap L^{\gamma}(\T)$ and $\vec R \in W^{-k,2}(\T,\R^{15})$ such that 
\begin{equation}\label{marklawB}
\PP= \LL\left[\varrho, \vec m=\varrho\vec u, \mathcal{S}_0+\int_{0}^{\cdot}\mathcal{S}\,\dd s,\vec R_0+\int_{0}^{\cdot}\vec{R}\,\dd s\right]\in \mathrm{Prob}[\Omega],
\end{equation}
 where $W_1$ and $W_2$ correspond to Wiener process generated by momentum equation and energy equality, respectively.
 \end{enumerate}
 \end{prop}
 
 The proof of Proposition \ref{equdfns} follows the same arguments as presented in \cite{DbEfMh} with appropriate adjustment to our system (\ref{eq:aEuler})-(\ref{eq:cEuler}). Moreover, on showing part (2) implies part (1) in Proposition \ref{equdfns}  we need two Wiener processes since we apply the martingale representation theorem (see \cite{Prato}, Theorem 8.2) twice.
 
 \section{Markov selection}
 In this section we state and prove the strong Markov selection to the complete stochastic Euler system (\ref{eq:aEuler})-(\ref{eq:cEuler}). Let $y\in Y$ be an admissible initial data(condition), we denote by $\PP_{y}$ a solution to the martingale problem associated with (\ref{eq:aEuler})-(\ref{eq:cEuler}) starting on $y$ at time $t=0$; that is, the marginal of $\PP_y$ at $t=0$ is $\Lambda_y$. We start off with the definition of strong Markov family;
 
 \begin{dfn}
 A family $(\PP_y)_{y \in Y}\in \mathrm{Prob}[\Omega]$ of probability measures is called a strong Markov selection family provided
 \begin{itemize}
     \item [(1)] For every $A \in \B$, the mapping $y\mapsto \PP_y(A)$ is $\B(Y)/\B([0,1])$-measurable.
     \item[(2)]For every finite $(\B_t)_{t\geq0}$-stopping time $T$, every $y \in Y$ and $\PP_y$-a.s. $\omega \in \Omega$
     \[
     \PP_y|_{\B_T}^{\omega} = \PP_{y}\circ\Phi_{-\tau}^{-1}.
     \]
 \end{itemize}
 \end{dfn}
 
 Accordingly, a strong Markov family follows from the so-called pre-Markov family via a selection procedure. Finally, we have all we need to state the following theorem.
 
 \begin{thm}\label{Mselection}
 Assume (\ref{eq:HS}) and (\ref{fyB}) holds. Then there exists a family $\{\PP_y\}_{y\in Y}$ of solutions to the martingale problem associated to (\ref{eq:aEuler})-(\ref{eq:cEuler}) in the sense of Definition \ref{dfnMarkov} with a.s. Markov property (as defined in Definition \ref{almostMark})
 \end{thm}
 
 We set $y =(y^1,\vec y^2,\vec y^3,\vec y^4) \in Y$ and denote by $\CC(y)$ the set of probability laws $\PP_{y}\in \mathrm{Prob}[\Omega]$ solving the martingale problem associated to (\ref{eq:aEuler})-(\ref{eq:cEuler}) with initial law [$\Lambda_{y}$]. The proof of Theorem \ref{Mselection} follows from applying abstract result of Theorem \ref{Mthm}. In particular, we show that the family $\{\CC(y)\}_{y\in Y}$ of solutions to the martingale problem satisfies the disintegration and reconstruction properties in Definition \ref{almostMark}.
 
 \begin{lem}
 For each $y =(y^1,\vec y^2,\vec y^3,\vec y^4) \in Y$. The set $\CC(y)$ is \textit{non-empty} and \textit{convex}. Furthermore, for every $\PP\in \CC(y)$, the marginal at every time $t\in (0,\infty)$ is supported on $Y$.
 \end{lem}
 
 \begin{proof}
 Assuming $y\in Y$, application of Theorem [existence] yields existence of a martingale solution to the problem (\ref{eq:aEuler})-(\ref{eq:cEuler}) in the sense of Definition \ref{E:dfn}. Consequently, by Proposition \ref{equdfns} we infer that for each $y\in Y$ the set $\CC(y)$ is non-empty. For some $\lambda \in (0,1)$, let $\PP_1,\PP_2 \in \CC(y)$ such that $\PP=\lambda \PP_1+(1-\lambda)\PP_2$. Then convexity follows from noting that properties of Definition \ref{E:dfn} involve integration with respect to the elements of $\CC(y)$. In view of Definition \ref{E:dfn} property (f) (energy equality), the marginal $\PP\in \CC(y)$ at every $t\in (0,\infty)$ is supported in $Y$.
 \end{proof}
 
 For compactness we consider the following Lemma.
 \begin{lem}
 Let $y\in Y$. Then $\CC(y)$ is a compact set and the map $\CC:Y\to \mathrm{Comp}(\mathrm{Prob}[\Omega])$ is Borel measurable.
 
 \end{lem}
 \begin{proof}
 The lemma follows from the claim:
 Let $(y_n = (\varrho_n,\vec m_n,\Ss_n,\vec R_n))_{n\in \N} \subset Y$ be a sequence converging in $Y$ to some $(y = (\varrho,\vec m,\Ss,\vec R))$ with respect to the metric $d_F$ in (ref). Let $\PP_n \in C(y_n), n\in \N$. Then for each $(\PP_n)_{n\in N}$, the sequence converges to some $\PP \in \CC(y)$ weakly in $\mathrm{Prob}[\Omega]$. While measurability of the map $y\mapsto \CC(y)$ follows from using Theorem 12.1.8 in \cite{STVAR} for the metric space $(Y,d_F)$.  Accordingly, the claim follows from Theorem \ref{ExMainr}. Consequently, by Proposition \ref{quad} $\PP$ is a solution to a martingale problem with initial law $\Lambda$. Therefore, $\PP\in \CC(y)$ as required.
 \end{proof}
 
 Finally, we verify that $\CC(y)$ has disintegration and reconstruction property in sense of Definition  \ref{almostMark}.
 
 \begin{lem}\label{dis}
 The family $\{\CC(y)\}_{y\in Y}$ satisfies the disintegration property
 of Definition \ref{almostMark}.
 \end{lem}
 
 \begin{proof}
 Fix $y\in Y$, $\PP\in \CC(y)$ and let $T$ be $\B_t$-stopping time. In view of Theorem \ref{Dis}, we know there exists a family of probability measures;
 \[
 \Omega \ni \tilde{\omega}\mapsto \PP|_{\B_T}^{\tilde{\omega}}\in \mathrm{Prob}[\Omega^{[T,\infty)}]
 \]
 such that 
 \begin{equation}\label{ameas}
     \omega(T)=\tilde\omega(T), \, \PP|_{\B_T}^{\tilde{\omega}}\text{-a.s.},\qquad \PP(\omega|_{[0,T]}\in A,\omega|_{[T,\infty)}\in B)= \int_{\tilde{\omega}|_{[0,T]}\in A}\PP|_{\B_T}^{\tilde{\omega}}(B)\,\dd \PP(\tilde{\omega}),
 \end{equation}
  for any Borel sets: $A\subset\Omega^{[0,T]}$ and  $B\subset\Omega^{[T,\infty)}$. Here, we want to show that 
  \[
   \Phi_{-\tau}\PP|_{\B_T}^{\tilde{\omega}} \in \CC(\omega(T))\quad\text{for $\tilde\omega \in \Omega, \PP$-a.s.}
  \]
  Thus we are seeking an $\PP|_{\B_T}^{\tilde{\omega}}$-nullset $N$ outside of which properties (a)-(f) of Definition \ref{dfnMarkov} holds for $\PP|_{\B_T}^{\tilde{\omega}}$. To begin with, set $N_a,\ldots,N_f$ to each of the properties (a)-(f) of Definition \ref{dfnMarkov}, respectively, and let $N= N_a\cup\cdots\cup N_f$. Arguing similarly along the lines of Lemma 4.4 in \cite{FlROm} and \cite{DbEfMh} we have the following observations:
  \begin{itemize}
      \item[(1)] Set
      \begin{eqnarray*}
           H_T&=&\bigg\{\omega \in \Omega:\omega|_{[0,T]} \in C([0,T];L^{\gamma}(\T))\times C([0,T];L^{\frac{2\gamma}{\gamma+1}}(\T))\\
           && \times W^{1,2}([0,\infty),L^{\gamma}(\T))\cap BV_{w, \text{loc}}^{2}(0,\infty;W^{-l,2}(\T))
           \times(W_{\text{weak-*}}^{1,\infty}(0,\infty;\mathcal{M}^{+}(\T)))^2\bigg\}\\
            H^T&=&\bigg\{\omega \in \Omega:\omega|_{[T,\infty)} \in C([0,\infty);L^{\gamma}(\T))\times C([0,\infty);L^{\frac{2\gamma}{\gamma+1}}(\T))\\
            &&\times W^{1,2}([0,\infty),L^{\gamma}(\T))\cap BV_{w, \text{loc}}^{2}(0,\infty;W^{-l,2}(\T))
           \times(W_{\text{weak-*}}^{1,\infty}(0,\infty;\mathcal{M}^{+}(\T)))^2\bigg\},
      \end{eqnarray*}
      in view of property (a)  in Definition \ref{dfnMarkov}, for $\PP$ we obtain
      \[1= \PP(H_T\cap H^T) = \int_{H_T}\PP|_{\B_T}^{\tilde\omega}(H^T)\,\dd \PP(\tilde{\omega}),\]
      therefore, there is an $\PP|_{\B_T}^{\tilde{\omega}}$-nullset $N_a$  such that $\PP|_{\B_T}^{\tilde\omega}(H^T) =1$ holds for $\PP$-a.a.$\tilde\omega\in \Omega$ (i.e. the remaining $\tilde{\omega}\in \Omega$ are contained in nullset $N_a$).
      \item[(2)] Similarly, for the total energy property (b) in Definition \ref{dfnMarkov} we set
      \begin{gather*}
          \mathfrak{H}_T=\{\omega \in \Omega:\mathfrak{E}|_{[0,T]\in L_{\text{loc}}}(0,T)\},\\
          \mathfrak{H}^T=\{\omega \in \Omega:\mathfrak{E}|_{[T,\infty)\in L_{\text{loc}}}(T,\infty)\}.
      \end{gather*}
      Since the property (b) holds for $\PP$ a.s., arguing as in proof for property (a) (i.e. substituting $H_T$ and $H^T$ with $\mathfrak{H}_T$ and $\mathfrak{H}^T$, respectively) we deduce that there holds $\PP|_{\B_T}^{\tilde\omega}(H^T) =1$ for $\PP$-a.s. $\omega$. Consequently, this yields the nullset $(N_b)$.
      \item[(3)]For property (c), let $(\psi_n)_{n\in \N}$ be a dense subset of $W^{k,2}(\T)$ and fix $n\in \N$. For each $n\in \N$ we assign an $\PP$-nullset $N_c^{n}$ and set $N_c =\bigcup_{n\in \N}N_{c}^n$. To proceed, we split the continuity equation as follows:
      \begin{gather}
    \left[ \int_{\T}\xi_{t}^1 \psi_n  \right]_{t = 0}^{t = \tau} -
    \int_0^\tau \int_{\T}{ \boldsymbol{\xi}_{t}^{2} \cdot \nabla \psi_n  } \dd x\dt = 0\quad \forall 0\leq \tau\leq T,\label{ceqA}\\
    \left[ \int_{\T}\xi_{t}^1 \psi_n  \right]_{t = T}^{t = \tau} -
    \int_0^\tau \int_{\T}{ \boldsymbol{\xi}_{t}^{2} \cdot \nabla \psi_n  } \dd x\dt = 0\quad \forall T\leq \tau<\infty,\label{ceqB}
\end{gather}
and consider the sets
\begin{gather*}
    \mathfrak{A}_T=\{\omega \in \Omega:\omega|_{[0,T]}\,\text{satisfies (\ref{ceqA})} \}\\
    \mathfrak{A}^T=\{\omega \in \Omega:\omega|_{[T,\infty)}\,\text{satisfies (\ref{ceqB})} \}.
\end{gather*}
    As the property (c) holds for $\PP$, arguing similarly as in proof of (a) and (b) yields a nullset $N_c^n$.
    \item[(4)]In case of momentum equation (d), let $(\varphi_n)_{n\in \N}$ be a dense subset of $W^{k,2}(\T,\R^3)$ and fix $n\in \N$. Similarly, we assign for each $n\in \N$ an $\PP$-nullset $N_d^n$ and set $N_d =\bigcup_{n\in \N}N_{d}^n$. Noting that property (d) holds for $\PP$, then $(\mathscr M_t(\varphi_n))_{t\geq 0}$ is a $((\B_t)_{t\geq 0},\PP)$-square  integrable martingale with quadratic variation
    \[
    (\mathscr Q(\varphi_n))_{\tau} =\frac{1}{2} \int_0^\tau \sum_{k=1}^\infty \left( \int_{\T}{\xi_{t}^1\varphi e_k \cdot \varphi } \right)^2\dd t;
    \]
    As a consequence of Proposition of \ref{quad}, for $\PP$-a.a. $\tilde\omega$ we deduce that $(\mathscr M_t(\varphi_n))_{t\geq T}$ is a $((\B_t)_{t\geq T},\PP|_{\B_T}^{\tilde\omega})$-square integrable martingale with quadratic variation
    $(\mathscr Q(\varphi_n))_{t\geq T}$.
    \item[(5)] In the entropy inequality (e), the arguments coincide with those of proof of (c).
    \item[(6)] Similarly, for (f) we can argue as in the proof of (d) \textit{to obtain the nullset $N_f^n$.}\newline\\
    \textit{Choosing $N= N_a\cup\cdots\cup N_f$ completes the proof.}
  \end{itemize}
 \end{proof}
 
 \begin{lem}\label{ris}
 The family $\{\CC(y)\}_{y\in Y}$ satisfies the reconstruction property
 of Definition \ref{almostMark}.
 \end{lem}
 
 \begin{proof}
 Fix $y\in Y$, $\PP\in \CC(y)$ and let $T$ be $\B_t$-stopping {time}. In view of Theorem \ref{Rec}, suppose that $Q_{\omega}$ is a family of probability measures such that
 $Q_{\omega}$ is a family of probability measures, such that
\[
\Omega \ni \omega \mapsto Q_{\omega} \in\text{ Prob$[\Omega^{[T,\infty)}]$},
\]
is $\B_T$-measurable. Then there exists a unique probability measure $\PP\otimes_{T}Q$ such that :
\begin{itemize}
    \item [(a)]For any Borel set $A \in \Omega^{[0,T]} $ we have 
    \[
    (\PP\otimes_{T}Q)(A)=\PP(A);
    \]
    \item[(b)] For $\tilde{\omega}\in \Omega$ we have $\PP$-a.s.
    \[
    (\PP\otimes_{T}Q)|_{\B_T}^{\tilde\omega}=Q_{\tilde \omega}
    \]
\end{itemize}
We aim to prove that for a $Q_{\omega}:\Omega \to  \mathrm{ Prob}[\Omega^{[T,\infty)}]$-$\B_T$-measurable map such that there is $N\in\B_T$ with $\PP(N)=0$ and for all $\omega \not\in N $ it holds
\[
\omega(T)\in Y\quad \text{and}\quad \Phi_{-T}Q_{\omega}\in \CC(\omega(T));
\]
then $(\PP\otimes_{T}Q) \in \CC(y)$. In order to do this we have to verify properties (a)-(f) in Definition \ref{dfnMarkov}. The proof follows along the lines of \cite{FlROm}, Lemma 4.5. Adopting the notation introduced in Lemma \ref{dis}, we argue as follows:\\
Here we note that $Q_{\omega}$ is a regular conditional probability distribution of  $(\PP\otimes_{T}Q)$ with respect to $\B_T$.
\begin{itemize}
    \item [(1)] Since (a) holds for $Q_{\omega}$ we have $Q_{\omega}(H^T)=1$ such that
    \[
    \PP\otimes_{T}Q(H_T\cap H^T) =\int_{H_T}Q_{\omega}[S^T]\,\dd \PP(\omega)=1.
    \]
    \item [(2)] For properties (b), (c) and (e) of Definition \ref{dfnMarkov} we argue as in property (a) (Using the notation developed for each property, respectively).
    \item[(3)] In the case of property (d),we proceed as follows:\\
    Since (d) holds for $Q_{\omega}$ we know that $(\mathscr M_t(\varphi_n))_{t\geq T}$ is a $((\B_t)_{t\geq T},Q_{\omega})$-square integrable martingale for all $\varphi \in C^1(\T)$. Consequently, by Proposition \ref{quad} we deduce that $(\mathscr M_t(\varphi_n))_{t\geq T}$ is a $((\B_t)_{t\geq T},\PP\otimes_{T}Q)$-square integrable martingale as well. Observing that $\PP$ and $\PP\otimes_{T}Q$ coincides on $\B(\Omega^{[0,T]})$ and $(\mathscr M_t(\varphi_n))_{0\leq t\leq T}$ is a $((\B_t)_{0\leq t\leq T},\PP)$-martingale (since $\PP$ satisfies property (d)) we infer that $(\mathscr M_t(\varphi_n))_{t\geq 0}$ is a $((\B_t)_{t\geq 0},\PP\otimes_{T}Q)$-martingale.
    \item[(4)] property (f) follows by the same argument as in property (d) (with obvious modifications).
\end{itemize}
 \end{proof}

\subsection*{Acknowledgement}
The author is grateful to D. Breit for insightful discussions, suggestions and corrections. The author was supported by the Maxwell Institute Graduate School in Analysis and its Applications, a Centre for Doctoral Training funded by the UK Engineering and Physical Sciences Research Council (grant EP/L016508/01), the Scottish Funding Council, Heriot-Watt University and the University of Edinburgh.

\end{document}